\documentclass[11pt, a4paper]{amsart}

\usepackage{amsmath, amsthm, amssymb}

\newcommand{\Res}{\operatorname{Res}}

\newcommand{\MinResLoc}{\operatorname{MinResLoc}}

\newcommand{\sP}{\mathsf{P}}
\newcommand{\bP}{\mathbb{P}}

\newcommand{\cS}{\mathcal{S}}
\newcommand{\GCD}{\mathrm{GCD}}
\newcommand{\GL}{\mathrm{GL}}
\newcommand{\Fix}{\operatorname{Fix}}
\newcommand{\FP}{\operatorname{FP}}
\newcommand{\FR}{\operatorname{FR}}
\newcommand{\can}{\operatorname{can}}
\newcommand{\ordRes}{\operatorname{ordRes}}
\newcommand{\Crucial}{\operatorname{Crucial}}

\newcommand{\rd}{\mathrm{d}}
\newcommand{\bR}{\mathbb{R}}

\newcommand{\ord}{\operatorname{ord}}
\newcommand{\bN}{\mathbb{N}}
\newcommand{\diam}{\operatorname{diam}}
\newcommand{\PGL}{\mathrm{PGL}}

\newcommand{\cO}{\mathcal{O}}

\newcommand{\sH}{\mathsf{H}}

\newcommand{\supp}{\operatorname{supp}}
\newcommand{\Id}{\mathrm{Id}}
\newcommand{\BC}{\operatorname{BC}}

\theoremstyle{plain}
\newtheorem{theorem}{Theorem}[section]
\newtheorem{lemma}[theorem]{Lemma}

\newtheorem{mainth}{Theorem}
\newtheorem{maincoro}{Corollary}
\theoremstyle{definition}
\newtheorem{definition}[theorem]{Definition}
\newtheorem{notation}[theorem]{Notation}
\newtheorem*{acknowledgement}{Acknowledgement}
\theoremstyle{remark}
\newtheorem{remark}[theorem]{Remark}

\numberwithin{equation}{section}

\begin{document} 

\title
[Rumely's weight function and crucial measure]
{Geometric formulas on Rumely's weight function and crucial measure
in non-archimedean dynamics}

\author
[Y\^usuke Okuyama]
{Y\^usuke Okuyama}

\address{
Division of Mathematics,
Kyoto Institute of Technology,
Sakyo-ku, Kyoto 606-8585 Japan.}
\email{okuyama@kit.ac.jp}

\date{\today}

\begin{abstract}
We introduce the $f$-crucial function $\Crucial_f$ associated to 
a rational function $f\in K(z)$ of degree $>1$ 
over an algebraically closed field $K$ of possibly positive characteristic
that is complete with respect to a non-trivial and
non-archimedean absolute value, and give a global and explicit expression of
Rumely's (resultant) function $\ordRes_f$ in terms of the hyperbolic metric $\rho$ 
on the Berkovich upper half space $\sH^1$ in the Berkovich projective line $\sP^1=\sP^1(K)$.
We also obtain geometric formulas for Rumely's weight function $w_f$ 
and crucial measure $\nu_f$ on $\sP^1$ associated to $f$, as well as 
improvements of Rumely's principal results.
As an application to dynamics, we obtain a quantitative equidistribution 
of the sequence $(\nu_{f^n})_n$ of $f^n$-crucial measures towards the $f$-equilibrium 
(or canonical) 
measure $\mu_f$ on $\sP^1$.
\end{abstract}



\maketitle 

\section{Introduction}\label{sec:intro}

Let $K$ be an algebraically closed field of possibly positive characteristic
that is complete with respect to a non-trivial 
absolute value $|\cdot|$, which 
is
{\itshape non-archimedean} in that the strong triangle inequality 
$|z-w|\le\max\{|z|,|w|\}$ holds for any $z,w\in K$, and
let $k=\cO_K/\mathfrak{m}_K$ be the residue field of $K$,
where
\begin{gather*}
 \cO_K=\{z\in K:|z|\le 1\} 
\end{gather*}
is the ring of $K$-integers and $\mathfrak{m}_K$ is the maximal
ideal in $\cO_K$.
The {\itshape Berkovich} projective line $\sP^1=\sP^1(K)$ is a compact,
Hausdorff,
and uniquely arcwise connected topological space augmenting $\bP^1=\bP^1(K)$.  
As a set, $\sP^1$ is the set of all {\em multiplicative} seminorms {\em on}
$K(z)$ which restricts to $|\cdot|$ on $K$ and is allowed
to take $+\infty$ in an appropriate way (see \cite[\S3.4]{Jonsson15}), 
and a point $\cS$($=[\cdot]_{\cS}$ as a multiplicative seminorm on $K(z)$) 
in $\sP^1$
is of one and only one of types I, II, III, and IV. 
Any type I, II, or III point $\cS=[\cdot]_{\cS}$ in 
$\sP^1\setminus\{\infty\}$
is such a supremum seminorm $|\cdot|_{B_{\cS}}$
on $K[z]$ that $|\phi|_{B_{\cS}}=\sup_{z\in B_{\cS}}|\phi(z)|$, 
$\phi\in K[z]$, for a ($K$-closed) disk 
$B_{\cS}:=\{z\in K:|z-a|\le r\}$ 
in $K$, and setting $\diam\cS:=r$, $\cS$ is type I 
if and only if $\diam\cS=0$ (or equivalently $\cS\in K=\bP^1\setminus\{\infty\}$), and is of type II (resp.\ III) if and only if $\diam\cS\in|K^*|$ 
(resp.\ $\diam\cS\in\bR_{>0}\setminus|K^*|$). 
The point $\infty\in\bP^1$(=the evaluation norm 
$[\ell]_\infty=|\ell(\infty)|$, $\ell\in K(z)$, at $\infty$ on $K(z)$)
is type I, and 
the {\itshape canonical $($or Gauss$)$} point $\cS_{\can}$ in $\sP^1$
is type II and represented by $\cO_K$.

The tree structure on $\sP^1$ in the sense of Jonsson 
\cite[Definition 2.2]{Jonsson15} is induced by the
partial ordering of all those disks in $K$ by inclusions;
for any $\cS,\cS'\in\sP^1$, let $[\cS,\cS']$ (resp.\ $(\cS,\cS'],[\cS,\cS')$) be
the closed (resp.\ left half open, right half open) interval in $\sP^1$.
The Berkovich {\em upper half space} 
\begin{gather*}
 \sH^1:=\sP^1\setminus\bP^1
\end{gather*}
is equipped with the {\itshape hyperbolic $($or path distance$)$
metric} $\rho$, and this $\rho$ 
extends to the {\itshape generalized hyperbolic metric} $\tilde{\rho}$ on $\sP^1$
as a generalized metric on $\sP^1$
(see Subsections \ref{sec:tangent} and \ref{sec:laplacian}).
The set of all type II points is dense in the Berkovich {\em hyperbolic space}
$(\sH^1,\rho)$.

The action on $\bP^1$ of $h\in K(z)$ extends {\em continuously}
to that on $\sP^1$ so that $[\ell]_{h(\cS)}=[h^*\ell]_{\cS}$
for any $\cS\in\sP^1$ and any $\ell\in K(z)$. 
If $d_0:=\deg h>0$, then 
the action on $\sP^1$ of $h$ is surjective, open, and
finite, and preserves the type of each point in $\sP^1$. Furthermore,
the local degree function $a\mapsto\deg_a h$ on $\bP^1$ also canonically
extends to an upper semicontinuous
function $\sP^1\to\{1,2,\ldots,d_0\}$ such that for every domain $V\subset\sP^1$
and every component $U$ of $h^{-1}(V)$, the function 
$\cS\mapsto\sum_{\cS'\in h^{-1}(\cS)\cap U}\deg_{\cS'}(h)$ is constant on $V$
(see Subsection \ref{sec:directional} including the induced pullback action $h^*$
of $h$ on the space of Radon measures on $\sP^1$).
The linear fractional transformation group
$\PGL(2,K)$ on $\sP^1$ acts transitively on all type II points, and
$\PGL(2,\cO_K)$ is the stabilizer of $\cS_{\can}$.

We fixed a projective coordinate system on $\bP^1$.
Let $\pi:K^2\setminus\{(0,0)\}\to\bP^1$ be the
canonical projection such that $\pi(0,1)=\infty$ and 
$\pi(p_0,p_1)=p_1/p_0$ if $p_0\neq 0$.
An ordered pair of homogeneous polynomials
$H=(H_0,H_1)\in(K[p_0,p_1]_{\deg h})^2$ is
a (non-degenerate homogeneous polynomial) {\itshape lift} of $h\in K(z)$ if
$\pi\circ H=h\circ\pi$
on $K^2\setminus\{(0,0)\}$ 
(and $\Res H:=\Res(H_0,H_1)\neq 0$), 
where $\Res(H_0,H_1)\in K$ is the homogeneous resultant of $(H_0,H_1)$
(see, e.g., \cite[\S2.4]{SilvermanDynamics}). A lift of $f$ is unique
up to multiplication in $K^*=K\setminus\{0\}$.

Following Kawaguchi--Silverman \cite[Definition 2]{KS09},
we say a lift $H$ of $h\in K(z)$ is {\itshape minimal} 
if 
$\max\{|\text{coefficients of }H_0\text{ or }H_1|\}=1$, or equivalently, if
\begin{gather*}
\sup_{\|(p_0,p_1)\|\le 1}\|H(p_0,p_1)\|(=\sup_{z\in\cO_K}\|H(1,z)\|)=1,
\end{gather*}
where $\|\cdot\|$
is the maximum norm on $K^2$. A minimal lift $H$ of $h$
is unique up to multiplication in $\{z\in K:|z|=1\}$, and 
$-\log|\Res H|\in\bR$ is $\ge 0$ and
is independent of the choice of a minimal lift $H$ of $h$.
In his investigation on algorithmically deciding
whether a given $f\in K(z)$ of degree $d>1$ 
{\itshape has a potentially good reduction} in that 
$f^{-1}(\cS')=\{\cS'\}$ 
for some (in fact unique type II) 
point $\cS'\in\sH^1$ (see also Definition \ref{th:reduction} below),
Rumely \cite{Rumely13} associated to $f$ a non-negative, continuous, piecewise affine,
and convex function
$\ordRes_f$ on $(\sH^1,\rho)$ such that for every $h\in\PGL(2,K)$,
\begin{gather*}
 \ordRes_f\bigl(h(\cS_{\can})\bigr)
 =-\log|\Res(\text{a minimal lift of }h^{-1}\circ f\circ h)|,
\end{gather*}
setting also $\ordRes_f:=+\infty$ on $\bP^1$ by convention; 
$f^{-1}(\cS')=\{\cS'\}$ for some $\cS'\in\sH^1$ if and only if
$\min_{\sP^1}\ordRes_f=0$. However, the $\ordRes_f$
was not explicitly expressed on $\sH^1$ except for type II points.
In his subsequent investigation of the $f$-{\itshape minimal resultant locus}
\begin{gather*}
 \MinResLoc_f:=(\ordRes_f)^{-1}\Bigl(\min_{\sP^1}\ordRes_f\Bigr)
\end{gather*}
(i.e., the minimum locus of $\ordRes_f$)
using also Faber's works \cite{Faber13topologyI,Faber13topologyII},
Rumely \cite[Definitions 8 and 9]{Rumely14} 
also associated to $f$ a weight function 
\begin{gather*}
 w_f:\sP^1\to\bN\cup\{0\}
\end{gather*}
(see Subsection \ref{sec:defining} for more details)
and showed that the $f$-{\itshape crucial set} $\sP^1\setminus w_f^{-1}(0)$
consists of {\itshape at most} $d-1$ type II points
and contains {\itshape all} type II fixed points $\cS'$ of $f$ 
satisfying $\deg_{\cS'}(f)>1$, 
that the $f$-{\itshape crucial measure} $\nu_f$ on $\sP^1$ defined as
\begin{gather*}
\nu_f:=
\frac{\sum_{\cS\in\sP^1\setminus w_f^{-1}(0)}w_f(\cS)\cdot\delta_{\cS}}{d-1} 
\end{gather*}
is a {\em probability} Radon measure on $\sP^1$, and that
$\MinResLoc_f$ is a possibly trivial closed interval
in the $f$-{\itshape crucial tree}
$\Gamma_f:=\Gamma_{\supp(\nu_f)}$ spanned by the $f$-crucial set 
$\supp(\nu_f)=\sP^1\setminus w_f^{-1}(0)$ in $\sP^1$ 
(see also Corollary \ref{th:crucialtree} below).
Here $\delta_{\cS}$ is the Dirac measure on $\sP^1$ at a point $\cS\in\sP^1$,
and we adopt the convention that $\bN=\{1,2,\ldots\}$.

Our aim is to contribute to the study of both $\MinResLoc_f$ and 
$\Gamma_f$,
not only by giving a {\em global} and {\em explicit} geometric expression
of $\ordRes_f$ in terms of the geometry of $(\sH^1,\rho)$
but also by giving a geometric formulas for $w_f$ and $\nu_f$. 
We will refine some of principal results in
Rumely \cite{Rumely13,Rumely14} (and the published version \cite{Rumely17}
of \cite{Rumely14}), and also give
a dynamical application of our geometric studies, refining Jacobs
\cite[Theorem 5]{Jacobs17}.

\subsection{The $f$-crucial function $\Crucial_f$ and a global and explicit geometric expression for $\ordRes_f$}
Let $f\in K(z)$ be of degree $d>1$. 
Let us introduce the following geometric function on $(\sH^1,\rho)$.

\begin{definition}\label{th:crucialdef}
The $f$-{\itshape crucial function} is defined by
 \begin{multline*}
 \sH^1\ni\cS\mapsto\Crucial_f(\cS)\\
 :=\frac{\rho(\cS,\cS_{\can})}{2}
 +\frac{\rho(\cS,f(\cS)\wedge_{\can}\cS)
 -\int_{\sP^1}\rho(\cS_{\can},\cS\wedge_{\can}\cdot)\rd(f^*\delta_{\cS_{\can}})}{d-1}\in\bR
 \end{multline*}
(the integration against $f^*\delta_{\cS_{\can}}$ 
in the right hand side is actually a finite sum). 
Here, for any $\cS_0,\cS,\cS'\in\sP^1$,
we let $\cS\wedge_{\cS_0}\cS'$ 
be the unique 
point in $\sP^1$ in the intersection of all the (closed) intervals
$[\cS,\cS']$, $[\cS,\cS_0]$, and $[\cS',\cS_0]$ in $\sP^1$, and 
for any $\cS,\cS'\in\sP^1$, we set
$\cS\wedge_{\can}\cS':=\cS\wedge_{\cS_{\can}}\cS'$ for simplicity. 
\end{definition}

One of our principal results is the following.

\begin{mainth}\label{th:resultant}
Let $f\in K(z)$ be of degree $d>1$. Then for every $h\in\PGL(2,K)$, 
\begin{gather}
 \Crucial_f(h(\cS_{\can}))
=-\frac{1}{2d(d-1)}
\log\frac{|\Res(\text{a minimal lift of }h^{-1}\circ f\circ h)|}{|\Res(\text{a minimal lift of }f)|}.
\label{eq:resultant}
\end{gather}
Moreover, $\Crucial_f$ is continuous on $(\sH^1,\rho)$ and satisfies
\begin{multline}
(\cS,\cS_0)\mapsto\Crucial_f(\cS)-\Crucial_f(\cS_0)\\
=\frac{\rho(\cS,\cS_0)}{2}
+\frac{\rho(\cS,f(\cS)\wedge_{\cS_0}\cS)
-\int_{\sP^1}\rho(\cS_0,\cS\wedge_{\cS_0}\cdot)\rd(f^*\delta_{\cS_0})}{d-1}
\quad\text{on }\sH^1\times\sH^1.
\label{eq:conjugate}
\end{multline}
\end{mainth}

The equality \eqref{eq:resultant} together with the continuity of
$\Crucial_f$ on $(\sH^1,\rho)$ recovers the {\itshape continuous interpolation} 
assertion in Rumely \cite[Theorem 1.1]{Rumely13}, giving 
a {\em global} and {\em explicit} hyperbolic geometric expression of $\ordRes_f$. 

\begin{maincoro}
Let $f\in K(z)$ be of degree $d>1$. Then
\begin{gather}
 \ordRes_f
 =2d(d-1)\cdot\Crucial_f
 -\log|\Res(\text{a minimal lift of }f)|\label{eq:formulaorder}
 \end{gather}
on $\sH^1$, or more explicitly, for every $\cS\in\sH^1$,
\begin{multline}
\ordRes_f(\cS)-\ordRes_f(\cS_{\can})\\
=d(d-1)\cdot\rho(\cS,\cS_{\can})
 +2d\cdot\rho(\cS,f(\cS)\wedge_{\can}\cS)\\
 -2d\cdot\sum_{\cS'\in f^{-1}(\cS_{\can})}\deg_{\cS'}(f)\cdot\rho(\cS\wedge_{\can}\cS',\cS_{\can}).
\end{multline}
More generally, for every $\cS\in\sH^1$ and every $\cS_0\in\sH^1$,
\begin{multline}
\ordRes_f(\cS)-\ordRes_f(\cS_0)\\
=d(d-1)\cdot\rho(\cS,\cS_0)
 +2d\cdot\rho(\cS,f(\cS)\wedge_{\cS_0}\cS)\\
 -2d\cdot\sum_{\cS'\in f^{-1}(\cS_0)}\deg_{\cS'}(f)\cdot\rho(\cS\wedge_{\cS_0}\cS',\cS_0).
\end{multline}
\end{maincoro}

We will see that the global and explicit 
{\em difference} formula \eqref{eq:conjugate}
for $\Crucial_f$ is useful for investigating 
the minimal resultant locus $\MinResLoc_f$, 
the $f$-crucial set $\supp(\nu_f)$, and 
the $f$-crucial tree $\Gamma_f=\Gamma_{\supp(\nu_f)}$.

\subsection{The convexity of $\Crucial_f$}
For a discussion of the tangent space $T_{\cS}\sP^1$
of $\sP^1$ at a point $\cS\in\sP^1$,
see Subsection \ref{sec:tangent}.
For every point $\cS\in\sH^1$ and every {\em direction} 
$\overrightarrow{v}=\overrightarrow{\cS\cS'}\in T_{\cS}\sP^1$, 
$\rd_{\overrightarrow{v}}=(\rd_{\overrightarrow{v}})_{\cS}$ is the (distributional) directional derivation
operator at $\cS$ with respect to $\overrightarrow{v}$
on the space of continuous functions on $(\sH^1,\rho)$ (see \eqref{eq:derivation}
below). 

Recall that for each $\cS\in\sP^1$, $\#T_{\cS}\sP^1=1$ 
iff $\cS$ is type either I or IV, $\#T_{\cS}\sP^1=2$ iff $\cS$ is type III,
and $\#T_{\cS}\sP^1>2$ iff $\cS$ is type II.
The following terminology will be useful.

\begin{definition}
 A function $\phi$ on $\sH^1$ 
 is {\itshape piecewise affine on} $(\sH^1,\rho)$
 if for every $\cS'\in\sH^1$ and every $\overrightarrow{\cS'\cS''}\in T_{\cS'}\sP^1$,
 diminishing $[\cS',\cS'']$ if necessary,
 $\phi$ (or the restriction of $\phi$ to $[\cS',\cS'']$)
 is affine on $([\cS',\cS''],\rho)$. A piecewise affine function $\phi$
 on $(\sH^1,\rho)$ is {\itshape locally affine on} $(\sP^1,\tilde{\rho})$
 {\itshape at a point} $\cS'\in\sP^1$ if (i) in the case where $\cS'\in\sH^1$,
 for any {\itshape distinct} directions 
 $\overrightarrow{\cS'\cS_1},\overrightarrow{\cS'\cS_2}\in T_{\cS'}\sP^1$, 
 diminishing $[\cS',\cS_1]$ and $[\cS',\cS_2]$ if necessary,
 $\phi$ is affine on $([\cS_1,\cS_2],\rho)$ 
 and (ii) in the case where $\cS'=a\in\bP^1$, 
 for every (indeed the unique) 
$\overrightarrow{v}=\overrightarrow{a\cS''}\in T_a\sP^1$, 
 diminishing $[a,\cS'']$ if necessary,
 $\phi$ is affine on $((a,\cS''],\rho)$. 

 A continuous
 function $\phi$ on $(\sH^1,\rho)$ is {\itshape convex} on $(\sH^1,\rho)$ if 
 for every $\cS\in\sH^1$ and any 
 distinct $\overrightarrow{v_1},\overrightarrow{v_2}\in T_{\cS}\sP^1$,
 $(\rd_{\overrightarrow{v_1}}+\rd_{\overrightarrow{v_2}})\phi\ge 0$.
\end{definition}

The following is obtained by differentiating $\Crucial_f$ on $(\sH^1,\rho)$
using \eqref{eq:conjugate}, and improves
Rumely \cite[Proposition 2.3]{Rumely13} 
so that at any classical fixed point $a\in\bP^1$ of $f$,
$\Crucial_f$ (or $\ordRes_f$)
is directly related to the chordal derivative $f^{\#}(a)$ 
(see Definition \ref{th:chordal}); this kind of 
relationship, in a less explicit way, 
has been known only in the case where $d=2$ in \cite{DJR15}.

\begin{mainth}\label{th:convex}
Let $f\in K(z)$ be of degree $d>1$. 
$($i$)$ For every $\cS_0\in\sH^1$ and every direction $\overrightarrow{v}
=\overrightarrow{\cS_0\cS'}\in T_{\cS_0}\sP^1$,
diminishing $[\cS_0,\cS']$ if necessary,
the functions $\cS\mapsto\rho(\cS,f(\cS)\wedge_{\cS_0}\cS)$ and
$\cS\mapsto\int_{\sP^1}\rho(\cS_0,\cS\wedge_{\cS_0}\cdot)
\rd(f^*\delta_{\cS_0})$ are affine on $[\cS_0,\cS']$, and 
on their slopes, we have
$d_{\overrightarrow{v}}(\cS\mapsto\rho(\cS,f(\cS)\wedge_{\cS_0}\cS))\in\{0,1\}$
and $d_{\overrightarrow{v}}(
\cS\mapsto\int_{\sP^1}\rho(\cS_0,\cS\wedge_{\cS_0}\cdot)
\rd(f^*\delta_{\cS_0}))\in\{0,1,\ldots,d\}$.
For every $a\in\bP^1$ and every $\cS_0\in\sH^1$,
there is $\cS_0'\in(a,\cS_0]$ so 
that 
on $(a,\cS_0']$,
\begin{gather}
 \cS\mapsto\rho(\cS,f(\cS)\wedge_{\cS_0}\cS)
 \begin{cases}
 \equiv\log\max\{1,f^\#(a)\} & \text{if }f(a)=a\\
 =\rho(\cS,\cS_0)-\rho\bigl(\cS_0,f(a)\wedge_{\cS_0}a\bigr)
  & \text{if }f(a)\neq a
 \end{cases}\text{ and}\label{eq:typeI}\\
 \cS\mapsto\int_{\sP^1}\rho(\cS_0,\cS\wedge_{\cS_0}\cdot)\rd(f^*\delta_{\cS_0})
\text{ is constant}. 
\label{eq:typeIpullback}
\end{gather}

$($ii$)$ In addition to being continuous,
the function $\Crucial_f$ is 
piecewise affine and convex on $(\sH^1,\rho)$, and
is locally affine on $(\sP^1,\tilde{\rho})$ at every
type {\em I}, {\em III}, or {\em IV} point in $\sP^1$.
For every $\cS\in\sH^1$ and every $\overrightarrow{v}\in T_{\cS}\sP^1$,
we have
\begin{gather}
\rd_{\overrightarrow{v}}\Crucial_f
\in\left\{\frac{1}{2}\cdot\frac{d+1-2m}{d-1}:
m\in\{0,1,2,\ldots,d+1\}\right\}\label{eq:rangecrucial} 
\end{gather}
and, for every type {\em I} or {\em IV} point 
$\cS'\in\sP^1$ and every $\cS_0\in\sP^1\setminus\{\cS'\}$, 
if $\cS\in[\cS',\cS_0)\cap\sH^1$ is close enough to $\cS'$, then
\begin{gather}
 \rd_{\overrightarrow{\cS\cS_0}}\Crucial_f
=
\begin{cases}
     -\frac{1}{2} &\text{if }f(\cS')=\cS'\\
     -\frac{1}{2}-\frac{1}{d-1} &\text{if }f(\cS')\neq\cS' 
\end{cases}\le-\frac{1}{2}<0.\label{eq:slopeI}
\end{gather}
\end{mainth}

A fundamental consequence of Theorem \ref{th:convex}(ii) is that
$\min_{\sH^1}\Crucial_f$ exists in $\bR$, so
\begin{gather*}
 (\MinResLoc_f=)\Crucial_f^{-1}\Bigl(\min_{\sH^1}\Crucial_f\Bigr)\neq\emptyset.
\end{gather*}

\subsection{The $\Gamma$-curvature $\nu_{f,\Gamma}$ of $\Crucial_f|\Gamma$ 
and the minimum locus of $\Crucial_f|(\Gamma\cap\sH^1)$
on a non-trivial finite subtree $\Gamma$ in $\sP^1$}

\begin{definition}
For each non-empty subset $S$ in $\sP^1$,
 let $\Gamma_S$ be the subtree in $\sP^1$ 
 {\itshape spanned by} $S$, that is, the minimal subtree in $\sP^1$
 containing $S$. 
 A {\itshape finite subtree} in $\sP^1$ is a subtree in $\sP^1$ spanned
 by a non-empty and finite subset in $\sP^1$
 (but not necessarily in $\sH^1$),
 and a {\itshape finite tree in} $\sH^1$ is a finite subtree in $\sP^1$
 contained in $\sH^1$. 
A subtree in $\sP^1$ is said to be {\itshape trivial}
 if it is a singleton in $\sP^1$, and otherwise, {\itshape non-trivial}. 

 For any subtrees $\Gamma,\Gamma'$ in $\sP^1$ 
 satisfying $\Gamma'\subset\Gamma$, 
 $\iota_{\Gamma',\Gamma}$ and
 $r_{\Gamma,\Gamma'}$ are the inclusion from $\Gamma'$ to $\Gamma$
 and the retraction from $\Gamma$ to $\Gamma'$, respectively.
\end{definition}

For a discussion of the tangent space $T_{\cS}\Gamma$
{\itshape of a subtree} $\Gamma$ in $\sP^1$ at a point $\cS\in\Gamma$, 
see Subsection \ref{sec:tangent}. 
A point $\cS\in\Gamma$ is called an {\itshape end $($resp.\ branch$)$ point}
of $\Gamma$ if the valence 
\begin{gather*}
 v_{\Gamma}(\cS)=\#(T_{\cS}\Gamma)\in\bN\cup\{0,+\infty\} 
\end{gather*}
of $\Gamma$ at $\cS$
is $<2$ (resp.\ $>2$); $v_{\Gamma}(\cS)=0$ occurs if and only if 
$\Gamma=\Gamma_{\{\cS\}}=\{\cS\}$.
For a discussion of the Laplacian $\Delta_{\Gamma}$
on a subtree $\Gamma$ in $\sP^1$, see Subsection \ref{sec:laplacian}.
To each point $\cS\in\sP^1$ and each {\em direction}
$\overrightarrow{v}\in T_{\cS}\sP^1$, we associate the component 
$U_{\overrightarrow{v}}=U_{\cS,\overrightarrow{v}}
 :=\{\cS'\in\sP^1:\overrightarrow{\cS\cS'}=\overrightarrow{v} \}$
of $\sP^1\setminus\{\cS\}$.

By Theorem \ref{th:convex}(ii),
for every non-trivial finite subtree $\Gamma$ in $\sP^1$,
the minimum of the restriction $\Crucial_f|(\Gamma\cap\sH^1)$ exists in $\bR$, and 
the (signed) Radon measure $\Delta_{\Gamma}(\Crucial_f|\Gamma)$ on $\Gamma$
is supported on a finite subset in $\Gamma$ and 
satisfies 
$(\Delta_{\Gamma}(\Crucial_f|\Gamma))(\Gamma)=0$,
$(\Delta_{\Gamma}(\Crucial_f|\Gamma))|
(\Gamma\setminus\{\text{end points of }\Gamma\})\ge 0$,
and by \eqref{eq:slopeI} (and \eqref{eq:outer} below),
for every type I or IV point $\cS'\in\Gamma$,
\begin{gather}
 (\Delta_{\Gamma}(\Crucial_f|\Gamma))(\{\cS'\})
=
\begin{cases}
     -\frac{1}{2} &\text{if }f(\cS')=\cS'\\
     -\frac{1}{2}-\frac{1}{d-1} &\text{if }f(\cS')\neq\cS' 
    \end{cases}\le-\frac{1}{2}<0.\label{eq:LaplacianI}
\end{gather}

\begin{definition}[the $\Gamma$-valency measure on $\Gamma$
and the $\Gamma$-curvature]
For every non-trivial finite subtree $\Gamma$ in $\sP^1$,
let us define the (signed) Radon measures
 \begin{gather}
 \nu_{\Gamma}:=(-2)^{-1}\cdot\sum_{\cS\in\Gamma}\bigl(v_{\Gamma}(\cS)-2\bigr)\cdot
 (r_{\sP^1,\Gamma})_*\delta_{\cS}\quad\text{on }\Gamma
\quad\text{(cf.\ \cite{CR93}) and}
 \label{eq:valency}\\
 \nu_{f,\Gamma}:=\Delta_{\Gamma}\bigl(\Crucial_f|\Gamma\bigr)+\nu_{\Gamma}\quad\text{on }\Gamma;
\label{eq:crucialgeneral}
 \end{gather}
by the Euler genus theorem, the {\em valency measure} $\nu_{\Gamma}$ 
of $\Gamma$ has
the total mass $1$, and in turn so does the $\Gamma$-{\em curvature} $\nu_{f,\Gamma}$
of $\Crucial_f|\Gamma$. Moreover,
each of them is supported on a {\itshape finite} subset in $\Gamma$. 
\end{definition}

For every non-trivial finite subtree $\Gamma$ in $\sP^1$,
we also call the $\Gamma$-curvature $\nu_{f,\Gamma}$ of $\Crucial_f|\Gamma$
the $\Gamma$-{\em crucial measure} of $f$, and call
the finite subset $\supp(\nu_{f,\Gamma})$ in $\Gamma$
and the finite subtree in $\Gamma$ spanned by 
$\supp(\nu_{f,\Gamma})$ the $\Gamma$-{\em crucial set} of $f$ and 
the $\Gamma$-{\em crucial tree} of $f$, respectively.

\begin{definition}[the $\Gamma$-barycenter]
For every probability Radon measure $\nu$ on a subtree $\Gamma$ in $\sP^1$, 
the $\Gamma$-{\itshape barycenter} $\BC_{\Gamma}(\nu)$
of $\nu$ is defined as
\begin{gather*}
\BC_{\Gamma}(\nu):=\Bigl\{\cS\in\Gamma:
\text{for every }\overrightarrow{v}\in T_{\cS}\Gamma,
((\iota_{\Gamma,\sP^1})_*\nu)(U_{\overrightarrow{v}})\le\frac{1}{2}\Bigr\}.
\end{gather*} 
\end{definition}

The following is also one of our principal results;
the part (i) consists of the {\itshape geometric} formulas of 
the $\Gamma$-curvature $\nu_{f,\Gamma}$
and of the slope $\rd_{\overrightarrow{v}}\Crucial_f$ of $\Crucial_f|\Gamma$, and
the part (iii) asserts that 
for every non-trivial subtree $\Gamma$ in $\sP^1$ having no type III end points,
the positivity of the $\Gamma$-curvature $\nu_{f,\Gamma}$ of $\Crucial_f|\Gamma$,
which is stronger than the positivity of
$(\Delta_{\Gamma}(\Crucial_f|\Gamma))|(\Gamma\setminus\{\text{end points of }\Gamma\})$ mentioned above, restricts the geometry of 
the minimum locus of $\Crucial_f|(\Gamma\cap\sH^1)$ and
identifies it with the $\Gamma$-barycenter of $\nu_{f,\Gamma}$.

\begin{mainth}\label{th:weight}
Let $f\in K(z)$ be of degree $d>1$. 
Let $\Gamma$ be a non-trivial finite subtree in $\sP^1$. Then 

$($i$)$ for any $\cS_0\in\sH^1$, 
\begin{gather}
\nu_{f,\Gamma}=\frac{\Delta_\Gamma(\cS\mapsto\rho(\cS,f(\cS)\wedge_{\cS_0}\cS))
+(r_{\sP^1,\Gamma})_*(f^*\delta_{\cS_0}-\delta_{\cS_0})
}{d-1}\quad\text{on }\Gamma,
\label{eq:geometric}
\end{gather}
and for every $\cS_0\in\Gamma\cap\sH^1$ and every direction
$\overrightarrow{v}\in T_{\cS_0}\Gamma$, 
\begin{gather}
 \rd_{\overrightarrow{v}}\Crucial_f=\frac{1}{2}
-\bigl((\iota_{\Gamma,\sP^1})_*\nu_{f,\Gamma}\bigr)(U_{\overrightarrow{v}}).
\label{eq:slopeinside}
\end{gather}

$($ii$)$ For every $\cS'\in\Gamma$, 
we have $(d-1)\cdot\nu_{f,\Gamma}(\{\cS'\})\in\bN\cup\{0,-1\}$, and
$\nu_{f,\Gamma}$ is supported on $($a finite subset in$)$
$\Gamma\setminus\{\text{type {\em I or IV} points fixed by }f\}$.

Moreover, for every $\cS'\in\Gamma\setminus\{\text{type 
{\em I or IV} points not fixed by }f\}$,
we have $\nu_{f,\Gamma}(\{\cS'\})\ge 0$ unless the following statement 
\begin{gather*}
 \#T_{\cS'}\Gamma=1\quad\text{\rm and}\quad r_{\sP^1,\Gamma}(f(\cS'))\neq\cS',\\
\text{\rm and moreover }
f^{-1}\bigl(U_{\overrightarrow{\cS'f(\cS')}}\cap 
U_{\overrightarrow{f(\cS')\cS'}}\bigr)\subset 
U_{\overrightarrow{\cS'f(\cS')}}
\end{gather*}
is the case.

$($iii$)$ Let $\Gamma$ be a non-trivial finite subtree in $\sP^1$
having no type {\em III} end points. If $\nu_{f,\Gamma}\ge 0$ on $\Gamma$,
then there is  a $($unique$)$ subset $S_{f,\Gamma}$
in $\Gamma$ consisting of at most two type {\em II} points such that
 \begin{gather}
 \bigl(\Crucial_f|(\Gamma\cap\sH^1)\bigr)^{-1}\Bigl(\min_{\Gamma\cap\sH^1}\Crucial_f\Bigr)
 =\BC_{\Gamma}(\nu_{f,\Gamma})=\Gamma_{S_{f,\Gamma}}
\label{eq:mrlessential}
\end{gather}
and
\begin{gather}
 S_{f,\Gamma}=\Gamma_{S_{f,\Gamma}}\cap\bigl(\supp(\nu_{f,\Gamma})\cup\supp(\nu_{\Gamma_{\supp(\nu_{f,\Gamma})}})\bigr),\label{eq:interioressential}
 \end{gather}
and moreover, $\#S_{f,\Gamma}=1$ either 
if $d$ is even or if $\min_{\Gamma\cap\sH}\Crucial_f>\min_{\sH^1}\Crucial_f$. 
\end{mainth}

\subsection{Geometric formulas for Rumely's weight function $w_f$ 
and the $f$-crucial measure $\nu_f$}\label{sec:defining}

Let $\Fix(f)$ be the set of all {\em classical} fixed points of $f$ in $\bP^1$, that is,
$\Fix(f):=\{a\in\bP^1:f(a)=a\}$,
and set
\begin{gather*}
\Gamma_{\FP}=\Gamma_{f,\FP}:=\bigcap_{a\in\bP^1}\Gamma_{\Fix(f)\cup f^{-1}(a)}.
\end{gather*}
This is a (possibly trivial) finite subtree in $\sP^1$ 
all end points of which are type I or II and which satisfies
$\Gamma_{\FP}\cap\bP^1=\Fix(f)$.
We note that if $\Gamma_{\FP}$ is non-trivial, then the $\Gamma_{\FP}$-curvature
\begin{gather*}
 \nu_{f,\Gamma_{\FP}}=\Delta_{\Gamma}\bigl(\Crucial_f|\Gamma_{\FP}\bigr)
+\nu_{\Gamma_{\FP}} 
\end{gather*}
of $\Crucial_f|\Gamma_{\FP}$ is
a {\itshape probability} Radon measure on $\Gamma_{\FP}$ 
supported on at most $d-1$ type II points, 
and induces in turn a {\em weight} function 
\begin{gather}
 (d-1)\cdot\bigl((\iota_{\Gamma_{\FP},\sP^1})_*\nu_{f,\Gamma_{\FP}}\bigr)(\{\cdot\}):\sP^1\to\{0,1,\ldots,d-1\};\label{eq:weight} 
\end{gather}
indeed,
any {\itshape end point} $\cS'\in\sH^1$ of $\Gamma_{\FP}$ not only is type II 
but, writing $T_{\cS'}\Gamma_{\FP}=\{\overrightarrow{v}\}$, also satisfies
$f^{-1}(a)\not\subset U_{\overrightarrow{v}}$ for every $a\in U_{\overrightarrow{v}}\cap\bP^1$. In particular, if 
$r_{\sP^1,\Gamma}(f(\cS'))\neq\cS'$, then we have
$\overrightarrow{\cS'f(\cS')}=\overrightarrow{v}$ and 
$U_{\overrightarrow{\cS'f(\cS')}}\cap
U_{\overrightarrow{f(\cS')\cS'}}\cap\bP^1\neq\emptyset$, so
$f^{-1}\bigl(U_{\overrightarrow{\cS'f(\cS')}}\cap
U_{\overrightarrow{f(\cS')\cS'}}
\bigr)\not\subset U_{\overrightarrow{\cS'f(\cS')}}$.
Hence $(d-1)\nu_{f,\Gamma_{\FP}}(\{\cdot\})\in\{0,1,\ldots,d-1\}$  
on $\Gamma_{\FP}$ by Theorem \ref{th:weight}(ii), 
if $\Gamma_{\FP}$ is non-trivial.

Next, let $\Gamma_{\FR}=\Gamma_{f,\FR}$ be the 
(possibly not finite) subtree in $\sP^1$ spanned by the union of $\Fix(f)$ and 
the set of all type II fixed 
points $\cS\in\sH^1$ of $f$ satisfying $\deg_{\cS}(f)>1$,
which is non-trivial (see Rumely \cite[Lemma 4.1]{Rumely14}).
Rumely's {\itshape tree intersection theorem} 
\cite[Theorem 4.2]{Rumely14} asserts that 
\begin{gather}
 \Gamma_{\FP}=\Gamma_{\FR}.\label{eq:tree}
\end{gather}
In particular, $\Gamma_{\FP}=\Gamma_{\FR}$ is a {\itshape non-trivial} and
{\itshape finite} subtree in $\sP^1$.

The following complements Theorems \ref{th:convex} and \ref{th:weight}
and improves Rumely's slope formulas
\cite[Propositions 5.1, 5.2, 5.3, 5.4]{Rumely14}
in a suitable form for the proof of Theorem \ref{th:quantitative} below.
The proof is based not only on the difference formula
\eqref{eq:conjugate} but also on \eqref{eq:tree} 
and the three {\itshape identification lemmas}
\cite[Lemmas 2.1, 2.2, and 4.5]{Rumely14} (see Theorems \ref{th:1st3rd},
\ref{th:2nd}, and \ref{th:3rd} in Section \ref{sec:explicit}).
For the reduction $\tilde{h}\in k(z)$ modulo $\mathfrak{m}_K$ of $h\in K(z)$, see Definition \ref{th:reduction} below.

\begin{mainth}\label{th:defining}
Let $f\in K(z)$ be of degree $d>1$. Then

$($i$)$ For every $\cS'\in\sP^1$,
\begin{multline}
(d-1)\cdot\bigl((\iota_{\Gamma_{\FR},\sP^1})_*\nu_{f,\Gamma_{\FR}}\bigr)(\{\cS'\})\\
=
\begin{cases}
 \deg_{\cS'}(f)-1
 +\#\{\overrightarrow{v}\in T_{\cS'}\sP^1:
 U_{\overrightarrow{v}}\cap\Gamma_{\FR}\neq\emptyset\text{ and }f_*(\overrightarrow{v})\neq\overrightarrow{v}\}\\
 \hspace*{225pt}\text{if }f(\cS')=\cS'\in\sH^1,\\
 \max\left\{0,\#\{\overrightarrow{v}\in T_{\cS'}\sP^1:
 U_{\overrightarrow{v}}\cap\Gamma_{\FR}\neq\emptyset\}-2\right\}
 \hspace*{10pt}
   \text{if }f(\cS')\neq\cS'\in\sH^1,\\
 0\hspace*{219pt}\text{if }\cS'\in\bP^1
 \end{cases}\label{eq:computation}
\end{multline}
$($see also \eqref{eq:computationtree} in Section $\ref{sec:explicit}$
for every $\cS'\in\Gamma_{\FR})$.

$($ii$)$ Let $\Gamma$ be a non-trivial finite subtree in $\sP^1$,
and suppose that there are a point $\cS_{\Gamma}\in\Gamma_{\FR}$
and a direction $\overrightarrow{v_{\Gamma}}\in(T_{\cS_{\Gamma}}\sP^1)\setminus(T_{\cS_{\Gamma}}\Gamma_{\FR})$ such that 
$\Gamma\subset U_{\overrightarrow{v_{\Gamma}}}\cup\{\cS_{\Gamma}\}$.
Then for every $\cS'\in\Gamma$, 
\begin{multline}
(d-1)\cdot\left(\nu_{f,\Gamma}-(r_{\sP^1,\Gamma})_*\delta_{\cS_{\Gamma}}\right)(\{\cS'\})=\\
\begin{cases}
 0 \hspace*{10pt} \text{if {\rm (A1)} }\cS'\text{ is a type {\rm II} fixed point of }f\text{ such that }\widetilde{h\circ f\circ h^{-1}}=\Id_{\bP^1(k)}\\
\hspace*{50pt}
\text{for some $($and indeed any$)$ }
 h\in\PGL(2,K)\text{ sending }\cS'\text{ to }\cS_{\can},\\
\bigl((-2)\cdot\nu_{\Gamma}+
(r_{\sP^1,\Gamma})_*\delta_{\cS_{\Gamma}}
\bigr)(\{\cS'\})+1\\
\hspace*{15pt} \text{if {\rm (A2)} }\cS'\text{ is a type {\rm II} fixed point of }f
\text{ such that }\widetilde{ h\circ f\circ h^{-1}}\neq\Id_{\bP^1(k)}\\
\hspace*{50pt}\text{for some } h\in\PGL(2,K)\text{ sending }\cS'\text{ to }\cS_{\can},\\
0\hspace*{10pt}\text{if {\rm (A3)} }
\cS'\text{ is a type {\em III} or {\em IV} point fixed by }f,\\
\bigl(
(-2)\cdot\nu_{\Gamma}+(r_{\sP^1,\Gamma})_*\delta_{\cS_{\Gamma}}
\bigr)(\{\cS'\}) 
\hspace*{19pt}\text{if {\rm (B1)} } f(\cS')\neq\cS'\text{ and }\cS'\neq r_{\sP^1,\Gamma}(f(\cS')),\\
\bigl(
(-2)\cdot\nu_{\Gamma}+2(r_{\sP^1,\Gamma})_*\delta_{\cS_{\Gamma}}\bigr)(\{\cS'\})
\hspace*{14pt}\text{if {\rm (B2)} } f(\cS')\neq\cS'\text{ and }\cS'=r_{\sP^1,\Gamma}(f(\cS')).
\end{cases}
\label{eq:extendedtree}
\end{multline}

$($iii$)$ For every $\cS'\in\sP^1\setminus\Gamma_{\FR}$,
if $\cS\in[\cS',r_{\sP^1,\Gamma_{\FR}}(\cS'))\cap\sH^1$
is close enough to $\cS'$, then
\begin{gather}
 \rd_{\overrightarrow{\cS(r_{\sP^1,\Gamma_{\FR}}(\cS'))}}\Crucial_f
=
\begin{cases}
    -\frac{1}{2} &\text{if }f(\cS')=\cS'\\
    -\frac{1}{2}-\frac{1}{d-1} &\text{if }f(\cS')\neq\cS' 
    \end{cases}\le-\frac{1}{2}<0
\label{eq:slopeoutside}
\end{gather}
and, if $f(\cS')=\cS'$, then 
$f(\cS'')=\cS''$ for every $\cS''\in[\cS',r_{\sP^1,\Gamma_{\FR}}(\cS')]$ 
and moreover, for every $\cS''\in(\cS',r_{\sP^1,\Gamma_{\FR}}(\cS')]$
of type {\em II}, we also have 
$\widetilde{ h\circ f\circ h^{-1}}=\Id_{\bP^1(k)}$
for some $($indeed any$)$
$ h\in\PGL(2,K)$ sending $\cS''$ to $\cS_{\can}$. 
\end{mainth}

The condition (A1) in Theorem \ref{th:defining}(ii) is indeed the
defining condition of 
a {\em type {\em II} id-indifferent fixed} point $\cS'$ of $f$ 
in Rumely \cite[Definition 2(A)]{Rumely14}.
Taking into account
Rumely \cite[Propositions 3.1, 3.3, 3.4, 3.5]{Rumely14} on 
$\Gamma_{\FR}\setminus\Gamma_{\Fix(f)}$, 
the right hand side of \eqref{eq:computation} is nothing but
\cite[Definition 8]{Rumely14} of $w_f(\cS')$.
Also applying \eqref{eq:geometric} to $\Gamma_{\FR}$, we obtain the 
following formulas for Rumely's weight function $w_f$ and 
the $f$-crucial measure $\nu_f$. 

\begin{maincoro}
Let $f\in K(z)$ be of degree $d>1$.
Then for every $\cS_0\in\sH^1$, 
\begin{multline}
\label{eq:formulaweight}
w_f(\cdot)
=(d-1)\cdot\bigl((\iota_{\Gamma_{\FR},\sP^1})_*\nu_{f,\Gamma_{\FR}}\bigr)(\{\cdot\})\\
=(\iota_{\Gamma_{\FR},\sP^1})_*\bigl(\Delta_{\Gamma_{\FR}}
(\cS\mapsto\rho(\cS,f(\cS)\wedge_{\cS_0}\cS))
+(r_{\sP^1,\Gamma_{\FR}})_*(f^*\delta_{\cS_0}-\delta_{\cS_0})\bigr)(\{\cdot\})
\quad\text{on }\sP^1
\end{multline}
and
\begin{multline}
\nu_f=(\iota_{\Gamma_{\FR},\sP^1})_*\nu_{f,\Gamma_{\FR}}\\
=\frac{(\iota_{\Gamma_{\FR},\sP^1})_*\bigl(\Delta_{\Gamma_{\FR}}
 \bigl(\cS\mapsto\rho(\cS,f(\cS)\wedge_{\cS_0}\cS)\bigr)
+(r_{\sP^1,\Gamma_{\FR}})_*(f^*\delta_{\cS_0}
 -\delta_{\cS_0})\bigr)}{d-1}\quad\text{on }\sP^1.\label{eq:formulameasure}
\end{multline}
\end{maincoro}

The formulas \eqref{eq:formulaweight} and \eqref{eq:formulameasure}
imply that the $\Gamma_{\FR}$-crucial set $\supp(\nu_{f,\Gamma_{\FR}})$ 
and the $\Gamma_{\FR}$-crucial tree $\Gamma_{\supp(\nu_{f,\Gamma_{\FR}})}$
of $f$
coincide with the $f$-crucial set $\sP^1\setminus w_f^{-1}(0)$
and the $f$-crucial tree $\Gamma_{\supp\nu_f}=\Gamma_f$, respectively.
Hence \eqref{eq:mrlessential} and \eqref{eq:interioressential} applied to 
$\Gamma_{\FR}$ and \eqref{eq:slopeoutside} recover Rumely 
\cite[Theorem A]{Rumely14}.

\begin{maincoro}\label{th:crucialtree}
there is a $($unique$)$ subset $S_f$ in $\Gamma_{\FR}$ consisting of 
at most two type {\em II} points such that
\begin{gather*}
\MinResLoc_f=\BC_{\sP^1}(\nu_f)=\Gamma_{S_f}\quad\text{and}\quad 
S_f=\Gamma_{S_f}\cap\bigl(\supp(\nu_f)\cup\supp(\nu_{\Gamma_f})\bigr),
\end{gather*}
and moreover, $\#S_f=1$ if $d$ is even.
\end{maincoro}

\subsection{Effective bounds on the hyperbolic radii of 
the $\Gamma_{\FR}$-crucial tree of $f$ and of the minimum locus 
of $\Crucial_f$} 

The effective bound 
\eqref{eq:diamminresloc} below on the hyperbolic radius 
(centered at $\cS_{\can}$) of the minimum locus
of $\Crucial_f$ ($=\MinResLoc_f$) improves 
$\sup_{\MinResLoc_f}\rho(\cdot,\cS_{\can})\le(-2\cdot\log|\Res(\text{a minimal lift of }f)|)/(d-1)$ (Rumely \cite[Theorem 1.1]{Rumely13}). 

\begin{mainth}[see also Remark \ref{th:geomiterate}]\label{th:diam}
Let $f\in K(z)$ be of degree $d>1$. Then 
\begin{multline}
\sup_{\cS\in\Crucial_f^{-1}(\min_{\sH^1}\Crucial_f)}
\bigl((d-1)\cdot\rho(\cS,\cS_{\can})+2\cdot\rho(\cS,f(\cS)\wedge_{\can}\cS)\bigr)\\
\le -2\cdot\log|\Res(\text{a minimal lift of }f)|\label{eq:diamminresloc} 
\end{multline}
and, if in addition $d>2$, then
\begin{multline}
\sup_{\cS\in\Gamma_{\supp(\nu_{f,\Gamma_{\FR}})}}
\bigl(\rho(\cS,\cS_{\can})+\rho(\cS,f(\cS)\wedge_{\can}\cS)\bigr)\\
\le -\log|\Res(\text{a minimal lift of }f)|.\label{eq:diamcrucial}
\end{multline}
\end{mainth}

Regarding \eqref{eq:diamcrucial}, no (non-trivial and effective) 
bounds on
$\sup_{\Gamma_f}\rho(\cdot,\cS_{\can})$
seem to have been known unless the $\Gamma_{\FR}$-crucial tree 
$\Gamma_{\supp(\nu_{f,\Gamma_{\FR}})}$ of $f$
($=$ the $f$-crucial tree $\Gamma_f$) is a singleton.
We 
note that $\supp(\nu_{f,\Gamma_{\FR}})=\supp(\nu_f)$, 
that $\#\supp(\nu_{f,\Gamma_{\FR}})=1$ if $d=2$ (see \eqref{eq:weight}), 
and that
if $\#\supp(\nu_{f,\Gamma_{\FR}})=1$, 
$\Gamma_{\supp(\nu_{f,\Gamma_{\FR}})}=\Crucial_f^{-1}(\min_{\sH^1}\Crucial_f)$.
Hence when $d=2$,
\eqref{eq:diamcrucial} still applies and improves Rumely's radius estimate for
$\MinResLoc_f$ mentioned above.

\subsection{Application to dynamics of $f$ on $\sP^1$}\label{sec:dynamics}

As an application of our {\itshape exact} computation 
of the $\Gamma$-curvatures in Theorem \ref{th:defining}(ii), 
we establish
the following {\itshape quantitative equidistribution} 
of the sequence $(\nu_{f^n})_n$ of $f^n$-crucial measures
towards the $f$-{\itshape equilibrium $($or canonical$)$ measure} 
$\mu_f$ on $\sP^1$, which is by definition (or characterized as)
the unique probability Radon measure $\nu$ on $\sP^1$ such that 
$f^*\nu=d\cdot \nu$ on $\sP^1$ and that $\nu(E(f))=0$, 
where $E(f):=\{a\in\bP^1:\#\bigcup_{n\in\bN}f^{-n}(a)<\infty\}$
is the exceptional set of $f$.

\begin{mainth}\label{th:quantitative}
Let $f\in K(z)$ be of degree $d>1$. Then 
for every continuous test function $\phi$ on $\sP^1$ such that
$\phi|\Gamma$ is continuous on $(\Gamma,\rho)$ for a finite tree $\Gamma$ in
$\sH^1$ and that $\phi=(r_{\sP^1,\Gamma})^*\phi$ on $\sP^1$,
every $n\in\bN$, and every $\cS_0\in\sP^1$, 
setting $\Gamma_n':=r_{\sP^1,\Gamma}(\Gamma_{f^n,\FR})$, we have
\begin{multline}
\left|\int_{\sP^1}\phi\rd(\nu_{f^n}-\mu_f)\right|
\le\frac{2\bigl(C_{\cS_0,f}/(d-1)+\sup_{\Gamma}\rho(\cdot,\cS_0)\bigr)}{d^n-1}
\cdot|\Delta\phi|(\sP^1)\\
+\frac{2\cdot\#(\{\text{end points of }\Gamma\}\setminus\Gamma_n')}{d^n-1}\cdot\sup_{\sP^1}|\phi|,\label{eq:quantitative}
\end{multline}
where $C_{\cS_0,f}:=\sup_{\cS\in\sP^1}\int_{\sP^1}\rho(\cS_0,\cS\wedge_{\cS_0}\cdot)\rd(f^*\delta_{\cS_0})<\infty$.
\end{mainth}

\begin{remark}
In particular, also using, e.g., \cite[(B) in Proposition 5.4]{BR10}$)$, 
\begin{gather*}
 \lim_{n\to\infty}\nu_{f^n}=\mu_f\quad\text{weakly on }\sP^1
\end{gather*}
(\cite[Theorem 2]{Jacobs17}).
As an order estimate, \eqref{eq:quantitative} (for $\cS_0=\cS_{\can}$)
was due to Jacobs \cite[Theorem 5]{Jacobs17}.
In the proof of \cite[Theorem 5]{Jacobs17},
Rumely's (rather involved) {\itshape second persistence lemma}
\cite[Lemma 9.5]{Rumely14} was also invoked.
Our argument replaces it with the latter half of Theorem \ref{th:defining}(iii).
\end{remark}

\subsection{Organization of the paper} 
In Section \ref{sec:background}, we recall background concerning
the tree structure 
on $\sP^1$, the canonical action on $\sP^1$ of $h\in K(z)$, and
Rivera-Letelier's decomposition of $\deg_{\cS}(h)$
(see Theorem \ref{th:local}). 
In Sections \ref{sec:geometric}, \ref{sec:convex}, 
and \ref{sec:weight}, we establish Theorems \ref{th:resultant}, 
\ref{th:convex}, and \ref{th:weight}, respectively, by computation
based on the global and explicit difference formula \eqref{eq:conjugate} 
for $\Crucial_f$, which is independent of (local) computation
for the (indirectly defined) $\ordRes_f$ in Rumely \cite{Rumely13,Rumely14}. 
In Section \ref{sec:explicit}, recalling also 
Rivera-Letelier's decomposition of $(h^*\delta_{\cS})(U_{\overrightarrow{v}})$
in terms of the directional local degree and the surplus local degree of $h\in K(z)$
on $U_{\overrightarrow{v}}$, $\overrightarrow{v}\in T_{\cS}\sP^1$,
(see Theorem \ref{th:semiglobal}) and Rumely's 
three identification lemmas (see Theorems \ref{th:1st3rd}, \ref{th:2nd},
\ref{th:3rd}), we prove Theorem \ref{th:defining}.
In Section \ref{sec:diam}, we show Theorem \ref{th:diam}
as an application of Theorems \ref{th:convex}, \ref{th:weight} 
(especially \eqref{eq:slopeinside}) and \ref{th:defining}.
In Section \ref{sec:quantitative},
we show Theorem \ref{th:quantitative}.

\section{Background}
\label{sec:background}

Recall that $\pi:K^2\setminus\{(0,0)\}\to\bP^1$ is the
canonical projection such that $\pi(0,1)=\infty$ and 
that $\pi(p_0,p_1)=p_1/p_0$
if $p_0\neq 0$. With the wedge product $(p_0,p_1)\wedge(q_0,q_1):=p_0q_1-p_1q_0$ 
and the maximum norm $\|(p_0,p_1)\|:=\max\{|p_0|,|p_1|\}$ on $K^2$, 
the {\itshape normalized} chordal metric $[z,w]_{\bP^1}$ on $\bP^1$ is
defined by
\begin{gather}
 (z,w)\mapsto [z,w]_{\bP^1}:=\frac{|p\wedge q|}{\|p\|\cdot\|q\|}
\quad\text{on }\bP^1\times\bP^1,\label{eq:chordaldef}
\end{gather}
where $p\in\pi^{-1}(z),q\in\pi^{-1}(w)$, so 
\begin{multline}
[z,\infty]_{\bP^1}=\frac{1}{\|(1,z)\|}\text{ on }K,\quad
|z-w|=\frac{[z,w]_{\bP^1}}{[z,\infty]_{\bP^1}[w,\infty]_{\bP^1}}
\text{ on }K\times K,\label{eq:chordalaffine}
\end{multline}
and $\PGL(2,\cO_K)$ acts on $(\bP^1,[z,w])$ isometrically.

\begin{definition}\label{th:chordal}
 The {\itshape chordal derivative} of $h\in K(z)$ is defined by
\begin{gather*}
  z\mapsto h^\#(z):=\lim_{w\to z}\frac{[h(w),h(z)]_{\bP^1}}{[w,z]_{\bP^1}}\quad\text{on }\bP^1. 
\end{gather*} 
\end{definition}

\begin{definition}[see, e.g., {\cite[\S2.3]{SilvermanDynamics}}]\label{th:reduction}
Let $\tilde{c}\in k$ be the class of $c\in\cO_K$ modulo $\mathfrak{m}_K$. 
The reduction $\tilde{a}$ modulo $\mathfrak{m}_K$
of a point $a=[p_0,p_1]\in\bP^1(K)$,
where $(p_0,p_1)\in K^2\setminus\{(0,0)\}$ is {\itshape minimal}
in that 
$\|(p_0,p_1)\|=1$, 
is defined by $[\tilde{p}_0,\tilde{p}_1]\in\bP^1(k)$, 
and the reduction $\tilde{h}$ modulo $\mathfrak{m}_K$ of
$h\in K(z)$ of degree $>0$ is defined by 
\begin{gather*}
\tilde{h}(\zeta):=
\frac{\tilde{H}_1(1,\zeta)/\GCD(\tilde{H}_0(1,\zeta),\tilde{H}_1(1,\zeta))}
{\tilde{H}_0(1,\zeta)/\GCD(\tilde{H}_0(1,\zeta),\tilde{H}_1(1,\zeta))}\in k(\zeta),
\end{gather*} 
where $(H_0(z_0,z_1),H_1(z_0,z_1))$ is a {\itshape minimal} lift of $h$ 
(recall Section \ref{sec:intro}) and
we set $\tilde{P}(\zeta):=\sum_{j=0}^{\deg P}\tilde{c}_j\zeta^j\in k[\zeta]$
for every $P(z)=\sum_{j=0}^{\deg P}c_jz^j\in\cO_K[z]$.
\end{definition}
For every $h\in K(z)$ of degree $>0$ and
every $a\in\bP^1$, $\tilde{h}(\tilde{a})=\widetilde{h(a)}$.
For every $f\in K(z)$ of degree $>1$, $f$ has a potentially good reduction
(that is, $f^{-1}(\cS')=\{\cS'\}$ for some type II point $\cS'\in\sH^1$)
if and only if there is $ h\in\PGL(2,K)$ such that
$\deg(\widetilde{ h\circ f\circ h^{-1}})=\deg h$.

\subsection{Tangent spaces on $\sP^1$}\label{sec:tangent}

Recall the description of $\sP^1$ in Section \ref{sec:intro}.
For every subtree $\Gamma$ in $\sP^1$ and
every point $\cS\in\Gamma$, the {\itshape tangent 
space} $T_{\cS}\Gamma$ at $\cS$ is the set of all {\itshape germs}, 
which are called {\itshape directions}, $\overrightarrow{\cS\cS'}$ 
of a non-empty left half open interval $(\cS,\cS']$ in $\Gamma$; 
$\overrightarrow{\cS\cS}$ is
undefined and, for every $\cS\in\Gamma$, $v_{\Gamma}(\cS):=\#T_{\cS}\Gamma\in\bN\cup\{0,+\infty\}$ is called the {\em valence} of $\Gamma$ at $\cS$.
For every $\cS\in\sP^1$ and every direction
$\overrightarrow{v}\in T_{\cS}\sP^1$, set 
\begin{gather*}
 U_{\overrightarrow{v}}=U_{\cS,\overrightarrow{v}}
 :=\{\cS'\in\sP^1:\overrightarrow{\cS\cS'}=\overrightarrow{v} 
 \}.
\end{gather*}
The collection $\{U_{\cS,\overrightarrow{v}}:\cS\in\sP^1\text{ and }\overrightarrow{v}\in T_{\cS}\sP^1\}$ 
is a quasi-open basis of
the Gel'fand (i.e., weak or pointwise-convergence) topology on $\sP^1$, so in particular
each $ U_{\cS,\overrightarrow{v}}$ is a component of $\sP^1\setminus\{\cS\}$
and $\bP^1$ is dense in $\sP^1$.
For every type I or IV point $\cS\in\sP^1$, $\#T_{\cS}\sP^1=1<2$ so
$\cS$ is an end point of $\sP^1$, and
for every type III point $\cS\in\sP^1$, $\#T_{\cS}\sP^1=2$.
For every type II point $\cS'\in\sP^1$, 
fixing such $ h\in\PGL(2,K)$ that $ h(\cS')=\cS_{\can}$,
$T_{\cS'}\sP^1$ can be identified with $\bP^1(k)$ in such a way that 
$\bP^1(k)\ni\tilde{a}\mapsto\overrightarrow{v}_a^{( h)}\in T_{\cS'}\sP^1$
is bijective, where $\overrightarrow{v}_a^{( h)}$ is the unique
direction $\overrightarrow{v}\in T_{\cS'}\sP^1$ so
that $h(U_{\overrightarrow{v}})=U_{\overrightarrow{\cS_{\can}a}}$ 
(see, e.g., Baker--Rumely \cite[\S2.6]{BR10},
Jonsson \cite[\S2.1.1 and \S3.5]{Jonsson15}).

\subsection{Generalized Hsia kernels, the hyperbolic space, and Laplacians}\label{sec:laplacian}
The {\itshape generalized Hsia kernel} $[\cS,\cS']_{\can}$
on $\sP^1$ {\itshape with respect to} $\cS_{\can}$
is the unique upper semicontinuous and separately continuous
extension to $\sP^1(\times\sP^1)$
of the normalized chordal metric $[z,w]_{\bP^1}$ on $\bP^1(\times\bP^1)$
(see \cite[\S 4.3]{BR10}). We omit the precise definition
of the {\itshape hyperbolic $($or path distance$)$ metric} $\rho$ on $\sH^1$,
which satisfies
$\rho(\cS,\cS')=\log(\diam(\cS')/\diam(\cS))$ 
for every $\cS\in\sH^1$ and every $\cS'\in[\cS,\infty]\cap\sH^1$,
and is related to the above kernel function $[\cS,\cS']_{\can}$ by the formula
\begin{gather}
 -\log[\cS,\cS']_{\can}=\rho(\cS_{\can},\cS\wedge_{\can}\cS')
\quad\text{on }\sH^1\times\sH^1\label{eq:GromovGauss}
\end{gather}
(see Baker--Rumely \cite[\S 2.7]{BR10},
Favre--Rivera-Letelier \cite[\S 3.4]{FR06}); 
fixing $\cS'\in\sH^1$, $\cS\mapsto\log[\cS,\cS']_{\can}$ is
continuous on $(\sH^1,\rho)$, locally constant on $\sP^1$
except on $[\cS_{\can},\cS']$, and affine on $([\cS_{\can},\cS'],\rho)$
(see the more general \eqref{eq:Gromovevery} 
below).

The linear fractional transformations group $\PGL(2,K)$ 
acts on $(\sH^1,\rho)$ isometrically.
The set of all type II points in $\sP^1$ is dense 
in $(\sH^1,\rho)$ (hence also dense in $\sP^1$). Extending $\rho$ to a {\itshape generalized}
metric $\tilde{\rho}$ on $\sP^1$ 
by setting $\tilde{\rho}(\cS,\cS')=+\infty$ unless
both $\cS$ and $\cS'$ are in $\sH^1$, 
the metric space $(\sH^1,\rho)$ 
coincides with the {\itshape hyperbolic space} in $(\sP^1,\tilde{\rho})$ 
in the sense of Jonsson \cite[\S 2.2.1]{Jonsson15},
and $\sP^1$ coincides with the direct limit of the directed set of
all finite trees in $\sH^1$, partially ordered by inclusions.
The {\itshape generalized Hsia kernel} 
on $\sP^1$ {\itshape with respect to} a point $\cS_0\in\sP^1$ is defined by
\begin{gather}
 [\cS,\cS']_{\cS_0}:=\frac{[\cS,\cS']_{\can}}{[\cS,\cS_0]_{\can}[\cS',\cS_0]_{\can}}\quad\text{on }\sP^1\times\sP^1
\label{eq:Hsiaevery}
\end{gather}
(see Baker--Rumely \cite[\S 4.4]{BR10}),
where we adopt
the convention that $1/0=0/0^2=+\infty$; for every $\cS_0\in\sH^1$,
\begin{gather}
\log[\cS,\cS']_{\cS_0}
=-\tilde{\rho}(\cS_0,\cS\wedge_{\cS_0}\cS')-\log[\cS_0,\cS_0]_{\can}
\quad\text{on }\sP^1\times\sP^1
\tag{\ref{eq:GromovGauss}$'$}
\label{eq:Gromovevery}
\end{gather}
(see also \eqref{eq:fundamental} below).

For every point $\cS\in\sH^1$ and every direction 
$\overrightarrow{v}\in T_{\cS}\sP^1$, 
let $\rd_{\overrightarrow{v}}=(\rd_{\overrightarrow{v}})_{\cS}$ be the (distributional) directional derivation
operator at $\cS$ with respect to $\overrightarrow{v}$
on the space of continuous functions $\phi$ on $(\sH^1,\rho)$;
if $\phi$ is affine on a non-trivial interval $([\cS_0,\cS_1],\rho)$, then
\begin{gather}
 \rd_{\overrightarrow{\cS_0\cS_1}}\phi=\lim_{\cS\to\cS_0\text{ in }((\cS_0,\cS_1],\rho)}\frac{\phi(\cS)-\phi(\cS_0)}{\rho(\cS,\cS_0)}.\label{eq:derivation}
\end{gather}
For every subtree $\Gamma$ in $\sP^1$ intersecting $\sH^1$, let $\Delta_{\Gamma}$ be
the (distributional) Laplacian on $\Gamma$ and set $\Delta:=\Delta_{\sP^1}$,
so that for any $\cS_0,\cS_1\in\sP^1$,
\begin{gather}
 \Delta\bigl(-\tilde{\rho}(\cS_0,\cdot\wedge_{\cS_0}\cS_1)\bigr)=\delta_{\cS_1}-\delta_{\cS_0}\quad\text{on }\sP^1\label{eq:fundamental}
\end{gather}
and that for any subtrees $\Gamma,\Gamma'$ in $\sP^1$ 
satisfying $\Gamma'\subset\Gamma$ and $\Gamma'\cap\sH^1\neq\emptyset$,
we have
\begin{gather}
 \Delta_{\Gamma'}\circ(\iota_{\Gamma',\Gamma})^*=(r_{\Gamma,\Gamma'})_*\circ\Delta_{\Gamma}
\quad\text{and}\quad\Delta_{\Gamma}\circ(r_{\Gamma,\Gamma'})^*=(\iota_{\Gamma',\Gamma})_*\circ\Delta_{\Gamma'}
\label{eq:laplacian}
\end{gather}
on the space of all functions of {\itshape bounded derivative variation} (BDV) 
on $(\Gamma,\tilde{\rho})$
and on that of those on $(\Gamma',\tilde{\rho})$
(see \cite[\S 5.4]{BR10} for the class BDV), respectively; if 
$\Gamma$ is a finite subtree in $\sP^1$, then 
for every $\cS\in\Gamma\cap\sH^1$ and every continuous and piecewise affine
function $\phi$ on $(\Gamma,\rho)$, we have
\begin{gather}
 (\Delta_{\Gamma}\phi)(\{\cS\})
 =\sum_{\overrightarrow{v}\in T_{\cS}\Gamma}\rd_{\overrightarrow{v}}\phi.\label{eq:outer}  
\end{gather}
Moreover, for every $h\in K(z)$ of degree $>0$, the following functoriality
\begin{gather}
 \Delta\circ h^*=h^*\circ\Delta\label{eq:functorial}
\end{gather}
holds (for more details, see the following references;
the construction of $\Delta$ is by Favre--Jonsson \cite{FJbook},
Baker--Rumely \cite[\S 5]{BR10}, Favre--Rivera-Letelier \cite[\S 4.1]{FR06},
and Thuillier \cite{ThuillierThesis}; 
see also Jonsson \cite[\S 2.5]{Jonsson15}.
In \cite{BR10} the opposite sign convention on $\Delta$ was adopted).

\subsection{The actions on the space of Radon measures and the tangent spaces}\label{sec:directional}
Let $h\in K(z)$ be of degree $d_0>0$.
Recall the description of the canonical
action on $\sP^1$ of $h$ and the local degree function $\deg_{\,\cdot}(h):\sP^1\to\{1,\ldots,\deg h\}$ of (the extended) $h$ in Section \ref{sec:intro}. 
The (extended) local degree function $\deg_{\,\cdot}(h)$ of $h$ on $\sP^1$ induces
a pullback action $h^*$ of $h$ on the space of Radon measures $\nu$ on $\sP^1$
such that for every $\cS\in\sP^1$, when $\nu=\delta_{\cS}$, 
\begin{gather*}
 h^*\delta_{\cS}
 =\sum_{\cS'\in h^{-1}(\cS)}(\deg_{\cS'}(h))\cdot\delta_{\cS'}\quad\text{on }\sP^1
\end{gather*}
(the construction of $\cS\mapsto\deg_{\cS}(h)$ on $\sP^1$ is by Baker--Rumely \cite[\S9]{BR10},
Favre--Rivera-Letelier \cite[\S2.2]{FR09}, 
and Thuillier \cite{ThuillierThesis}; see also Jonsson \cite[\S 4.6]{Jonsson15}).

For every $\cS\in\sP^1$, there is a {\itshape tangent map} 
$h_*=(h_*)_{\cS}:T_{\cS}\sP^1\to T_{h(\cS)}\sP^1$ of $h$ at $\cS$ 
such that for every direction 
$\overrightarrow{v}=\overrightarrow{\cS\cS'}\in T_{\cS}\sP^1$, 
diminishing $[\cS,\cS']$ if necessary,
$h$ maps $[\cS,\cS']$ homeomorphically onto $[h(\cS),h(\cS')]$ and
\begin{gather*}
 h_*(\overrightarrow{v})=\overrightarrow{h(\cS)h(\cS')};
\end{gather*}
see, e.g., Jonsson \cite[\S 2.6, \S 4.5]{Jonsson15} for the precise definition
and more details about $h_*$.

\begin{definition}[the directional local degree]\label{th:directionaldegree}
Let $h\in K(z)$ be of degree $>0$. 

For every type II point $\cS\in\sP^1$,
 fixing $\ell_1\in\PGL(2,K)$ sending $\cS$ to $\cS_{\can}$ and
 $\ell_2\in\PGL(2,K)$ sending (the type II point) $h(\cS)$ to $\cS_{\can}$,
 the tangent map $h_*=(h_*)_{\cS}$
 can be identified with the action on $\bP^1(k)$ of the reduction
 $\widetilde{\ell_2\circ h\circ\ell_1^{-1}}
 \in k(\zeta)$ of $\ell_2\circ h\circ\ell_1^{-1}$
 (indeed $\deg(\widetilde{\ell_2\circ h\circ\ell_1^{-1}})=\deg_{\cS}(h)$) so that
\begin{gather*}
 h_*(\overrightarrow{v})
 =\overrightarrow{v}_{\ell_2\circ h\circ\ell_1^{-1}(a)}^{(\ell_2)}
 \quad\text{when } a\in\bP^1\text{ and } 
 \overrightarrow{v}=\overrightarrow{v}_a^{(\ell_1)}\in T_{\cS}\sP^1,
\end{gather*} 
and then $m_{\overrightarrow{v}}(h)
:=\deg_{\tilde{a}}\bigl(\widetilde{\ell_2\circ h\circ\ell_1^{-1}}\bigr)\in
 \{1,\ldots,\deg_{\cS}(h)\}$ is well-defined. 
 For every type I, III, or IV point $\cS\in\sP^1$
 and every direction $\overrightarrow{v}\in T_{\cS}\sP^1$,
 we also set $m_{\overrightarrow{v}}(h):=\deg_{\cS}(h)$.

 In any case, we call the $m_{\overrightarrow{v}}(h)$
 the {\itshape directional local degree} of $h$ on $U_{\overrightarrow{v}}$.
\end{definition}

Let us recall the following {\itshape local}
results 
\cite[Proposition 3.1]{Juan05} and
\cite[Lemmas 5.3 and 5.4]{Juan03Compositio}
by Rivera-Letelier 
(cf.\ a {\itshape semiglobal} result in Theorem \ref{th:semiglobal} below
also by Rivera-Letelier).
For a proof of Theorem \ref{th:local}
based on an algebraic definition or characterization of $\deg_{\cS}(h)$,
see Jonsson \cite[\S4.6]{Jonsson15}. 

\begin{theorem}[{\cite[Proposition 3.1]{Juan05}}]\label{th:local}
 Let $h\in K(z)$ be of degree $>0$. Then
 for every $\cS\in\sP^1$ and every direction
 $\overrightarrow{v}=\overrightarrow{\cS\cS'}\in T_{\cS}\sP^1$,
 diminishing $[\cS,\cS']$ if necessary,
 not only $h$ maps $[\cS,\cS']$ homeomorphically onto $[h(\cS),h(\cS')]$ but also,
 for any $\cS_1,\cS_2\in[\cS,\cS']$, 
 \begin{gather}
\tilde{\rho}\bigl(h(\cS_1),h(\cS_2)\bigr)
=m_{\overrightarrow{v}}(h)\cdot\tilde{\rho}(\cS_1,\cS_2);\label{eq:expansion}
 \end{gather}
in particular, the action of $h$ on $(\sH^1,\rho)$
is $(\deg h)$-Lipschitz continuous, and for every {\em fixed}
point $\cS\in\sH^1$ of $h$
and every direction $\overrightarrow{\cS\cS'}\in T_{\cS}\sP^1$ {\em fixed} by $h_*$,
diminishing $[\cS,\cS']$ if necessary, we have $[\cS,\cS']\subset[\cS,h(\cS')]$.

 Moreover, for every $\cS\in\sP^1$ and every
 $\overrightarrow{w}\in T_{h(\cS)}\sP^1$,
 \begin{gather}
 \sum_{\overrightarrow{v}\in T_{\cS}\sP^1:h_*(\overrightarrow{v})
  =\overrightarrow{w}}m_{\overrightarrow{v}}(h)
  =\deg_{\cS}(h).\label{eq:directionaldegree}
 \end{gather}
\end{theorem}

\begin{theorem}[{\cite[Lemmas 5.3 and 5.4]{Juan03Compositio} (see also \cite[Lemma 10.80]{BR10})}]\label{th:indifferent}
Let $h\in K(z)$ be of degree $>1$. Then for every
type {\em III} or {\em IV} point $\cS$ {\em fixed} by $h$, we have
$\deg_{\cS}(h)=1$. Moreover, for every
type {\em III} or {\em IV} point $\cS$ {\em fixed} by $h$ and
every direction $\overrightarrow{w}\in T_{\cS}\sP^1$, we have
$h_*(\overrightarrow{w})=\overrightarrow{w}$
$($so $m_{\overrightarrow{w}}(h)=\deg_{\cS}(h)=1)$.
\end{theorem}

\section{Proof of Theorem \ref{th:resultant}}
\label{sec:geometric}

Let $f\in K(z)$ be of degree $d>1$. 
We begin with the following observation;
for any $h,\ell\in\PGL(2,K)$ such that $h(\cS_{\can})=\ell(\cS_{\can})$,
we have $\iota:=h^{-1}\circ \ell\in\PGL(2,\cO_K)$ since
$\iota(\cS_{\can})=\cS_{\can}$. Fixing
a (minimal) lift $I\in\GL(2,\cO_K)$ of $\iota$,
$I^{-1}\circ(\text{a minimal lift of }h^{-1}\circ f\circ h)\circ I$
is a minimal lift of $\ell^{-1}\circ f\circ\ell$
since the norm $\|\cdot\|$ is invariant under
$I$, and then
\begin{align}
\notag  &|\Res(\text{a minimal lift of }\ell^{-1}\circ f\circ\ell)|\\
 &=|(\det I)^{d(d-1)}\cdot\Res(\text{a minimal lift of }h^{-1}\circ f\circ h)|
\label{eq:resultantDeMarco}\\
\notag&=|\Res(\text{a minimal lift of }h^{-1}\circ f\circ h)|.
\end{align}

{\bfseries (a)} Let us see the equality \eqref{eq:resultant} for each
$h\in\PGL(2,K)$. Set 
\begin{gather*}
 \cS=\cS_h:=h(\cS_{\can}) 
\end{gather*}
and fix $a\in K$ and $b\in K^*$ such that
$B_{\cS}={\{z\in K:|z-a|\le|b|\}}$. Then by the definition of $\diam\cS$,
the latter half of \eqref{eq:chordalaffine}, and \eqref{eq:Hsiaevery}, 
we compute
\begin{gather}
|b|=\diam\cS=
\begin{cases}
[\cS,\cS]_{\infty} =\dfrac{[\cS,\cS]_{\can}}{[\cS,\infty]_{\can}^2}\quad\text{and}\\
[\cS,a]_{\infty}
 =\dfrac{[\cS,a]_{\can}}{[\cS,\infty]_{\can}\cdot[a,\infty]_{\bP^1}}.
\end{cases}
\label{eq:diameter}
\end{gather}
Fix also a {\itshape minimal} lift $F=(F_0,F_1)$ of $f$. 
To simplify the computation, set 
\begin{gather*}
 h_{a,b}(z):=a+bz\in\PGL(2,K), 
\end{gather*}
so that $h_{a,b}(\cS_{\can})=\cS(:=h(\cS_{\can}))$, 
and fix a (not necessarily minimal) lift 
\begin{gather*}
 H_{a,b}(p_0,p_1):=(p_0,ap_0+bp_1)\in\GL(2,K) 
\end{gather*}
of $h_{a,b}$, so that 
$\det H_{a,b}=b$ and $H_{a,b}^{-1}(p_0,p_1)=b^{-1}(bp_0,p_1-ap_0)\in\GL(2,K)$.
We also note that for every $c\in K^*$,
\begin{gather}
 \Res(c\cdot H_{a,b}^{-1}\circ F\circ H_{a,b})
=c^{2d}\cdot b^{d(d-1)}\cdot\Res F;\label{eq:resultantcong} 
\end{gather}
$c$ will be chosen so that
$|c|^{-1}=\sup_{z\in\cO_K}\|H_{a,b}^{-1}\circ F\circ H_{a,b}(1,z)\|$,
or equivalently, that $c\cdot H_{a,b}^{-1}\circ F\circ H_{a,b}$ is a {\itshape minimal} lift of
$h_{a,b}^{-1}\circ f\circ h_{a,b}$.

Then we can compute ($2d(d-1)\times$) the right hand side of \eqref{eq:resultant} as
\begin{align}
\notag&-\log\frac{|\Res(\text{a minimal lift of }h^{-1}\circ f\circ h)|}{|\Res(\text{a minimal lift of }f)|}\\
\notag=&-d(d-1)\log|b|
+2d\cdot\log\sup_{z\in\cO_K}\|H_{a,b}^{-1}\circ F\circ H_{a,b}(1,z)\|\\
\notag=&-d(d-1)\log|b|+2d\cdot\log\max\left\{|F_0(1,\cdot)|_{B_{\cS}},
\frac{|F_1(1,\cdot)-aF_0(1,\cdot)|_{B_{\cS}}}{|b|}\right\}\\
\notag=&-d(d-1)\log[\cS,\cS]_{\can}+2d(d-1)\log[\cS,\infty]_{\can}\\
\notag&+2d\cdot\log\max\left\{[f(\cS),\infty]_{\can},
[\cS,\infty]_{\can}\frac{[f(\cS),a]_{\can}}{[\cS,a]_{\can}}
\right\}
\notag+2d\cdot\log\|F(1,\cS)\|\\
=&-d(d-1)\log[\cS,\cS]_{\can}
+2d\cdot\log\max\left\{\frac{[f(\cS),\infty]_{\can}}{[\cS,\infty]_{\can}},
\frac{[f(\cS),a]_{\can}}{[\cS,a]_{\can}}\right\}+2d\cdot T_F(\cS),\label{eq:takeoff}
\end{align}
where the first equality follows from \eqref{eq:resultantDeMarco} and
\eqref{eq:resultantcong}, and so does the second
from the computation of $H_{a,b}^{-1}$ (and the definition of 
the supremum seminorm $|\cdot|_{B_{\cS}}$ on $K[z]$).
On the third equality, we compute
not only
\begin{align*}
-d(d-1)\log|b|
=&-d(d-1)\log\left([\cS,\cS]_{\can}\cdot[\cS,\infty]_{\can}^{-2}\right)\\
=&-d(d-1)\log[\cS,\cS]_{\can}+2d(d-1)\log[\cS,\infty]_{\can}
\end{align*} 
using the former case in \eqref{eq:diameter}, but also
\begin{gather*}
2d\cdot\log\max\left\{|F_0(1,\cdot)|_{B_{\cS}},
\frac{|F_1(1,\cdot)-aF_0(1,\cdot)|_{B_{\cS}}}{|b|}\right\}\\
=2d\cdot\log\max\left\{|F_0(1,\cdot)|_{\cS},
\frac{|F_1(1,\cdot)-aF_0(1,\cdot)|_{\cS}}{|b|}\right\}\\
=2d\cdot\lim_{\bP^1\ni z\to\cS\text{ in }\sP^1}\log\max\left\{|F_0(1,z)|,
\frac{|F_1(1,z)-aF_0(1,z)|}{|b|}\right\}\\
=2d\cdot\lim_{\bP^1\ni z\to\cS\text{ in }\sP^1}\log\biggl(\|F(1,z)\|\max\left\{[f(z),\infty]_{\bP^1},
[a,\infty]_{\bP^1}^{-1}\frac{[f(z),a]_{\bP^1}}{|b|}\right\}\biggr)\\
=2d\cdot\log\max\left\{[f(\cS),\infty]_{\can},
[a,\infty]_{\bP^1}^{-1}\frac{[f(\cS),a]_{\can}}{|b|}\right\}
+2d\cdot\log\|F(1,\cS)\|\\
=2d\cdot\log\max\left\{[f(\cS),\infty]_{\can},
[\cS,\infty]_{\can}\frac{[f(\cS),a]_{\can}}{[\cS,a]_{\can}}\right\}
+2d\cdot\log\|F(1,\cS)\|
\end{gather*}
from the definition of the Gel'fand (i.e., weak or pointwise-convergence)
topology of $\sP^1$, 
the defining equality \eqref{eq:chordaldef} of the chordal metric on $\bP^1$,
the separate continuity of the generalized Hsia kernels on $\sP^1\times\sP^1$
with respect to $\cS_{\can}$
and the continuity of the (canonical) action of $f$ on $\sP^1$,
and the latter case in \eqref{eq:diameter}. Hence the third equality holds.
The final equality 
holds by recalling that the function 
\begin{gather}
 \log\|F\|-d\cdot\log\|\cdot\|\quad\text{on }K^2\setminus\{(0,0)\}\label{eq:cohomological}
\end{gather}
descends to the function $\log\|F(1,\cdot)\|+d\cdot\log[\cdot,\infty]_{\bP^1}$ on
$\bP^1\setminus\{\infty\}$ under the canonical projection 
$\pi:K^2\setminus\{(0,0)\}\to\bP^1$, 
and in turn uniquely extends to a continuous function, say $T_F$, on $\sP^1$
(see, e.g., \cite[\S 2]{okulog}). 

We will show that each of the three terms in the final expression \eqref{eq:takeoff}
in the above computation of $2d(d-1)\times($the right hand side 
of \eqref{eq:resultant}) gives rise to one of the terms in
$2d(d-1)\cdot\Crucial_f$ (recall Definition \ref{th:crucialdef}).

By \eqref{eq:GromovGauss}, the first and second terms are 
computed as
$-\log[\cS,\cS]_{\can}=\rho(\cS,\cS_{\can})$ and 
\begin{align}
\notag&\log\max\left\{\frac{[f(\cS),\infty]_{\can}}{[\cS,\infty]_{\can}},
\frac{[f(\cS),a]_{\can}}{[\cS,a]_{\can}}\right\}\\
=&\max\bigl\{\rho(\cS\wedge_{\can}\infty,\cS_{\can})
-\rho(f(\cS)\wedge_{\can}\infty,\cS_{\can}),\label{eq:division}\\
\notag&\quad\quad\quad\rho(\cS\wedge_{\can}a,\cS_{\can})
-\rho(f(\cS)\wedge_{\can}a,\cS_{\can})\bigr\},
\end{align}
respectively.
Combining the formula
$\Delta T_{F}=f^*\delta_{\cS_{\can}}-d\cdot\delta_{\cS_{\can}}$ on $\sP^1$
(see, e.g., \cite[Definition 2.8]{OkuCharacterization}), 
the equality \eqref{eq:fundamental}, and the equality
\begin{gather*}
 T_F(\cS_{\can})=\sup_{z\in\cO_K}(\log\|F(1,z)\|-d\cdot\log\|(1,z)\|)
=0 
\end{gather*}
under the assumption that the lift $F$ of $f$ is minimal, 
the third term 
is computed as
$2d\cdot T_F=-2d\cdot\int_{\sP^1}\rho(\cS_{\can},\cdot\wedge_{\can}\cS')
\rd(f^*\delta_{\cS_{\can}})(\cS')$.

It remains to give a case by case verification that the right hand side
of \eqref{eq:division}
always equals $\rho(\cS,f(\cS)\wedge_{\can}\cS)$. Let us denote by
\begin{gather*}
 r:=r_{\sP^1,\Gamma_{\{a,\infty,\cS_{\can}\}}}\quad\text{and}\quad
s:=r_{\sP^1,[a,\infty]},
\end{gather*}
so that $s(\cS_{\can})=a\wedge_{\can}\infty$. Recall that
\begin{gather*}
 \cS=h_{a,b}(\cS_{\can})\in[a,\infty]\subset
[a,\infty]\cup[s(\cS_{\can}),\cS_{\can}]=\Gamma_{\{a,\infty,\cS_{\can}\}}.
\end{gather*}
{\bfseries (i)} If $r(f(\cS))\in[a,\infty]$ and $s(\cS_{\can})\in[\cS,\infty]$,
then 
\begin{align*}
&(\text{the right hand side of }\eqref{eq:division})=\\
&\begin{cases}
\max\bigl\{\rho(s(\cS_{\can}),\cS_{\can})-\rho(s(\cS_{\can}),\cS_{\can}),
\rho(\cS,\cS_{\can})-\rho(r(f(\cS)),\cS_{\can})\bigr\}\\
\quad=\max\{0,-\rho(\cS,r(f(\cS)))\}=0\\
\quad=\rho(\cS,\cS)=\rho(\cS,f(\cS)\wedge_{\can}\cS)
\quad\text{if }r(f(\cS))\in[a,\cS],
\\
\max\bigl\{\rho(s(\cS_{\can}),\cS_{\can})-\rho(s(\cS_{\can}),\cS_{\can}),
\rho(\cS,\cS_{\can})-\rho(r(f(\cS)),\cS_{\can})\bigr\}\\
\quad=\rho(\cS,r(f(\cS)))=\rho(\cS,f(\cS)\wedge_{\can}\cS)
\quad\text{if }r(f(\cS))\in[\cS,s(\cS_{\can})],
\\
\max\bigl\{\rho(s(\cS_{\can}),\cS_{\can})-\rho(r(f(\cS)),\cS_{\can}),
\rho(\cS,\cS_{\can})-\rho(s(\cS_{\can}),\cS_{\can})\bigr\}\\
\quad=\max\{-\rho(s(\cS_{\can}),r(f(\cS))),\rho(\cS,s(\cS_{\can}))\}
=\rho(\cS,s(\cS_{\can}))\\
\quad=\rho(\cS,f(\cS)\wedge_{\can}\cS)
\quad\text{if }r(f(\cS))\in[s(\cS_{\can}),\infty].
\end{cases}
\end{align*}
{\bfseries (ii)} If $r(f(\cS))\in[a,\infty]$ and $s(\cS_{\can})\in[a,\cS]$, then 
\begin{align*}
&(\text{the right hand side of }\eqref{eq:division})=\\
&\begin{cases}
\max\bigl\{\rho(\cS,\cS_{\can})-\rho(s(\cS_{\can}),\cS_{\can}),
\rho(s(\cS_{\can}),\cS_{\can})-\rho(r(f(\cS)),\cS_{\can})\bigr\}\\
\quad=\max\{\rho(\cS,s(\cS_{\can})),-\rho(s(\cS_{\can}),r(f(\cS)))\}
=\rho(\cS,s(\cS_{\can}))\\
\quad=\rho(\cS,f(\cS)\wedge_{\can}\cS)
\quad\text{if }r(f(\cS))\in[a,s(\cS_{\can})],
\\
\max\bigl\{\rho(\cS,\cS_{\can})-\rho(r(f(\cS)),\cS_{\can}),
\rho(s(\cS_{\can}),\cS_{\can})-\rho(s(\cS_{\can}),\cS_{\can})\bigr\}\\
\quad=\rho(\cS,r(f(\cS)))=\rho(\cS,f(\cS)\wedge_{\can}\cS)
\quad\text{if }r(f(\cS))\in[s(\cS_{\can}),\cS],
\\
\max\bigl\{\rho(\cS,\cS_{\can})-\rho(r(f(\cS)),\cS_{\can}),
\rho(s(\cS_{\can}),\cS_{\can})-\rho(s(\cS_{\can}),\cS_{\can})\bigr\}\\
\quad=\max\{-\rho(\cS,r(f(\cS))),0\}=0\\
\quad=\rho(\cS,\cS)=\rho(\cS,f(\cS)\wedge_{\can}\cS)\quad\text{if }r(f(\cS))\in[\cS,\infty].
\end{cases}
\end{align*}
{\bfseries (iii)} Finally, if 
$r(f(\cS))\in[s(\cS_{\can}),\cS_{\can}]$, then
no matter whether $s(\cS_{\can})\in[\cS,\infty]$ or 
$s(\cS_{\can})\in[a,\cS]$, we have
\begin{align*}
&(\text{the right hand side of }\eqref{eq:division})\\
=&\max\{\rho(\cS,\cS_{\can})-\rho(r(f(\cS)),\cS_{\can}),
 \rho(s(\cS_{\can}),\cS_{\can})-\rho(r(f(\cS)),\cS_{\can})\}\\
 =&\max\{\rho(\cS,r(f(\cS))),\rho(s(\cS_{\can}),r(f(\cS)))\}\\
 =&\rho(\cS,r(f(\cS)))=\rho(\cS,f(\cS)\wedge_{\can}\cS).
\end{align*}
Hence \eqref{eq:resultant} holds.

{\bfseries (b)} We next establish
the continuity of $\Crucial_f$ on $(\sH^1,\rho)$.
The functions
$\cS\mapsto\rho(\cS,\cS_{\can})$ and
\begin{gather}
\cS\mapsto
T_F(\cS)=-\int_{\sP^1}\rho(\cS_{\can},\cS\wedge_{\can}\cdot)\rd(f^*\delta_{\cS_{\can}})
=\int_{\sP^1}\log[\cS,\cdot]_{\can}\rd(f^*\delta_{\cS_{\can}})
\label{eq:potentialminimal}
\end{gather}
are continuous on $(\sH^1,\rho)$. We claim that the function
\begin{gather*}
 \cS\mapsto-\rho(\cS_{\can},f(\cS)\wedge_{\can}\cS)=\log[f(\cS),\cS]_{\can} 
\end{gather*}
is also continuous on $(\sH^1,\rho)$; indeed, by Theorem \ref{th:local}, 
we have the Lipschitz continuously of the action of $f$
on $(\sH^1,\rho)$ and moreover,
for every $\cS'\in\sH^1$ and every 
$\overrightarrow{v}=\overrightarrow{\cS'\cS''}\in T_{\cS'}\sP^1$,  
diminishing $[\cS',\cS'']$ if necessary, 
we have
\begin{multline}
\cS\mapsto f(\cS)\wedge_{\can}\cS\\
\begin{cases}
\equiv f(\cS')\wedge_{\can}\cS'
& \text{if }f(\cS')\not\in[\cS',\cS_{\can}]
\text{ and }\#(U_{\overrightarrow{v}}\cap\{\cS_{\can},f(\cS')\})\in\{0,2\},
\\
=\cS & \text{if }f(\cS')\not\in[\cS',\cS_{\can}]
\text{ and }\#(U_{\overrightarrow{v}}\cap\{\cS_{\can},f(\cS')\})=1,
\\
\equiv f(\cS') & \text{if }f(\cS')\in(\cS',\cS_{\can}]
\text{ and }U_{f_*(\overrightarrow{v})}\cap\{\cS',\cS_{\can}\}=\emptyset,
\\
=f(\cS)
& \text{if }f(\cS')\in(\cS',\cS_{\can}]\text{ and }
U_{f_*(\overrightarrow{v})}\cap\{\cS',\cS_{\can}\}\neq\emptyset,
\\
=\cS & \text{if }f(\cS')=\cS',f_*(\overrightarrow{v})=\overrightarrow{v},
 \text{ and }\cS_{\can}\not\in U_{\overrightarrow{v}},
\\
=f(\cS)\in[\cS'\cS_{\can})
& \text{if }f(\cS')=\cS',f_*(\overrightarrow{v})=\overrightarrow{v},
 \text{ and }\cS_{\can}\in U_{\overrightarrow{v}},
\\
\equiv\cS' & \text{if }f(\cS')=\cS',
f_*(\overrightarrow{v})\neq\overrightarrow{v},\text{ and }
\cS_{\can}\not\in U_{\overrightarrow{v}}\cup U_{f_*(\overrightarrow{v})},
\\
=\cS & \text{if }f(\cS')=\cS',
f_*(\overrightarrow{v})\neq\overrightarrow{v},\text{ and }
\cS_{\can}\in U_{\overrightarrow{v}},
\\
=f(\cS) & \text{if }f(\cS')=\cS',
f_*(\overrightarrow{v})\neq\overrightarrow{v},\text{ and }
\cS_{\can}\in U_{f_*(\overrightarrow{v})}
\end{cases}\label{eq:imagewedge}
\end{multline}
on $[\cS',\cS'']$ (in fact in \eqref{eq:imagewedge},
we can replace $\cS_{\can}$ and $\wedge_{\can}$
with $\cS_0$ and $\wedge_{\cS_0}$, respectively,
for any $\cS_0\in\sH^1$), so the claim holds. 
Once this claim is at our disposal, also using the equality
$\cS\mapsto\rho(\cS,f(\cS)\wedge_{\can}\cS)
 =\rho(\cS,\cS_{\can})-\rho(\cS_{\can},f(\cS)\wedge_{\can}\cS)$
on $\sH^1$,
we have the continuity of $\Crucial_f$ on $(\sH^1,\rho)$.

{\bfseries (c)} It remains to show \eqref{eq:conjugate}.
We begin by demonstrating that \eqref{eq:conjugate}
holds for every $(\cS,\cS_0)\in\sH^1\times\sH^1_{\mathrm{II}}$,
where 
\begin{gather*}
 \sH^1_{\mathrm{II}}:=\{\text{type II points in }\sP^1\}.
\end{gather*}
For any $h,\ell\in\PGL(2,K)$, by \eqref{eq:resultant}, we have 
$\Crucial_f((h\circ \ell)(\cS_{\can}))
 =\Crucial_{h^{-1}\circ f\circ h}(\ell(\cS_{\can}))+\Crucial_f(h(\cS_{\can}))$.
Hence for any $h\in\PGL(2,K)$, by
the transitivity of the action of $\PGL(2,K)$ on $\sH^1_{\rm II}$, we first have 
$\Crucial_f\circ h-\Crucial_f(h(\cS_{\can}))
=\Crucial_{h^{-1}\circ f\circ h}$ on $\sH^1_{\rm II}$.
Then since $h^{-1}(\sH^1_{\rm II})=\sH^1_{\rm II}$, we have
\begin{gather}
\Crucial_f-\Crucial_f(h(\cS_{\can}))
=\Crucial_{h^{-1}\circ f\circ h}\circ h^{-1}
\tag{\ref{eq:conjugate}$'$}\label{eq:conjugatebalc}
\end{gather}
on $\sH^1_{\rm II}$, and in turn on $\sH^1$
by the density of $\sH^1_{\rm II}$ in $(\sH^1,\rho)$ and the continuity 
on $(\sH^1,\rho)$ of both sides of \eqref{eq:conjugatebalc}. 

For every $(\cS,\cS_0)\in\sH^1\times\sH^1_{\rm II}$,
by the transitivity of the action of $\PGL(2,K)$ on $\sH^1_{\rm II}$,
there exists $h\in\PGL(2,K)$ such that $\cS_0=h(\cS_{\can})$.
We also recall that $h$ acts on $(\sH^1,\rho)$ isometrically.
By \eqref{eq:conjugatebalc} and the definition
of $\Crucial_{h^{-1}\circ f\circ h}$ in Definition \ref{th:crucialdef},
we compute
\begin{align*}
&2(d-1)(\Crucial_f(\cS)-\Crucial_f(\cS_0))\\
=&2(d-1)\cdot\Crucial_{h^{-1}\circ f\circ h}(h^{-1}(\cS))\\
=&(d-1)\cdot\rho(h^{-1}(\cS),\cS_{\can})
 +2\rho\bigl(h^{-1}(\cS),(h^{-1}\circ f\circ h)(h^{-1}(\cS))\wedge_{\can}h^{-1}(\cS)\bigr)\\
 &-2\int_{\sP^1}\rho\bigl(\cS_{\can},h^{-1}(\cS)\wedge_{\can}\cdot\bigr)\rd((h^{-1}\circ f\circ h)^*\delta_{\cS_{\can}})\\
=&(d-1)\cdot\rho(\cS,\cS_0)+2\rho(\cS,f(\cS)\wedge_{\cS_0}\cS)\\
 &-2\int_{\sP^1}(h^{-1})^*\bigl(\rho(\cS_{\can},h^{-1}(\cS)\wedge_{\can}\cdot)\bigr)
\rd((h^{-1})^*(h^*f^*\delta_{\cS_0}))\\
=&(d-1)\cdot\rho(\cS,\cS_0)+2\rho(\cS,f(\cS)\wedge_{\cS_0}\cS)
 -2\int_{\sP^1}\rho(\cS_0,\cS\wedge_{\cS_0}\cdot)\rd(f^*\delta_{\cS_0}),
\end{align*}
that is, the equality \eqref{eq:conjugate} holds for every $(\cS,\cS_0)\in\sH^1\times\sH^1_{\rm II}$.

{\bfseries (d)}
Let us complete the proof of \eqref{eq:conjugate}
by a continuity argument.
Fixing $\cS\in\sH^1$, 
the function
$\cS_0\mapsto\rho(\cS,\cS_0)$ is continuous on $(\sH^1,\rho)$,
and so is the function $\cS_0\mapsto\rho(\cS,f(\cS)\wedge_{\cS_0}\cS)
=\rho(\cS,f(\cS)\wedge_{\cS}\cS_0)$.

For any $\cS,\cS_0\in\sH^1$, we compute as
\begin{align*}
&-\int_{\sP^1}\rho(\cS_0,\cS\wedge_{\cS_0}\cdot)\rd(f^*\delta_{\cS_0})\\
=&\int_{\sP^1}(\log[\cS,\cdot]_{\can}-\log[\cdot,\cS_0]_{\can})\rd(f^*\delta_{\cS_0})
+d\cdot(\log[\cS_0,\cS_0]_{\can}-\log[\cS,\cS_0]_{\can})\\
=&
\int_{\sP^1}
(\log[\cS,\cdot]_{\can}-\log[\cdot,\cS_0]_{\can})
\rd(f^*\Delta\log[\cdot,\cS_0]_{\can})\\
&\quad
+T_F(\cS)-T_F(\cS_0)
-d\cdot(\rho(\cS_0,\cS_{\can})+\log[\cS,\cS_0]_{\can})\\
=&
\int_{\sP^1}f^*(\log[\cdot,\cS_0]_{\can})
\Delta(\log[\cdot,\cS]_{\can}-\log[\cdot,\cS_0]_{\can})\\
&\quad
+T_F(\cS)-T_F(\cS_0)
-d\cdot\bigl(\rho(\cS_0,\cS_{\can})+\log[\cS,\cS_0]_{\can}\bigr)\\
=&
\log[f(\cS),\cS_0]_{\can}-\log[f(\cS_0),\cS_0]_{\can}\\
&\quad
+T_F(\cS)-T_F(\cS_0)-d\cdot\bigl(\rho(\cS_0,\cS_{\can})+\log[\cS,\cS_0]_{\can}\bigr),
\end{align*}
where the first equality follows from \eqref{eq:Gromovevery} and
\eqref{eq:Hsiaevery}, and so does the second
from \eqref{eq:fundamental},
\eqref{eq:potentialminimal}, and 
\eqref{eq:GromovGauss}.
The third equality follows from \eqref{eq:functorial} and 
Green's formula (see e.g.\ \cite[Proposition 3.2]{BR10}), 
and so does the final equality from \eqref{eq:GromovGauss} and 
\eqref{eq:fundamental}.

In the above part {\bfseries (b)}, we have seen that 
$\cS_0\mapsto\log[f(\cS_0),\cS_0]_{\can}$ is continuous
on $(\sH^1,\rho)$. Hence fixing $\cS\in\sH^1$, 
the function
$\cS_0\mapsto-\int_{\sP^1}\rho(\cS_0,\cS\wedge_{\cS_0}\cdot)\rd(f^*\delta_{\cS_0})$
is also continuous on $(\sH^1,\rho)$.
Now fixing the first variable $\cS\in\sH^1$,
both sides of \eqref{eq:conjugate} are continuous on 
$(\sH^1,\rho)$ as functions of the second variable $\cS_0$, 
so by the density of $\sH^1_{\rm II}$ in $(\sH^1,\rho)$,
the equality \eqref{eq:conjugate} holds on the whole $\sH^1\times\sH^1$. \qed

\section{Proof of Theorem \ref{th:convex}}\label{sec:convex}
Let $f\in K(z)$ be of degree $d>1$.
For every $\cS_0\in\sH^1$ and every direction
$\overrightarrow{v}=\overrightarrow{\cS_0\cS'}\in T_{\cS_0}\sP^1$,
diminishing $[\cS_0,\cS']$ if necessary, we have
\begin{multline}
\cS\mapsto\rho(\cS,f(\cS)\wedge_{\cS_0}\cS)
=\begin{cases}
\rho(\cS,\cS)\equiv 0 & \text{if }f(\cS_0)\neq\cS_0\text{ and }
\overrightarrow{\cS_0f(\cS_0)}=\overrightarrow{v},\\
\rho(\cS,\cS_0) & \text{if }f(\cS_0)\neq\cS_0\text{ and }
\overrightarrow{\cS_0f(\cS_0)}\neq\overrightarrow{v},\\
\rho(\cS,\cS)\equiv 0 & \text{if }f(\cS_0)=\cS_0
\text{ and }f_*(\overrightarrow{v})=\overrightarrow{v},\\
\rho(\cS,\cS_0) & \text{if }f(\cS_0)=\cS_0
\text{ and }f_*(\overrightarrow{v})\neq\overrightarrow{v}
\end{cases}\label{eq:fixed}
\end{multline}
on $[\cS_0,\cS']$,
by Theorem \ref{th:local} 
(cf.\ the 1st, 2nd, 5th, and 6th
cases in \eqref{eq:imagewedge}).
On the other hand, on $[\cS_0,\cS']$,
\begin{multline}
\cS\mapsto\int_{\sP^1}\rho(\cS_0,\cS\wedge_{\cS_0}\cdot)\rd(f^*\delta_{\cS_0})
\Bigl(=(f^*\delta_{\cS_0})(U_{\overrightarrow{v}})\cdot\rho(\cS_0,\cS)\Bigr)\\
\begin{cases}
=(f^*\delta_{\cS_0})(U_{\overrightarrow{v}})\cdot\rho(\cS_0,\cS)
&\text{if }\overrightarrow{v}\in T_{\cS_0}(\Gamma_{\{\cS_0\}\cup f^{-1}(\cS_0)}),\\
\equiv(f^*\delta_{\cS_0})(\sP^1)\cdot\rho(\cS_0,\cS_0)
=0 & \text{otherwise}.
\end{cases}\label{eq:outsidetree}
\end{multline}

{\bfseries (a)}
By \eqref{eq:fixed} and \eqref{eq:outsidetree},
for every $\cS_0\in\sH^1$ and
every $\overrightarrow{v}=\overrightarrow{\cS_0\cS'}\in T_{\cS_0}\sP^1$,
diminishing $[\cS_0,\cS']$ if necessary,
the functions $\cS\mapsto\rho(\cS,f(\cS)\wedge_{\cS_0}\cS)$ and
$\cS\mapsto\int_{\sP^1}\rho(\cS_0,\cS\wedge_{\cS_0}\cdot)
\rd(f^*\delta_{\cS_0})$ are affine on $([\cS_0,\cS'],\rho)$,
and we have the desired ranges for
$\rd_{\overrightarrow{v}}(\cS\mapsto\rho(\cS,f(\cS)\wedge_{\cS_0}\cS))$
and $\rd_{\overrightarrow{v}}(\cS\mapsto\int_{\sP^1}\rho(\cS_0,\cS\wedge_{\cS_0}\cdot)\rd(f^*\delta_{\cS_0}))$. 

It now follows also from \eqref{eq:conjugate} that
in addition to being continuous,
$\Crucial_f$ is piecewise affine on $(\sH^1,\rho)$, 
and for every $\cS_0\in\sH^1$ and every $\overrightarrow{v}\in T_{\cS_0}\sP^1$,
we also have the desired range \eqref{eq:rangecrucial} 
for $\rd_{\overrightarrow{v}}\Crucial_f$.

{\bfseries (b)} 
Let us see that $\Crucial_f$ is convex on $(\sH^1,\rho)$.
Let $\cS_0\in\sH^1$, and fix {\itshape distinct} directions
$\overrightarrow{v_1},\overrightarrow{v_2}\in T_{\cS_0}\sP^1$,
so that $U_{\overrightarrow{v_1}}\cap U_{\overrightarrow{v_2}}=\emptyset$.

If $f(\cS_0)\neq\cS_0$, then we also have 
$\#\{i\in\{1,2\}:\overrightarrow{\cS_0f(\cS_0)}=\overrightarrow{v_i}\}\le 1$.
Hence by \eqref{eq:fixed} and \eqref{eq:outsidetree},
we respectively have 
\begin{gather*}
 (\rd_{\overrightarrow{v_1}}+\rd_{\overrightarrow{v_2}})
 (\cS\mapsto\rho(\cS,f(\cS)\wedge_{\cS_0}\cS))\ge 1
\end{gather*}
and 
\begin{multline*}
 (\rd_{\overrightarrow{v_1}}+\rd_{\overrightarrow{v_2}})
 \Bigl(\cS\mapsto\int_{\sP^1}\rho(\cS_0,\cS\wedge_{\cS_0}\cdot)\rd(f^*\delta_{\cS_0})\Bigr)\\
 =(f^*\delta_{\cS_0})(U_{\overrightarrow{v_1}}\cup U_{\overrightarrow{v_2}})
 \le(f^*\delta_{\cS_0})(\sP^1)=d,
\end{multline*}
so that also using \eqref{eq:conjugate}, we have
\begin{gather}
 (\rd_{\overrightarrow{v_1}}+\rd_{\overrightarrow{v_2}})
 \Crucial_f\ge 2\cdot\frac{1}{2}+\frac{1-d}{d-1}=0.\label{eq:convexnonfixed}
\end{gather}
Alternatively, if $f(\cS_0)=\cS_0$, then by \eqref{eq:fixed} and \eqref{eq:outsidetree},
we respectively have 
\begin{gather*}
 (\rd_{\overrightarrow{v_1}}+\rd_{\overrightarrow{v_2}})
 (\cS\mapsto\rho(\cS,f(\cS)\wedge_{\cS_0}\cS))\ge 0
\end{gather*}
and
\begin{multline*}
 (\rd_{\overrightarrow{v_1}}+\rd_{\overrightarrow{v_2}})
 \Bigl(\cS\mapsto\int_{\sP^1}\rho(\cS_0,\cS\wedge_{\cS_0}\cdot)\rd(f^*\delta_{\cS_0})\Bigr)\\
=(f^*\delta_{\cS_0})(U_{\overrightarrow{v_1}}\cup U_{\overrightarrow{v_2}})
\le(f^*\delta_{\cS_0})(\sP^1)-\deg_{\cS_0}(f)=d-\deg_{\cS_0}(f), 
\end{multline*} 
so that also using \eqref{eq:conjugate}, we have
 \begin{gather}
 \left(\rd_{\overrightarrow{v_1}}+\rd_{\overrightarrow{v_2}}\right)\Crucial_f
 \ge 2\cdot\frac{1}{2}+\frac{0-(d-\deg_{\cS_0}(f))}{d-1}
 =\frac{\deg_{\cS_0}(f)-1}{d-1}\ge 0.\label{eq:convexfixed}
 \end{gather}
By \eqref{eq:convexnonfixed} and \eqref{eq:convexfixed},
in addition to being piecewise affine,
$\Crucial_f$ is 
convex on $(\sH^1,\rho)$.
 
{\bfseries (c)} 
We next establish that $\Crucial_f$ is locally affine at
every point of one of types I, III, and IV. 

For every type III point $\cS_0\in\sH^1$,
we can write $T_{\cS_0}\sP^1=\{\overrightarrow{v_1},\overrightarrow{v_2}\}$.
If $f(\cS_0)\neq \cS_0$, then we have
$\#\{i\in\{1,2\}:\overrightarrow{\cS_0f(\cS_0)}=\overrightarrow{v_i}\}=1$
as well as $U_{\overrightarrow{v_1}}\cup U_{\overrightarrow{v_2}}
=\sP^1\setminus\{\cS_0\}\supset f^{-1}(\cS_0)$, 
so by \eqref{eq:fixed} and \eqref{eq:outsidetree}, we have
the {\itshape equalities}
$(\rd_{\overrightarrow{v_1}}+\rd_{\overrightarrow{v_2}})
(\cS\mapsto\rho(\cS,f(\cS)\wedge_{\cS_0}\cS))=1$
and 
$(\rd_{\overrightarrow{v_1}}+\rd_{\overrightarrow{v_2}})
 (\cS\mapsto\int_{\sP^1}\rho(\cS_0,\cS\wedge_{\cS_0}\cdot)\rd(f^*\delta_{\cS_0}))
 =(f^*\delta_{\cS_0})(U_{\overrightarrow{v_1}}\cup U_{\overrightarrow{v_2}})=d$,
 respectively.
Thus the equality holds in \eqref{eq:convexnonfixed} in this case.
If $f(\cS_0)=\cS_0$, then by Theorem \ref{th:indifferent},
we have not only 
$\deg_{\cS_0}(f)=1$ but also $f_*(\overrightarrow{v_i})=\overrightarrow{v_i}$
 for each $i\in\{1,2\}$, as well as
 $U_{\overrightarrow{v_1}}\cup U_{\overrightarrow{v_2}}
 =\sP^1\setminus\{\cS_0\}$. Hence
 by \eqref{eq:fixed} and \eqref{eq:outsidetree}, we have
 the {\itshape equalities} 
 $(\rd_{\overrightarrow{v_1}}+\rd_{\overrightarrow{v_2}})(\cS\mapsto\rho(\cS,f(\cS)\wedge_{\cS_0}\cS))=0$ and 
$(\rd_{\overrightarrow{v_1}}+\rd_{\overrightarrow{v_2}})
 (\cS\mapsto\int_{\sP^1}\rho(\cS_0,\cS\wedge_{\cS_0}\cdot)\rd(f^*\delta_{\cS_0}))
 =(f^*\delta_{\cS_0})(U_{\overrightarrow{v_1}}\cup U_{\overrightarrow{v_2}})
 =d-\deg_{\cS_0}(f)$, respectively.
Thus the equalities also hold throughout in \eqref{eq:convexfixed} in this case.
Hence the function
 $\Crucial_f$ is also locally affine on $(\sP^1,\tilde{\rho})$ at 
 every type III point. 

On the other hand,
the convexity and the piecewise affineness of $\Crucial_f$ on $(\sH^1,\rho)$,
together with
the finiteness of the range of $\rd_{\overrightarrow{v}}\Crucial_f$
for each $\cS\in\sH^1$
and each $\overrightarrow{v}\in T_{\cS}\sP^1$,
imply that $\Crucial_f$ is locally affine 
on $(\sP^1,\tilde{\rho})$ at every type I point.
For every type IV point $\cS_0\in\sH^1$, since $\#T_{\cS_0}\sP^1=1<2$,
$\Crucial_f$ is trivially locally affine on $(\sP^1,\tilde{\rho})$ at $\cS_0$.

{\bfseries (d)} Let us see \eqref{eq:typeI} and \eqref{eq:typeIpullback}.
Let $a\in\bP^1$, and fix $\cS_0\in\sH^1$. 
For every $\cS'\in f^{-1}(\cS_0)$, we have
\begin{multline*}
\cS\mapsto\cS\wedge_{\cS_0}\cS'
=r_{\sP^1,[\cS_0,\cS']}(\cS)
=\bigl(r_{\Gamma_{\{\cS_0\}\cup f^{-1}(\cS_0)},[\cS_0,\cS']}\circ
r_{\sP^1,\Gamma_{\{\cS_0\}\cup f^{-1}(\cS_0)}}\bigr)(\cS)\\
=\bigl(r_{\sP^1,\Gamma_{\{\cS_0\}\cup f^{-1}(\cS_0)}}(\cS)\bigr)\wedge_{\cS_0}\cS'
\equiv\bigl(r_{\sP^1,\Gamma_{\{\cS_0\}\cup f^{-1}(\cS_0)}}(a)\bigr)\wedge_{\cS_0}\cS'
\end{multline*} 
on $[a,r_{\sP^1,\Gamma_{\{\cS_0\}\cup f^{-1}(\cS_0)}}(a)\bigr]$.
Hence fixing $\cS_0'\in(a,r_{\sP^1,\Gamma_{\{\cS_0\}\cup f^{-1}(\cS_0)}}(a)]$,
we have \eqref{eq:typeIpullback} on $(a,\cS_0']$.
Let us see \eqref{eq:typeI}. 

If $f(a)\neq a$, then diminishing $[a,\cS_0']$ if necessary, 
$\cS\mapsto\rho(\cS,f(\cS)\wedge_{\cS_0}\cS)
=\rho(\cS,\cS_0)-\rho(\cS_0,f(\cS)\wedge_{\cS_0}\cS)
\equiv\rho(\cS,\cS_0)-\rho(\cS_0,f(a)\wedge_{\cS_0}a)$
on $(a,\cS_0']$, 
that is, \eqref{eq:typeI} holds when $f(a)\neq a$.

Suppose next that $f(a)=a$. By Theorem \ref{th:local},
diminishing $[a,\cS_0']$ if necessary, 
$f$ maps $[a,\cS_0']$ homeomorphically onto $[a,f(\cS_0')]$ and,
for any $\cS_1,\cS_2\in(a,\cS_0']$, we have
$\rho(f(\cS_1),f(\cS_2))=(\deg_a f)\cdot\rho(\cS_1,\cS_2)$.
We note that $f(\cS_0')\wedge_a\cS_0'\in(a,f(\cS_0')]$, so
(replacing $\cS_0'$ with the point
$(f|(a,\cS_0'])^{-1}\bigl(f(\cS_0')\wedge_a\cS_0'\bigr)$
if necessary,) we can also assume that 
$f(\cS_0')\in(a,r_{\sP^1,\Gamma_{\{\cS_0\}\cup f^{-1}(\cS_0)}}(a)]$. 
By the strong triangle inequality for $|\cdot|$ and 
the power series expansion of $f$ around $a$,
the chordal derivative $f^\#(a)$ of $f$ at $a$ is computed as
\begin{gather*}
 f^\#(a)=\lim_{\cS\to a}\frac{[f(\cS),a]_{\can}}{[\cS,a]_{\can}}.
\end{gather*}

First, in the case where $\deg_af>1$, i.e., $f^\#(a)=0$, we claim that
if $\cS\in(a,\cS_0']$ is close enough to $a$, then $f(\cS)\in(a,\cS]$; 
for, {\bfseries (Subcase 1)} 
if $\cS_0'\in(a,f(\cS_0')]$, then for every $\cS\in(a,\cS_0']$ 
close enough to $a$, we have $\rho(\cS,\cS_0')\ge\rho(f(\cS_0'),\cS_0')$(, $\{\cS,f(\cS)\}\subset(a,f(\cS_0')]$,) and
$\rho(f(\cS),f(\cS_0'))=(\deg_a f)\cdot\rho(\cS,\cS_0')
\ge (\deg_a f-1)\cdot\rho(\cS,\cS_0')+\rho(\cS_0',f(\cS_0'))\ge\rho(\cS,f(\cS_0'))$,
so that $f(\cS)\in(a,\cS]$. {\bfseries (Subcase 2)} 
If $f(\cS_0')\in(a,\cS_0']$, then 
for any $\cS\in(a,\cS_0']$, we have ($f(\cS)\in(a,f(\cS_0')]$ and)
$\rho(f(\cS),\cS_0')=\rho(f(\cS),f(\cS_0'))+\rho(f(\cS_0'),\cS_0')
\ge(\deg_af)\cdot\rho(\cS,\cS_0')
\ge\rho(\cS,\cS_0')$, so that $f(\cS)\in(a,\cS]$.
Hence the claim holds, so diminishing $[a,\cS_0']$ if necessary,
we have
$\cS\mapsto\rho(\cS,f(\cS)\wedge_{\cS_0}\cS)=\rho(\cS,\cS)\equiv 0
=\log\max\{1,f^\#(a)\}$
on $(a,\cS_0']$ in this case.

Next, in the case where $\deg_af=1$ and $f(\cS_0')\in(a,\cS_0']$, 
 for every $\cS\in(a,\cS_0']$, we have
 ($f(\cS)\in(a,f(\cS_0')]$ and)
$\rho(f(\cS),\cS_0')=\rho(f(\cS),f(\cS_0'))+\rho(f(\cS_0'),\cS_0')
 \ge 1\cdot\rho(\cS,\cS_0')
=\rho(\cS,\cS_0')$,
so that $f(\cS)\in(a,\cS]$. Then 
diminishing $[a,\cS_0']$ so that $\cS_0'\in(a,r_{\sP^1,[a,\cS_0]}(\cS_{\can})]$
if necessary, also by \eqref{eq:GromovGauss}, we have 
\begin{gather*}
\cS\mapsto\rho(f(\cS),\cS)=\rho(f(\cS),\cS_{\can})-\rho(\cS,\cS_{\can})
=-\log\frac{[f(\cS),a]_{\can}}{[\cS,a]_{\can}}
\quad\text{on }(a,\cS_0'],
\end{gather*}
so that $f^\#(a)=\lim_{\cS\to a}([f(\cS),a]_{\can}/[\cS,a]_{\can})\le 1$.
Hence $\cS\mapsto\rho(\cS,f(\cS)\wedge_{\cS_0}\cS)
=\rho(\cS,\cS)\equiv 0=\log\max\{1,f^\#(a)\}$ on $(a,\cS_0']$ in this case.

Finally, in the case where $\deg_af=1$ and $\cS_0'\in(a,f(\cS_0')]$,
for every $\cS\in(a,\cS_0']$, we have ($f(\cS)\in(a,f(\cS_0')]$ and)
$\rho(f(\cS),f(\cS_0'))=(\deg_a f)\cdot\rho(\cS,\cS_0')
 \le 1\cdot\rho(\cS,\cS_0')+\rho(\cS_0',f(\cS_0'))=\rho(\cS,f(\cS_0'))$, 
so that $\cS\in(a,f(\cS)]$. Then 
diminishing $[a,\cS_0']$ so that $f(\cS_0')\in(a,r_{\sP^1,[a,\cS_0]}(\cS_{\can})]$
if necessary, also by \eqref{eq:GromovGauss},
we have not only
\begin{multline*}
\cS\mapsto\rho(\cS,f(\cS)\wedge_{\cS_0}\cS)=\rho(\cS,f(\cS))
=\rho(\cS,\cS_{\can})-\rho(f(\cS),\cS_{\can})\\
=\log\frac{[f(\cS),a]_{\can}}{[\cS,a]_{\can}}
\quad\text{on }(a,\cS_0'],
\end{multline*}
so $f^\#(a)=\lim_{\cS\to a}([f(\cS),a]_{\can}/[\cS,a]_{\can})\ge 1$,
but also
for any $\cS_1,\cS_2\in(a,\cS_0']$,
\begin{multline*}
\biggl|\log\frac{[f(\cS_1),a]_{\can}}{[\cS_1,a]_{\can}}-\log\frac{[f(\cS_2),a]_{\can}}{[\cS_2,a]_{\can}}\biggr|
=\biggl|\log\frac{[f(\cS_1),a]_{\can}}{[f(\cS_2),a]_{\can}}-\log\frac{[\cS_2,a]_{\can}}{[\cS_1,a]_{\can}}\biggr|\\
 =|\rho(f(\cS_1),f(\cS_2))-\rho(\cS_1,\cS_2)|=|1\cdot\rho(\cS_1,\cS_2)-\rho(\cS_1,\cS_2)|=0.
\end{multline*}
Hence $\cS\mapsto\rho(\cS,f(\cS)\wedge_{\cS_0}\cS)=
\log([f(\cS),a]_{\can}/[\cS,a]_{\can})\equiv\log(f^\#(a))=\log\max\{1,f^\#(a)\}$
on $(a,\cS_0']$ in this final case.

{\bfseries (e)} It remains to see \eqref{eq:slopeI}.
Once \eqref{eq:typeI} and \eqref{eq:typeIpullback} are at our disposal,
we have \eqref{eq:slopeI} for every type I point $\cS'=a\in\bP^1$ and every $\cS_0\in\sP^1\setminus\{a\}$ by \eqref{eq:conjugate}. 

Fix a type IV point $\cS'\in\sP^1$. For every $\cS_0\in\sP^1\setminus\{\cS'\}$,
set $\overrightarrow{v}:=\overrightarrow{\cS'\cS_0}$, so that
$T_{\cS'}\sP^1=\{\overrightarrow{v}\}$. 
If $f(\cS')\neq\cS'$, then
$\emptyset\neq f^{-1}(\cS')\subset\sP^1\setminus\{\cS'\}$, 
so that 
$\overrightarrow{\cS'f(\cS')}=\overrightarrow{v}\in T_{\cS'}\sP^1
=T_{\cS'}(\Gamma_{\cS'\cup f^{-1}(\cS')})$ and $(f^*\delta_{\cS'})(U_{\overrightarrow{v}})=(f^*\delta_{\cS'})(\sP^1\setminus\{\cS'\})=d$. 
Alternatively, if $f(\cS')=\cS'$, then
$\deg_{\cS'}f=1(<d)$ by Theorem \ref{th:indifferent}, 
so that
$f_*(\overrightarrow{v})=\overrightarrow{v}\in T_{\cS'}\sP^1
=T_{\cS'}(\Gamma_{\cS'\cup f^{-1}(\cS')})$ and 
$(f^*\delta_{\cS'})(U_{\overrightarrow{v}})
=(f^*\delta_{\cS'})(\sP^1\setminus\{\cS'\})=d-1$.
Hence by \eqref{eq:conjugate}, \eqref{eq:fixed}, and \eqref{eq:outsidetree},
we can fix $\cS_0'\in(\cS',\cS_0]$ so close to $\cS'$ that
\begin{multline*}
\cS\mapsto\Crucial_f(\cS)-\Crucial_f(\cS')\\
=\Biggl(\frac{1}{2}+\frac{0-\begin{cases}
		     d & \text{if }f(\cS')\neq\cS'\\
		     d-1 & \text{if }f(\cS')=\cS'
		    \end{cases}}{d-1}\Biggr)\cdot\rho(\cS,\cS')
\quad\text{on }[\cS',\cS_0'].
\end{multline*}
Hence \eqref{eq:slopeI} also holds for every type IV point $\cS'$ and 
every $\cS_0\in\sP^1\setminus\{\cS'\}$. \qed

\section{Proof of Theorem \ref{th:weight}}
\label{sec:weight}
Let $f\in K(z)$ be of degree $d>1$. 
Let $\Gamma$ be a non-trivial finite subtree in $\sP^1$.
Then for every $\cS_0\in\sH^1$, 
by \eqref{eq:conjugate}, 
\eqref{eq:fundamental}, \eqref{eq:laplacian}, and \eqref{eq:outer},
we have
\begin{align*}
\nu_{f,\Gamma}
=&\Bigl((-\nu_{\Gamma}+(r_{\sP^1,\Gamma})_*\delta_{\cS_0})\\
\notag&\quad+\frac{\Delta_{\Gamma}(\cS\mapsto\rho(\cS,f(\cS)\wedge_{\cS_0}\cS))
+(r_{\sP^1,\Gamma})_*(f^*\delta_{\cS_0}-d\cdot\delta_{\cS_0})}{d-1}\Bigr)+\nu_{\Gamma}\\
=&\frac{\Delta_{\Gamma}(\cS\mapsto\rho(\cS,f(\cS)\wedge_{\cS_0}\cS))
 +(r_{\sP^1,\Gamma})_*(f^*\delta_{\cS_0}-\delta_{\cS_0})}{d-1}\quad\text{on }\Gamma,
\end{align*}
so the equality \eqref{eq:geometric} holds. 

For every $\cS'\in\Gamma\cap\sH^1$ and 
every $\cS_0\in\sH^1\setminus\bigl((r_{\sP^1,\Gamma})^{-1}(\cS')\cup\{f(\cS')\}\bigr)$, 
we have $\cS_0\neq\cS'$,
$\overrightarrow{\cS'\cS_0}\in T_{\cS'}\Gamma$, and
\begin{multline}
(d-1)\cdot\nu_{f,\Gamma}(\{\cS'\})\\
=\sum_{\overrightarrow{v}\in T_{\cS'}\Gamma}\rd_{\overrightarrow{v}}
\bigl(\cS\mapsto\rho(\cS,f(\cS)\wedge_{\cS_0}\cS)\bigr)
+\sum_{\overrightarrow{v}\in(T_{\cS'}\sP^1)\setminus(T_{\cS'}\Gamma)}(f^*\delta_{\cS_0})(U_{\overrightarrow{v}}),\label{eq:alwaysgeneral}
\end{multline}
the second term in the right hand side of which is always in $\bN\cup\{0\}$. 

{\bfseries (a)} For every fixed point $\cS'\in\sH^1$ of $f$,
every $\cS_0\in\sH^1\setminus\{\cS'\}$, and 
every $\overrightarrow{v}=\overrightarrow{\cS'\cS''}\in T_{\cS'}\sP^1$, 
diminishing $[\cS',\cS'']$ if necessary, we have on $[\cS',\cS'']$,
\begin{multline}
\cS\mapsto\rho(\cS,f(\cS)\wedge_{\cS_0}\cS)=\\
 \begin{cases}
 \rho(\cS,f(\cS))=\bigl(m_{\overrightarrow{v}}(f)+1\bigr)\rho(\cS,\cS') & \text{if }
 \overrightarrow{v}\neq\overrightarrow{\cS'\cS_0}\text{ and }
 f_*(\overrightarrow{v})=\overrightarrow{\cS'\cS_0},\\
 \rho(\cS,\cS)\equiv 0 & \text{if }
 \overrightarrow{v}\neq\overrightarrow{\cS'\cS_0}\text{ and }
 f_*(\overrightarrow{v})=\overrightarrow{v},\\
 \rho(\cS,\cS') & \text{if }
 \overrightarrow{v}\neq\overrightarrow{\cS'\cS_0}\text{ and }
 f_*(\overrightarrow{v})\not\in\{\overrightarrow{v},\overrightarrow{\cS'\cS_0}\},\\
 \rho(\cS,f(\cS))=\bigl(m_{\overrightarrow{v}}(f)-1\bigr)\rho(\cS,\cS') &\text{if }
 \overrightarrow{v}=\overrightarrow{\cS'\cS_0}\text{ and }
 f_*(\overrightarrow{v})=\overrightarrow{\cS'\cS_0},\\
 \rho(\cS,\cS)\equiv 0 &\text{if }
 \overrightarrow{v}=\overrightarrow{\cS'\cS_0}\text{ and }
 f_*(\overrightarrow{v})\neq\overrightarrow{\cS'\cS_0}
 \end{cases}
\label{eq:fixedcase}
\end{multline}
by Theorem \ref{th:local}
(cf.\ the final five cases in \eqref{eq:imagewedge}). 

Hence for every fixed point $\cS'\in\Gamma\cap\sH^1$ of $f$ and
every $\cS_0\in\sH^1\setminus(r_{\sP^1,\Gamma})^{-1}(\cS')$, we have
\begin{gather}
 \sum_{\overrightarrow{v}\in T_{\cS'}\Gamma}\rd_{\overrightarrow{v}}
(\cS\mapsto\rho(\cS,f(\cS)\wedge_{\cS_0}\cS))\in\bN\cup\{0\}\label{eq:fixedcasegeneral}
\end{gather}
and compute it as
\begin{align}
\notag&\sum_{\overrightarrow{v}\in T_{\cS'}\Gamma}\rd_{\overrightarrow{v}}
\bigl(\cS\mapsto\rho(\cS,f(\cS)\wedge_{\cS_0}\cS)\bigr)\\
\notag=&\sum_{\overrightarrow{v}\in T_{\cS'}\Gamma
\setminus\{\overrightarrow{\cS'\cS_0}\}:
f_*(\overrightarrow{v})=\overrightarrow{\cS'\cS_0}}\bigl(m_{\overrightarrow{v}}(f)+1\bigr)
+\sum_{\overrightarrow{v}\in T_{\cS'}\Gamma
\setminus\{\overrightarrow{\cS'\cS_0}\}:
f_*(\overrightarrow{v})=\overrightarrow{v}}0\\
\notag&+\#\Bigl\{\overrightarrow{v}\in T_{\cS'}\Gamma
\setminus\bigl\{\overrightarrow{\cS'\cS_0}\bigr\}:
f_*(\overrightarrow{v})
\not\in\bigl\{\overrightarrow{v},\overrightarrow{\cS'\cS_0}\bigr\}\Bigr\}\\
\notag&+\begin{cases}
  m_{\overrightarrow{\cS'\cS_0}}(f)-1 &\text{if }f_*(\overrightarrow{\cS'\cS_0})=\overrightarrow{\cS'\cS_0}\\
  0 &\text{if }f_*(\overrightarrow{\cS'\cS_0})\neq\overrightarrow{\cS'\cS_0}
 \end{cases}\\
\notag=&\Biggl(\sum_{\overrightarrow{v}\in T_{\cS'}\Gamma:f_*(\overrightarrow{v})=\overrightarrow{\cS'\cS_0}}m_{\overrightarrow{v}}(f)\Biggr)\\
\notag&+\#\bigl\{\overrightarrow{v}\in T_{\cS'}\Gamma\setminus\{\overrightarrow{\cS'\cS_0}\bigr\}:f_*(\overrightarrow{v})=\overrightarrow{\cS'\cS_0}
\neq\overrightarrow{v}\}\\
\notag&+\#\Bigl\{\overrightarrow{v}\in T_{\cS'}\Gamma
\setminus\bigl\{\overrightarrow{\cS'\cS_0}\bigr\}:
f_*(\overrightarrow{v})
\not\in\bigl\{\overrightarrow{v},\overrightarrow{\cS'\cS_0}\bigr\}\Bigr\}
+
\begin{cases}
-1
&\text{if }f_*(\overrightarrow{\cS'\cS_0})=\overrightarrow{\cS'\cS_0},\\
0&\text{if }f_*(\overrightarrow{\cS'\cS_0})\neq\overrightarrow{\cS'\cS_0}
\end{cases}\\
=&
\Biggl(\sum_{\overrightarrow{v}\in T_{\cS'}\Gamma:f_*(\overrightarrow{v})=\overrightarrow{\cS'\cS_0}}m_{\overrightarrow{v}}(f)\Biggr)
+\#\{\overrightarrow{v}\in T_{\cS'}\Gamma:
f_*(\overrightarrow{v})\neq\overrightarrow{v}\}-1.\label{eq:fixedLaplaciangeneral}
\end{align}
{\bfseries (b)} For every $\cS'\in\sH^1$ not fixed by $f$,
every $\cS_0\in(\sH^1\setminus\{\cS'\})\cap U_{\overrightarrow{f(\cS')\cS'}}$, and
every $\overrightarrow{v}=\overrightarrow{\cS'\cS''}\in T_{\cS'}\sP^1$,
diminishing $[\cS',\cS'']$ if necessary, we have
 \begin{align*}
 &\cS\mapsto\rho(\cS,f(\cS)\wedge_{\cS_0}\cS)\\
 =&
 \begin{cases}
 \rho(\cS,f(\cS')\wedge_{\cS_0}\cS')=\rho(\cS',f(\cS')\wedge_{\cS_0}\cS')+\rho(\cS,\cS')
 \\
 \hspace*{104pt}
 \text{if }
 \overrightarrow{v}\not\in\{\overrightarrow{\cS'\cS_0},\overrightarrow{\cS'f(\cS')}\},\\
 \rho(\cS,\cS)\equiv 0
 \hspace*{50pt}
 \text{if }
 \overrightarrow{v}\in\{\overrightarrow{\cS'\cS_0},\overrightarrow{\cS'f(\cS')}\}
 \text{ and }\overrightarrow{\cS'\cS_0}\neq\overrightarrow{\cS'f(\cS')},\\
 \rho(\cS,f(\cS')\wedge_{\cS_0}\cS')=\rho(\cS',f(\cS')\wedge_{\cS_0}\cS')-\rho(\cS,\cS')\\
 \hspace*{104pt}
 \text{if }
 \overrightarrow{v}=\overrightarrow{\cS'\cS_0}=\overrightarrow{\cS'f(\cS')}
 \end{cases}
\end{align*}
on $[\cS',\cS'']$
(cf.\ the first two cases in \eqref{eq:imagewedge}). 
Hence for every $\cS'\in\Gamma\cap\sH^1$ not fixed by $f$
and every $\cS_0\in(\sH^1\setminus(r_{\sP^1,\Gamma})^{-1}(\cS'))
\cap U_{\overrightarrow{f(\cS')\cS'}}$, we have
\begin{align}
\notag&\sum_{\overrightarrow{v}\in T_{\cS'}\Gamma}\rd_{\overrightarrow{v}}
\bigl(\cS\mapsto\rho(\cS,f(\cS)\wedge_{\cS_0}\cS)\bigr)\\ 
\notag=&\begin{cases}
\#(T_{\cS'}\Gamma\setminus\{\overrightarrow{\cS'\cS_0}\}) & \text{if }r_{\sP^1,\Gamma}(f(\cS'))=\cS',\\
\#(T_{\cS'}\Gamma\setminus\{\overrightarrow{\cS'\cS_0},\overrightarrow{\cS'f(\cS')}\}) & \text{if }r_{\sP^1,\Gamma}(f(\cS'))\neq\cS'\text{ and }\overrightarrow{\cS'\cS_0}\neq\overrightarrow{\cS'f(\cS')},\\
\#(T_{\cS'}\Gamma\setminus\{\overrightarrow{\cS'\cS_0}\})-1 & \text{if }r_{\sP^1,\Gamma}(f(\cS'))\neq\cS'\text{ and }\overrightarrow{\cS'\cS_0}=\overrightarrow{\cS'f(\cS')}
  \end{cases}\\
=&\begin{cases}
   v_{\Gamma}(\cS')-1\in\bN\cup\{0\} & \text{if }r_{\sP^1,\Gamma}(f(\cS'))=\cS',\\
   v_{\Gamma}(\cS')-2\in\bN\cup\{0,-1\} & \text{if }r_{\sP^1,\Gamma}(f(\cS'))\neq\cS'.
  \end{cases}\label{eq:nonfixedgeneral}
\end{align}

\begin{proof}[Proof of Theorem $\ref{th:weight}(i)$]
For every $\cS_0\in\sH^1$, we have already seen \eqref{eq:geometric}.

For every $\cS_0\in\sH^1$ and every $\overrightarrow{v}\in T_{\cS_0}\sP^1$,
by \eqref{eq:conjugate} and \eqref{eq:outsidetree}, we have
\begin{gather*}
 \rd_{\overrightarrow{v}}\Crucial_f=\frac{1}{2}+
\frac{\rd_{\overrightarrow{v}}(\cS\mapsto\rho(\cS,f(\cS)\wedge_{\cS_0}\cS))-(f^*\delta_{\cS_0})(U_{\overrightarrow{v}})}{d-1}.
\end{gather*}
Moreover, for every $\cS_0\in\Gamma\cap\sH^1$ and
every direction $\overrightarrow{v}\in T_{\cS_0}\Gamma$, 
by \eqref{eq:geometric} and $\cS_0\not\in U_{\overrightarrow{v}}$, we have
\begin{align*}
&(d-1)\cdot\bigl((\iota_{\Gamma,\sP^1})_*\nu_{f,\Gamma}\bigr)(U_{\overrightarrow{v}})\\
=&\bigl(\Delta_{\Gamma}(\cS\mapsto\rho(\cS,f(\cS)\wedge_{\cS_0}\cS))\bigr)(\Gamma\cap U_{\overrightarrow{v}})+(f^*\delta_{\cS_0}-\delta_{\cS_0})(U_{\overrightarrow{v}})\\
=&\bigl(\Delta_{\Gamma}(\cS\mapsto\rho(\cS,f(\cS)\wedge_{\cS_0}\cS))\bigr)(\Gamma\cap U_{\overrightarrow{v}})+(f^*\delta_{\cS_0})(U_{\overrightarrow{v}}),
\end{align*}
and setting $\Gamma':=\Gamma\cap(U_{\overrightarrow{v}}\cup\{\cS_0\})$,
by \eqref{eq:outer}, we also have
\begin{align*}
 &\bigl(\Delta_{\Gamma}(\cS\mapsto\rho(\cS,f(\cS)\wedge_{\cS_0}\cS))\bigr)(\Gamma\cap U_{\overrightarrow{v}})\\
=&\bigl(\Delta_{\Gamma'}(\cS\mapsto\rho(\cS,f(\cS)\wedge_{\cS_0}\cS))\bigr)
(\Gamma')-\rd_{\overrightarrow{v}}(\cS\mapsto\rho(\cS,f(\cS)\wedge_{\cS_0}\cS))
\\
=&0-\rd_{\overrightarrow{v}}(\cS\mapsto\rho(\cS,f(\cS)\wedge_{\cS_0}\cS))
=-\rd_{\overrightarrow{v}}(\cS\mapsto\rho(\cS,f(\cS)\wedge_{\cS_0}\cS)).
\end{align*}
Hence \eqref{eq:slopeinside} also holds.
\end{proof}

\begin{proof}[Proof of Theorem $\ref{th:weight}(ii)$]
For every $\cS'\in\Gamma\cap\sH^1$, 
(fixing $\cS_0\in\sH^1\setminus((r_{\sP^1,\Gamma})^{-1}(\cS')\cup\{f(\cS')\})$,)
by \eqref{eq:alwaysgeneral}, \eqref{eq:fixedcasegeneral},
and \eqref{eq:nonfixedgeneral},
we always have 
$(d-1)\cdot\nu_{f,\Gamma}(\{\cS'\})\in\bN\cup\{0,-1\}$,
which is also the case for every $\cS'\in\Gamma\cap\bP^1$ by \eqref{eq:LaplacianI}. 
By the convexity of $\Crucial_f$ on $(\sH^1,\rho)$, 
\eqref{eq:rangecrucial}, and \eqref{eq:LaplacianI}, 
$\nu_{f,\Gamma}$ is supported on (a finite subset)
in $\Gamma\setminus\{\text{type I or IV points fixed by }f\}$. 

Next, fix $\cS'\in\Gamma\setminus\{\text{type I or IV points not fixed by }f\}$.
If $(f(\cS')=)\cS'\in\bP^1$, then
$\nu_{f,\Gamma}(\{\cS'\})=0\ge 0$. 
If $f(\cS')=\cS'\in\sH^1$, then 
$\nu_{f,\Gamma}(\{\cS'\})\ge 0$ by \eqref{eq:alwaysgeneral} and
\eqref{eq:fixedcasegeneral} (fixing $\cS_0\in\sH^1\setminus((r_{\sP^1,\Gamma})^{-1}(\cS')\cup\{f(\cS')\})$).

Suppose now that $f(\cS')\neq\cS'\in\sH^1$.
{\bfseries (Subcase a)} Either if $\#T_{\cS'}\Gamma\ge 2$ or if
$\#T_{\cS'}\Gamma=1$ but $r_{\sP^1,\Gamma}(f(\cS'))=\cS'$,
then $\nu_{f,\Gamma}(\{\cS'\})\ge 0$ 
by \eqref{eq:alwaysgeneral} and \eqref{eq:nonfixedgeneral}
(fixing $\cS_0\in(\sH^1\setminus(r_{\sP^1,\Gamma})^{-1}(\cS'))
\cap U_{\overrightarrow{f(\cS')\cS'}}$).
{\bfseries (Subcase b)} If $\#T_{\cS'}\Gamma=1$,
$r_{\sP^1,\Gamma}(f(\cS'))\neq\cS'$, but
$f^{-1}\bigl(U_{\overrightarrow{\cS'f(\cS')}}
\cap U_{\overrightarrow{f(\cS')\cS'}}
\bigr)\not\subset 
U_{\overrightarrow{\cS'f(\cS')}}$, then 
there are such $\cS_0
\in(\sH^1\setminus(r_{\sP^1,\Gamma})^{-1}(\cS'))\cap U_{\overrightarrow{f(\cS')\cS'}}
(=U_{\overrightarrow{\cS'f(\cS')}}
\cap U_{\overrightarrow{f(\cS')\cS'}})$ and $\overrightarrow{w}
\in(T_{\cS'}\sP^1)\setminus\bigl\{\overrightarrow{\cS'f(\cS')}\bigr\}
(=(T_{\cS'}\sP^1)\setminus(T_{\cS'}\Gamma))$
that $(f^*\delta_{\cS_0})(U_{\overrightarrow{w}})\in\bN$.
Then we have $\nu_{f,\Gamma}(\{\cS'\})\ge (-1)+1=0$
by \eqref{eq:alwaysgeneral} and \eqref{eq:nonfixedgeneral}.
\end{proof}

\begin{proof}[Proof of Theorem $\ref{th:weight}(iii)$]
Suppose also that no end points of $\Gamma$ are type III and
that $\nu_{f,\Gamma}\ge 0$ on $\Gamma$ 
(so $\nu_{f,\Gamma}$ is a probability measure on $\Gamma$). Then
by \eqref{eq:LaplacianI}, we have 
$\BC_{\Gamma}(\nu_{f,\Gamma})\subset\sH^1$, so
by \eqref{eq:slopeinside}, 
the first equality in \eqref{eq:mrlessential} holds. 

By \eqref{eq:slopeI} and 
the locally affine continuity of $\Crucial_f$ on $(\sP^1,\tilde{\rho})$
at every type III point, 
$\Gamma':=(\Crucial_f|(\Gamma\cap\sH^1))^{-1}(\min_{\Gamma\cap\sH^1}\Crucial_f)$
is a finite subtree 
in $\Gamma$ 
all end points of which are type II, and 
by \eqref{eq:slopeinside}, for every $\cS\in\Gamma'$ and every
$\overrightarrow{v}\in T_{\cS}\Gamma'$, we have
$((\iota_{\Gamma,\sP^1})_*\nu_{f,\Gamma})(U_{\overrightarrow{v}})=1/2$.
Hence, since $U_{\overrightarrow{v}}\cap U_{\overrightarrow{w}}=\emptyset$ 
for any distinct directions 
$\overrightarrow{v},\overrightarrow{w}\in T_{\cS}\Gamma'$
(and $\nu_{f,\Gamma}$ is a probability measure on $\Gamma$),
we have $v_{\Gamma'}\le 2$ on $\Gamma'$, so
there is a (unique) subset $S_{f,\Gamma}$ in $\Gamma$
consisting of at most two type II points such that
$\Gamma'=\Gamma_{S_{f,\Gamma}}$.
Hence the latter equality in \eqref{eq:mrlessential} holds.
By \eqref{eq:rangecrucial},
we have $\#S_{f,\Gamma}=1$ if $d$ is even (i.e., when $d+1$ is odd).
By the convexity of $\Crucial_f$ on $(\sH^1,\rho)$, we also have $\#S_{f,\Gamma}=1$
if $\min_{\Gamma\cap\sH}\Crucial_f>\min_{\sH^1}\Crucial_f$.

Let us see the remaining \eqref{eq:interioressential}.
For any $\cS\in S_{f,\Gamma}$, 
if $\nu_{\Gamma_{\supp(\nu_{f,\Gamma})}}(\{\cS\})=0$, or equivalently,
if $T_{\cS}\Gamma_{\supp(\nu_{f,\Gamma})}$ consists of 
exactly two directions $\overrightarrow{v_1},\overrightarrow{v_2}$, then
$\rd_{\overrightarrow{v_i}}\Crucial_f>0$
for at least one of $i\in\{1,2\}$, so that using \eqref{eq:slopeinside},
we have 
\begin{gather*}
 \nu_{f,\Gamma}(\{\cS\})=1-((\iota_{\Gamma_{\FR},\sP^1})_*\nu_{f,\Gamma})(U_{\overrightarrow{v_1}}\cup U_{\overrightarrow{v_2}})>1-(1/2+1/2)=0.
\end{gather*}
Hence $S_{f,\Gamma}\subset\Gamma_{S_{f,\Gamma}}\cap(\supp(\nu_{f,\Gamma})\cup\supp(\nu_{\Gamma_{\supp(\nu_{f,\Gamma})}}))$.
Moreover, for every $\cS\in\Gamma_{S_{f,\Gamma}}\setminus S_{f,\Gamma}$,
we have $\#T_{\cS}\Gamma_{S_{f,\Gamma}}=2$, so
writing $T_{\cS}\Gamma_{S_{f,\Gamma}}=\{\overrightarrow{v_1},\overrightarrow{v_2}\}$ and using \eqref{eq:slopeinside},
we have $((\iota_{\Gamma_{\FR},\sP^1})_*\nu_{f,\Gamma})(U_{\overrightarrow{v_i}})
=1/2$ for every $i\in\{1,2\}$ and
$((\iota_{\Gamma_{\FR},\sP^1})_*\nu_{f,\Gamma})(U_{\overrightarrow{w}})=0$
for every $\overrightarrow{w}\in T_{\sP^1}\Gamma\setminus\{\overrightarrow{v_1},\overrightarrow{v_2}\}$. Then we have not only
 $\nu_{f,\Gamma}(\{\cS\})=1-((\iota_{\Gamma_{\FR},\sP^1})_*\nu_{f,\Gamma})(U_{\overrightarrow{v_1}}\cup U_{\overrightarrow{v_2}})=1-(1/2+1/2)=0$
but also $\nu_{\Gamma_{\supp(\nu_{f,\Gamma})}}(\{\cS\})=0$.
Hence we also have $\Gamma_{S_{f,\Gamma}}\cap(\supp(\nu_{f,\Gamma})\cup\supp(\nu_{\Gamma_{\supp(\nu_{f,\Gamma})}}))\subset S_{f,\Gamma}$.
\end{proof}

\section{Proof of Theorem \ref{th:defining}}
\label{sec:explicit}
Let $f\in K(z)$ be of degree $d>1$. 
The following notation and convention will be useful.

\begin{notation}\label{th:divisor}
Let $\mathcal{K}$ be any field. For every $h\in\mathcal{K}(z)$ of degree $>0$ 
and every $\ell\in\mathcal{K}(z)$ other than $h$,
let $[h=\ell]$ be the effective $\mathcal{K}$-divisor on 
$\bP^1(\overline{\mathcal{K}})$ defined by the solutions to
the algebraic equation $h=\ell$ on $\bP^1(\overline{\mathcal{K}})$, 
taking into account their multiplicities. When $\mathcal{K}=K$, 
the divisor $[h=\ell]$ on $\bP^1=\bP^1(K)$ can also be regarded as the 
positive Radon measure 
\begin{gather*}
 \sum_{a\in\bP^1:h(a)=\ell(a)}(\ord_a[h=\ell])\cdot\delta_a\quad
\text{on }\sP^1.
\end{gather*}
\end{notation}

Let us begin by recalling the following {\itshape semiglobal} result by
Rivera-Letelier \cite[Lemma 2.1]{Juan03} and
Rumely's foundational results from \cite{Rumely14}.

\begin{theorem}[{\cite[Lemma 2.1]{Juan03}}]\label{th:semiglobal}
 Let $h\in K(z)$ be of degree $>0$.
 For every $\cS'\in\sP^1$ and every $\overrightarrow{v}\in T_{\cS'}\sP^1$, 
 there is $s_{\overrightarrow{v}}(h)\in\bN\cup\{0\}$,
 which we call the {\em surplus local degree}
of $h$ on $U_{\overrightarrow{v}}$,
 such that for every $\cS\in\sP^1\setminus\{h(\cS')\}$, 
 \begin{gather}
 (h^*\delta_{\cS})(U_{\overrightarrow{v}})
 =s_{\overrightarrow{v}}(h)
 +\begin{cases}
  m_{\overrightarrow{v}}(h) 
&\text{if }h_*(\overrightarrow{v})=\overrightarrow{h(\cS')\cS},\\
  0 &\text{otherwise}.
 \end{cases}\label{eq:preimages}
 \end{gather}
In particular, for every $\cS'\in\sP^1$,
we have
\begin{gather}
 \sum_{\overrightarrow{v}\in T_{\cS'}\sP^1}s_{\overrightarrow{v}}(h)=\deg h-\deg_{\cS}(h),\label{eq:surplustotal}
\end{gather}
by fixing $\cS\in\sP^1\setminus\{h(\cS')\}$,
summing \eqref{eq:preimages} up over all $\overrightarrow{v}\in T_{\cS'}\sP^1$,
and \eqref{eq:directionaldegree}. 
\end{theorem}

\begin{theorem}[{\cite[Lemma 2.1, the first identification lemma]{Rumely14}}]\label{th:1st3rd}
 Let $f\in K(z)$ be of degree $d>1$, and $\cS'\in\sP^1$ 
 be a type {\em II} fixed point of $f$.
 If $\widetilde{ h\circ f\circ h^{-1}}\neq\Id_{\bP^1(k)}$
 for some $($and indeed any$)$ $ h\in\PGL(2,K)$ 
 sending $\cS'$ to $\cS_{\can}$, then 
 for every $\overrightarrow{v}=\overrightarrow{v}_a^{( h)}
 \in T_{\cS'}\sP^1$ where $a\in\bP^1$ $($see \S$\ref{sec:tangent})$, 
 we have
 \begin{gather}
 [f=\Id_{\bP^1}](U_{\overrightarrow{v}})=s_{\overrightarrow{v}}(f)
 +\ord_{\tilde{a}}[\widetilde{ h\circ f\circ h^{-1}}=\Id_{\bP^1(k)}].
 \label{eq:first}
 \end{gather}
Summing \eqref{eq:first} 
up over all $\overrightarrow{v}\in T_{\cS'}\sP^1$, also 
using \eqref{eq:surplustotal},
we have
\begin{gather}
 \sum_{\tilde{a}\in\bP^1(k)}\ord_{\tilde{a}}[\widetilde{ h\circ f\circ h^{-1}}=\Id_{\bP^1(k)}]
=\deg_{\cS'}(f)+1\ge 2>0.\label{eq:fixedreduction}
\end{gather}
\end{theorem}

\begin{theorem}[{\cite[Lemma 2.2, the second identification lemma]{Rumely14}}]\label{th:2nd}
Let $f\in K(z)$ be of degree $d>1$, $\cS'\in\sP^1$
be a type {\em II} point not fixed by $f$, and
$\overrightarrow{v}\in T_{\cS'}\sP^1$.
Then $U_{\overrightarrow{v}}\cap\Fix(f)\neq\emptyset$ if and only if
at least one of the following is the case$;$ 
$($a$)$ $\overrightarrow{v}=\overrightarrow{\cS'f(\cS')}$,
$($b$)$ $f_*(\overrightarrow{v})=\overrightarrow{f(\cS')\cS'}$,
$($c$)$ $s_{\overrightarrow{v}}(f)>0$. 
\end{theorem}

\begin{theorem}[{\cite[Lemma 4.5, the third identification lemma]{Rumely14}}]\label{th:3rd}
 Let $f\in K(z)$ be of degree $d>1$, and $\cS'\in\sP^1$ 
 be a type {\em II} fixed point of $f$.
 If $\widetilde{ h\circ f\circ h^{-1}}=\Id_{\bP^1(k)}$
 for some 
 $ h\in\PGL(2,K)$ sending $\cS'$ to $\cS_{\can}$, then 
 we have $U_{\overrightarrow{v}}\cap
\Gamma_{\FR}
\neq\emptyset$
 for every such direction $\overrightarrow{v}\in T_{\cS'}\sP^1$
 that $s_{\overrightarrow{v}}(f)>0$.
\end{theorem}

Recall also that $\Gamma_{\FR}=\Gamma_{\FP}$ 
(Rumely's tree intersection theorem \eqref{eq:tree}) and 
$\Gamma_{\FR}\cap\bP^1=\Fix(f)$ (by the definition of $\Gamma_{\FR}$).

\subsection{Proof of \eqref{eq:computation} on $\Gamma_{\FR}$}
We note that every end point of $\Gamma_{\FR}=\Gamma_{\FP}$ 
is a type I or II point fixed by $f$.
The equality \eqref{eq:computation} on $\Gamma_{\FR}$ becomes
\begin{multline}
(d-1)\cdot\nu_{f,\Gamma_{\FR}}(\{\cS'\})=\\
\begin{cases}
\deg_{\cS'}(f)-1+\#\{\overrightarrow{v}\in T_{\cS'}\Gamma_{\FR}:f_*(\overrightarrow{v})\neq\overrightarrow{v}\} & \text{if }f(\cS')=\cS'\in\Gamma_{\FR}\cap\sH^1,\\
v_{\Gamma_{\FR}}(\cS')-2=(-2)\cdot\nu_{\Gamma_{\FR}}(\{\cS'\}) 
& \text{if }f(\cS')\neq\cS'\in\Gamma_{\FR}\cap\sH^1,\\
0 & \text{if }\cS'\in\Gamma_{\FR}\cap\bP^1=\Fix(f);
\end{cases}\tag{\ref{eq:computation}$'$}
\label{eq:computationtree}
\end{multline}
the final case holds
since $\supp(\nu_{f,\Gamma_{\FR}})\subset\sH^1$.

{\bfseries (A)} Let $\cS'\in\Gamma_{\FR}\cap\sH^1$ be 
fixed by $f$. Fix 
$\cS_0\in\sH^1\setminus(r_{\sP^1,\Gamma_{\FR}})^{-1}(\cS')$, and 
apply \eqref{eq:alwaysgeneral} to $\Gamma_{\FR}$.
{\bfseries (Subcase 1)} If $\cS'$ is of type II, then 
by the inclusion $\Fix(f)\subset\Gamma_{\FR}$ and Theorems \ref{th:1st3rd} and \ref{th:3rd}, 
for every $\overrightarrow{v}\in(T_{\cS'}\sP^1)\setminus(T_{\cS'}\Gamma_{\FR})$,
$s_{\overrightarrow{v}}(f)=0$. 
By \eqref{eq:preimages}, for every $\overrightarrow{v}\in T_{\cS'}\sP^1$,
\begin{gather*}
 (f^*\delta_{\cS_0})(U_{\overrightarrow{v}})=s_{\overrightarrow{v}}(f)+
 \begin{cases}
 m_{\overrightarrow{v}}(f) &
 \text{if } f_*(\overrightarrow{v})=\overrightarrow{\cS'\cS_0},\\
 0 &
 \text{otherwise}.
 \end{cases}
\end{gather*}
Hence, for every
$\overrightarrow{v}\in(T_{\cS'}\sP^1)\setminus(T_{\cS'}\Gamma_{\FR})$,
$(f^*\delta_{\cS_0})(U_{\overrightarrow{v}})>0$ if and only if
$f_*(\overrightarrow{v})=\overrightarrow{\cS'\cS_0}$, 
and then $(f^*\delta_{\cS_0})(U_{\overrightarrow{v}})=m_{\overrightarrow{v}}(f)$,
so that 
the second term in the right hand side of \eqref{eq:alwaysgeneral} 
applied to $\Gamma_{\FR}$ is 
\begin{gather*}
 \sum_{\overrightarrow{v}\in(T_{\cS'}\sP^1)\setminus(T_{\cS'}\Gamma_{\FR})
}(f^*\delta_{\cS_0})(U_{\overrightarrow{v}})
= \sum_{\overrightarrow{v}\in(T_{\cS'}\sP^1)\setminus(T_{\cS'}\Gamma_{\FR}):
 f^*(\overrightarrow{v})=\overrightarrow{\cS'\cS_0}}m_{\overrightarrow{v}}(f).
\end{gather*}
Hence also by \eqref{eq:fixedLaplaciangeneral} and
\eqref{eq:directionaldegree}, the equality
\eqref{eq:computationtree} holds in this case.
{\bfseries (Subcase 2)} If $\cS'$ is of type III,
then $\#T_{\cS'}\Gamma_{\FR}>1$, which with
$\#T_{\cS'}\sP^1=2$ yields $T_{\cS'}\sP^1=T_{\cS'}\Gamma_{\FR}$. 
In particular,
the second term in the right hand side of \eqref{eq:alwaysgeneral} 
applied to $\Gamma_{\FR}$ vanishes.
Writing 
$T_{\cS'}\Gamma_{\FR}=\{\overrightarrow{\cS'\cS_0},\overrightarrow{w}\}$,
by Theorem \ref{th:indifferent}, we have 
$f_*(\overrightarrow{\cS'\cS_0})=\overrightarrow{\cS'\cS_0}$,
$f_*(\overrightarrow{w})=\overrightarrow{w}$, and
$m_{\overrightarrow{\cS'\cS_0}}(f)=m_{\overrightarrow{w}}(f)=\deg_{\cS'}(f)=1$.
Then also by \eqref{eq:fixedcase},
the first term in the right hand side of \eqref{eq:alwaysgeneral} 
applied to $\Gamma_{\FR}$ can be computed as
\begin{multline*}
 \sum_{\overrightarrow{v}\in T_{\cS'}\Gamma_{\FR}}\rd_{\overrightarrow{v}}
 (\cS\mapsto\rho(\cS,f(\cS)\wedge_{\cS_0}\cS))
 =\bigl(m_{\overrightarrow{\cS'\cS_0}}(f)-1\bigr)+0\\
 =0=\deg_{\cS'}(f)-1+
 \#\{\overrightarrow{v}\in T_{\cS'}\Gamma_{\FR}:f_*(\overrightarrow{v})\neq\overrightarrow{v}\}.
\end{multline*}
Hence \eqref{eq:computationtree} holds in this case.

{\bfseries (B)} Let $\cS'\in\Gamma_{\FR}\cap\sH^1$ be not fixed by $f$. 
Then $\#T_{\cS'}\Gamma_{\FR}>1$ by the tree intersection theorem \eqref{eq:tree}.
We first claim that $r_{\sP^1,\Gamma_{\FR}}(f(\cS'))\neq\cS'$; for, 
suppose to the contrary that
$r_{\sP^1,\Gamma_{\FR}}(f(\cS'))=\cS'$. Then 
we have $\overrightarrow{\cS'f(\cS')}\not\in T_{\cS'}\Gamma_{\FR}$,
so that $\cS'$ must be of type II (otherwise, $\cS'$ is of type III
and then $T_{\cS'}\sP^1=T_{\cS'}\Gamma_{\FR}$, 
which contradicts $\overrightarrow{\cS'f(\cS')}\not\in T_{\cS'}\Gamma_{\FR}$).
Then by Theorem \ref{th:2nd},
we must have $\emptyset\neq U_{\overrightarrow{\cS'f(\cS')}}\cap\Fix(f)
\subset U_{\overrightarrow{\cS'f(\cS')}}\cap\Gamma_{\FR}$,
which is a contradiction. Hence this claim holds.
Fix now 
$\cS_0\in(\sH^1\setminus(r_{\sP^1,\Gamma_{\FR}})^{-1}(\cS'))\cap U_{\overrightarrow{f(\cS')\cS'}}$.
We also claim that
for every $\overrightarrow{v}\in(T_{\cS'}\sP^1)\setminus(T_{\cS'}\Gamma_{\FR})$,
$(f^*\delta_{\cS_0})(U_{\overrightarrow{v}})=0$; for,
if $\cS'$ is of type III, then $T_{\cS'}\sP^1=T_{\cS'}\Gamma_{\FR}$
and there is nothing to do. Suppose that $\cS'$ is type II and,
to the contrary, that $(f^*\delta_{\cS_0})(U_{\overrightarrow{v}})>0$
for some $\overrightarrow{v}\in(T_{\cS'}\sP^1)\setminus(T_{\cS'}\Gamma_{\FR})$.
Then by the inclusion $\Fix(f)\subset\Gamma_{\FR}$, 
Theorem \ref{th:2nd}, \eqref{eq:preimages},
and the choice of $\cS_0$,
we have ($s_{\overrightarrow{v}}(f)=0$ and in turn)
$f_*(\overrightarrow{v})=\overrightarrow{f(\cS')\cS_0}=\overrightarrow{f(\cS')\cS'}$,
so $\emptyset\neq U_{\overrightarrow{v}}\cap\Fix(f)\subset
U_{\overrightarrow{v}}\cap\Gamma_{\FR}$ by Theorem \ref{th:2nd}. 
This contradicts $U_{\overrightarrow{v}}\cap\Gamma_{\FR}=\emptyset$.
Hence this claim also holds.

Now by the above two claims and
\eqref{eq:alwaysgeneral} and \eqref{eq:nonfixedgeneral}
applied to $\Gamma_{\FR}$, we have \eqref{eq:computationtree} in this case.

\subsection{Proof of \eqref{eq:computation} on $\sP^1\setminus\Gamma_{\FR}$}
In \eqref{eq:computation},
the left hand side trivially vanishes on $\sP^1\setminus\Gamma_{\FR}$, and
so does the right hand side on $\bP^1$, trivially. 

Let us see that the right hand side also vanishes on $\sH^1\setminus\Gamma_{\FR}$.
The following argument is similar to that in 
\cite[Proof of Proposition 6.1]{Rumely14}.
Fix $\cS'\in\sH^1\setminus\Gamma_{\FR}$, and set
$\overrightarrow{v}:=\overrightarrow{\cS'(r_{\sP^1,\Gamma_{\FR}}(\cS'))}
\in T_{\cS'}\sP^1$, so that $\bigl\{\overrightarrow{w}\in T_{\cS'}\sP^1:
U_{\overrightarrow{w}}\cap\Gamma_{\FR}\neq\emptyset\bigr\}=\{\overrightarrow{v}\}$.
If $f(\cS')\neq\cS'$, then 
$\#\bigl\{\overrightarrow{w}\in T_{\cS'}\sP^1:
U_{\overrightarrow{w}}\cap\Gamma_{\FR}\neq\emptyset\bigr\}-2
=1-2=-1<0$,
so
the right hand side of \eqref{eq:computation} vanishes. 

Suppose now that $f(\cS')=\cS'$.
We claim both $\deg_{\cS'}(f)=1$ and $f_*(\overrightarrow{v})=\overrightarrow{v}$, which will imply that
the right hand side of \eqref{eq:computation} vanishes.
Recall the description
of $f_*=(f_*)_{\cS'}$ in Definition \ref{eq:directionaldegree}. 
(a) if $\cS'$ is of type II, 
then $\deg_{\cS'}(f)=1$ by the definition of $\Gamma_{\FR}$.
If in addition $\widetilde{ h\circ f\circ h^{-1}}\neq\Id_{\bP^1(k)}$
for some $ h\in\PGL(2,K)$ 
sending $\cS'$ to $\cS_{\can}$, then 
by the inclusion $\Fix(f)\subset\Gamma_{\FR}$ and Theorem \ref{th:1st3rd},
for every $\overrightarrow{w}\in T_{\cS'}\sP^1\setminus\{\overrightarrow{v}\}$, 
$f_*(\overrightarrow{w})\neq\overrightarrow{w}$ 
(and $s_{\overrightarrow{w}}(f)=0$). 
Hence by \eqref{eq:fixedreduction}, we have
$f_*(\overrightarrow{v})=\overrightarrow{v}$.
Alternatively, if in addition $\widetilde{ h\circ f\circ h^{-1}}=\Id_{\bP^1(k)}$
for some $ h\in\PGL(2,K)$ sending $\cS'$ to $\cS_{\can}$, then 
$f_*(\overrightarrow{w})=\overrightarrow{w}$
for {\itshape any} $\overrightarrow{w}\in T_{\cS'}\sP^1$,
so $f_*(\overrightarrow{v})=\overrightarrow{v}$. 
(b) If $\cS'$ is of type III or IV, then
by Theorem \ref{th:indifferent}, we still have $\deg_{\cS'}(f)=1$ and
$f_*(\overrightarrow{v})=\overrightarrow{v}$. 

\subsection{Proof of \eqref{eq:extendedtree}} 
Let $\Gamma$ be a non-trivial finite subtree in $\sP^1$,
and suppose that there are a point $\cS_{\Gamma}\in\Gamma_{\FR}$
and a direction $\overrightarrow{v_{\Gamma}}\in(T_{\cS_{\Gamma}}\sP^1)\setminus(T_{\cS_{\Gamma}}\Gamma_{\FR})$ such that 
$\Gamma\subset U_{\overrightarrow{v_{\Gamma}}}\cup\{\cS_{\Gamma}\}$.
Then $U_{\overrightarrow{v_\Gamma}}\cap\Fix(f)=\emptyset$, 
and the point 
\begin{gather*}
 \cS'_{\Gamma}:=r_{\sP^1,\Gamma}(\cS_{\Gamma})
\end{gather*}
is neither type I nor IV.

Hence for every type I or IV point $\cS'\in\Gamma$,
we note that $f(\cS')=\cS'$ only if $\cS'$ is type IV,
and that $\cS'\neq r_{\sP^1,\Gamma}(f(\cS'))$ if $f(\cS')\neq\cS'$,
so only the cases (A3) and (B1) in the right hand side in 
\eqref{eq:extendedtree} could occur. On the other hand,
by \eqref{eq:LaplacianI}, we also have 
\begin{gather*}
\left(\nu_{f,\Gamma}-(r_{\sP^1,\Gamma})_*\delta_{\cS_{\Gamma}}\right)(\{\cS'\})
=\nu_{f,\Gamma}(\{\cS'\}) 
=\bigl(\Delta_{\Gamma}(\Crucial_f|\Gamma)+\nu_{\Gamma}\bigr)(\{\cS'\})\\
=\begin{cases}
     0 &\text{if }f(\cS')=\cS'(\Leftrightarrow\text{(A3)}),\\
     -\frac{1}{d-1}
=\frac{((-2)\cdot\nu_{\Gamma}+(r_{\sP^1,\Gamma})_*\delta_{\cS_{\Gamma}})(\{\cS'\})}{d-1} &\text{if }f(\cS')\neq\cS'(\Leftrightarrow\text{(B1)}), 
    \end{cases}
\end{gather*}
which yields \eqref{eq:extendedtree} in this case.

Let us show \eqref{eq:extendedtree} for every type II or III point $\cS'\in\Gamma$.
Recall that type II points are dense in $(\sH^1,\rho)$ and
that the function $\Crucial_f$ on $(\sP^1,\tilde{\rho})$ 
is locally affine continuous at every type III point.

Te following computation \eqref{eq:truncation}
and \eqref{eq:enlarging} will be used in the proof of \eqref{eq:computation}
in the cases (A3, B1, B2); 
let $\cS'$ be a type III {\em end} point of $\Gamma$.
If $\cS'=\cS'_{\Gamma}$, then using \eqref{eq:outer},
for any type II point $\cS''\in\Gamma$ close enough to $\cS'$,
we have
\begin{align}
\label{eq:truncation} \bigl(\nu_{f,\Gamma}-(r_{\sP^1,\Gamma})_*\delta_{\cS_{\Gamma}}\bigr)(\{\cS'\})
=&\bigl(\nu_{f,[\cS'',\cS']}
-(r_{\sP^1,\Gamma})_*\delta_{\cS_{\Gamma}}\bigr)(\{\cS'\})\\
\notag=&-\bigl(\nu_{f,[\cS'',\cS']}-(r_{\sP^1,\Gamma})_*\delta_{\cS_{\Gamma}}\bigr)(\{\cS''\}).
\end{align}
If $\cS'\neq \cS'_{\Gamma}$, then for any type II point 
$\cS''\in r_{\sP^1,\Gamma}^{-1}(\{\cS'\})\cap\sH^1$,
we have $[\cS',\cS'']\subset U_{v_{\Gamma}}\cup\{\cS_{\Gamma}\}$ and 
$r_{\sP^1,[\cS',\cS'']}(\cS_{\Gamma})=\cS'$, and in turn using \eqref{eq:outer}, 
\begin{multline}
 \bigl(\nu_{f,\Gamma}-(r_{\sP^1,\Gamma})_*\delta_{\cS_{\Gamma}}\bigr)(\{\cS'\})
=\bigl(\nu_{f,[\cS',\cS'_{\Gamma}]}
-(r_{\sP^1,\Gamma})_*\delta_{\cS_{\Gamma}}\bigr)(\{\cS'\})\\
=\bigl(\nu_{f,[\cS'',\cS'_{\Gamma}]}
-(r_{\sP^1,\Gamma})_*\delta_{\cS_{\Gamma}}\bigr)(\{\cS'\})
-\bigl(\nu_{f,[\cS'',\cS']}-(r_{\sP^1,[\cS'',\cS']})_*\delta_{\cS'}\bigr)(\{\cS'\})\\
=\bigl(\nu_{f,[\cS'',\cS'_{\Gamma}]}
-(r_{\sP^1,\Gamma})_*\delta_{\cS_{\Gamma}}\bigr)(\{\cS'\})
+\bigl(\nu_{f,[\cS'',\cS']}
-(r_{\sP^1,[\cS'',\cS']})_*\delta_{\cS_{\Gamma}}\bigr)(\{\cS''\}).\label{eq:enlarging}
\end{multline}

{\bfseries (The cases A1, A2)} 
Let $\cS'\in\Gamma$ be a type II fixed point of $f$. 
Fix $\cS_0\in\sH^1\setminus(r_{\sP^1,\Gamma})^{-1}(\cS')$, 
and recall \eqref{eq:alwaysgeneral}.
Using Theorem \ref{th:semiglobal}, the inclusion $\Fix(f)\subset\Gamma_{\FR}$,
and Theorems \ref{th:1st3rd} and \ref{th:3rd}, we can compute
the second term in the right hand side of \eqref{eq:alwaysgeneral} 
applied to $\Gamma$ as
\begin{align*}
 &\sum_{\overrightarrow{v}\in(T_{\cS'}\sP^1)\setminus(T_{\cS'}\Gamma)}(f^*\delta_{\cS_0})(U_{\overrightarrow{v}})\\
=&\sum_{\overrightarrow{v}\in(T_{\cS'}\sP^1)\setminus(T_{\cS'}\Gamma):f_*(\overrightarrow{v})=\overrightarrow{\cS'\cS_0}}m_{\overrightarrow{v}}(f)
 +\sum_{\overrightarrow{v}\in(T_{\cS'}\sP^1)\setminus(T_{\cS'}\Gamma)}s_{\overrightarrow{v}}(f)\\
 =&\sum_{\overrightarrow{v}\in(T_{\cS'}\sP^1)\setminus(T_{\cS'}\Gamma):f_*(\overrightarrow{v})=\overrightarrow{\cS'\cS_0}}m_{\overrightarrow{v}}(f)\\
 &+
 \begin{cases}
\displaystyle\sum_{\overrightarrow{v}\in(T_{\cS'}\sP^1)\setminus(T_{\cS'}\Gamma)}0
=0 & \text{if }\cS'\in\Gamma\setminus\{\cS'_{\Gamma}\}\\
\displaystyle\sum_{\overrightarrow{v}\in T_{\cS'}\sP^1}s_{\overrightarrow{v}}(f)
=d-\deg_{\cS'}(f) & \text{if }\cS'=\cS'_{\Gamma}
 \end{cases}\\
=&\Biggl(\sum_{\overrightarrow{v}\in(T_{\cS'}\sP^1)\setminus(T_{\cS'}\Gamma):f_*(\overrightarrow{v})=\overrightarrow{\cS'\cS_0}}m_{\overrightarrow{v}}(f)\Biggr)
+(d-\deg_{\cS'}(f))\cdot\bigl((r_{\sP^1,\Gamma})_*\delta_{\cS_{\Gamma}}\bigr)(\{\cS'\}),
 \end{align*}
and by the definition of $\Gamma_{\FR}$, we have
$\deg_{\cS'}(f)=1$ if $\cS'\in\Gamma\setminus\{\cS'_{\Gamma}\}$. 
Hence also by \eqref{eq:fixedLaplaciangeneral}
and \eqref{eq:directionaldegree} applied to $h=f,\cS=\cS'$, and $\overrightarrow{w}=\overrightarrow{\cS'\cS_0}$, we have
 \begin{align*}
 &(d-1)\cdot\nu_{f,\Gamma}(\{\cS'\})\\
 =&\deg_{\cS'}(f)+\#\{\overrightarrow{v}\in T_{\cS'}\Gamma:
 f_*(\overrightarrow{v})\neq\overrightarrow{v}\}-1\\
&+(d-\deg_{\cS'}(f))\cdot\bigl((r_{\sP^1,\Gamma})_*\delta_{\cS_{\Gamma}}\bigr)(\{\cS'\})\\
 =&\#\{\overrightarrow{v}\in T_{\cS'}\Gamma:
 f_*(\overrightarrow{v})\neq\overrightarrow{v}\}
  +
  \begin{cases}
   1-1=0 &\text{ if }\cS'\in\Gamma\setminus\{\cS'_{\Gamma}\}\\
   d-1 &\text{ if }\cS'=\cS'_{\Gamma}
  \end{cases}\\ 
 =&\#\{\overrightarrow{v}\in T_{\cS'}\Gamma:
 f_*(\overrightarrow{v})\neq\overrightarrow{v}\}
 +(d-1)\cdot\bigl((r_{\sP^1,\Gamma})_*\delta_{\cS_{\Gamma}}\bigr)(\{\cS'\}).
 \end{align*}
Recall the description
of $f_*=(f_*)_{\cS'}$ in Definition \ref{eq:directionaldegree}. 
If in addition $\widetilde{ h\circ f\circ h^{-1}}=\Id_{\bP^1(k)}$
 for some $ h\in\PGL(2,K)$ sending $\cS'$ to $\cS_{\can}$,
 then 
 $\#\{\overrightarrow{v}\in T_{\cS'}\Gamma:
 f_*(\overrightarrow{v})\neq\overrightarrow{v}\}=\#\emptyset=0$.
 Alternatively, if 
 $\widetilde{ h\circ f\circ h^{-1}}\neq\Id_{\bP^1(k)}$
 for some $ h\in\PGL(2,K)$ sending $\cS'$ to $\cS_{\can}$,
 then by \eqref{eq:first},
 we have
\begin{align*}
 \#\{\overrightarrow{v}\in T_{\cS'}\Gamma:
 f_*(\overrightarrow{v})\neq\overrightarrow{v}\}
 =&
\begin{cases}
 \#\bigl((T_{\cS'}\Gamma)\setminus\{\overrightarrow{\cS'\cS_{\Gamma}}\}\bigr) &\text{if }\cS'\in\Gamma\setminus\{\cS'_{\Gamma}\}\\
 \#T_{\cS'}\Gamma &\text{if }\cS'=\cS'_{\Gamma}
 \end{cases}\\
 =&\bigl((-2)\cdot\nu_{\Gamma}+(r_{\sP^1,\Gamma})_*\delta_{\cS_{\Gamma}}\bigr)(\{\cS'\})+1.
\end{align*}
Hence \eqref{eq:extendedtree} holds in this case.

{\bfseries (The case A3)}
Next, let $\cS'\in\Gamma$ be a type III fixed point of $f$.
There are two cases. If $\#T_{\cS'}\Gamma>1$, then $\cS'\neq \cS'_{\Gamma}$,
$T_{\cS'}\sP^1=T_{\cS'}\Gamma$, and $\#T_{\cS'}\Gamma=2$, so that
$(d-1)\cdot((-2)\cdot\nu_{\Gamma}+(r_{\sP^1,\Gamma})_*\delta_{\cS_{\Gamma}})(\{\cS'\})=0$.
On the other hand, by Theorem \ref{th:indifferent},
for every $\overrightarrow{v}\in T_{\cS'}\sP^1$,
we have $f_*(\overrightarrow{v})=\overrightarrow{v}$ 
and $m_{\overrightarrow{v}}(f)=\deg_{\cS'}(f)=1$, so
fixing $\cS_0\in\sH^1\setminus(r_{\sP^1,\Gamma})^{-1}(\cS')$,
also by \eqref{eq:alwaysgeneral}, \eqref{eq:fixedLaplaciangeneral}, 
and $T_{\cS'}\sP^1=T_{\cS'}\Gamma$, we have
\begin{gather*}
  (d-1)\cdot\nu_{f,\Gamma}(\{\cS'\})
 =\Bigl(m_{\overrightarrow{\cS'\cS_0}}(f)
 +\#\{\overrightarrow{v}\in T_{\cS'}\Gamma:
 f_*(\overrightarrow{v})\neq\overrightarrow{v}\}-1\Bigr)+0=0. 
\end{gather*}
Hence \eqref{eq:extendedtree} holds in this case. 

Suppose now that $\#T_{\cS'}\Gamma=1$.
If $\cS'=\cS'_{\Gamma}$, then by Theorem \ref{th:indifferent}, 
for every type II point $\cS''\in\Gamma$ close enough to $\cS'$, 
we have $\cS''=f(\cS'')$, so that
by \eqref{eq:truncation} and the already seen \eqref{eq:extendedtree}
applied to $[\cS'',\cS']$ and $\cS''$, we have
\begin{multline*}
\bigl(\nu_{f,\Gamma}-(r_{\sP^1,\Gamma})_*\delta_{\cS_{\Gamma}}\bigr)(\{\cS'\})
=-\bigl(\nu_{f,[\cS'',\cS']}-(r_{\sP^1,\Gamma})_*\delta_{\cS_{\Gamma}}\bigr)(\{\cS''\})\\
=-\frac{1}{d-1}\begin{cases}
   0 & \text{if }\cS''\text{ is in the case (A1)}\\
   \bigl((v_{[\cS'',\cS']}(\cS'')-2)+0\bigr)+1=0
& \text{if }\cS''\text{ is in the case (A2)}
  \end{cases}=0.
\end{multline*}
Alternatively, if $\cS'\neq \cS'_{\Gamma}$,
then for every type II point $\cS''\in r_{\sP^1,\Gamma}^{-1}(\cS')$
close enough to $\cS'$, we have not only $\cS''=f(\cS'')$
by Theorem \ref{th:indifferent} but also $\#T_{\cS'}[\cS'',\cS'_{\Gamma}]=2$,
so that by \eqref{eq:enlarging}, 
the already seen \eqref{eq:extendedtree} applied to $[\cS'',\cS'_{\Gamma}]$ and $\cS'$, and 
the already seen \eqref{eq:extendedtree} applied to $[\cS'',\cS']$
and $\cS''$ (no matter whether $\cS''$ is in (A1) or (A2)), we similarly
have
\begin{align*}
&\bigl(\nu_{f,\Gamma}-(r_{\sP^1,\Gamma})_*\delta_{\cS_{\Gamma}}\bigr)(\{\cS'\})\\
=&\bigl(\nu_{f,[\cS'',\cS'_{\Gamma}]}
-(r_{\sP^1,\Gamma})_*\delta_{\cS_{\Gamma}}\bigr)(\{\cS'\})
+\bigl(\nu_{f,[\cS'',\cS']}-(r_{\sP^1,[\cS'',\cS']})_*\delta_{\cS_{\Gamma}}\bigr)(\{\cS''\})\\
=&0+0=0.
\end{align*}
Hence \eqref{eq:extendedtree} holds in this case. 

{\bfseries (The cases B1, B2)} 
Let $\cS'\in\Gamma$ be a type II or III point not fixed by $f$.
We prepare the following.

\begin{lemma}\label{th:typeIIIend}
$(i)$ If $\cS'$ is type {\em II} and 
$(f(\cS')\neq)\cS'=r_{\sP^1,\Gamma}(f(\cS'))$,
then $\cS'_{\Gamma}=\cS'$. 

$(ii)$ If $T_{\cS'}\Gamma=1$ and
$(f(\cS')\neq)\cS'=\cS'_{\Gamma}$, then
$\cS'=r_{\sP^1,\Gamma}(f(\cS'))$.
\end{lemma}

\begin{proof}
(i) Suppose to the contrary that $\cS'_{\Gamma}\neq\cS'$.
Then $\Gamma_{\FR}\subset U_{\overrightarrow{\cS'\cS_{\Gamma}}}$.
On the other hand,
by Theorem \ref{th:2nd} and the inclusion $\Fix(f)\subset\Gamma_{\FR}$,
we must also have $\emptyset\neq U_{\overrightarrow{\cS'f(\cS')}}\cap\Fix(f)
\subset U_{\overrightarrow{\cS'f(\cS')}}\cap\Gamma_{\FR}$. Hence 
$\overrightarrow{\cS'f(\cS')}=\overrightarrow{\cS'\cS_{\Gamma}}\in T_{\cS'}\Gamma$,
which contradicts $r_{\sP^1,\Gamma}(f(\cS'))=\cS'$.

 (ii) Suppose to the contrary that
 $\cS'\neq r_{\sP^1,\Gamma}(f(\cS'))$.
 Then $U_{\overrightarrow{\cS'f(\cS')}}\subset U_{\overrightarrow{v_{\Gamma}}}$. 
 For every type II point $\cS''\in U_{\overrightarrow{\cS'f(\cS')}}$ 
 close enough to $\cS'$, we also have $f(\cS'')\neq\cS''$ and 
 $U_{\overrightarrow{\cS''f(\cS'')}}\subset U_{\overrightarrow{\cS'f(\cS')}}$, 
 and then by Theorem \ref{th:2nd} and the inclusion
 $\Fix(f)\subset\Gamma_{\FR}$, we must have
 $\emptyset\neq U_{\overrightarrow{\cS''f(\cS'')}}\cap\Fix(f)\subset U_{\overrightarrow{v_{\Gamma}}}\cap\Gamma_{\FR}$.
 This contradicts $U_{\overrightarrow{v_{\Gamma}}}\cap\Gamma_{\FR}=\emptyset$. 
\end{proof}

{\bfseries (Subcase a)} 
Suppose that $\cS'$ is {\em not} a type III {\em end} point of $\Gamma$.
Fix $\cS_0\in(\sH^1\setminus(r_{\sP^1,\Gamma})^{-1}(\cS'))\cap U_{\overrightarrow{f(\cS')\cS'}}$, and apply \eqref{eq:alwaysgeneral} to $\Gamma$. 
If $r_{\sP^1,\Gamma}(f(\cS'))=\cS'$,
then $\cS'$ is type II under the above assumption, so that 
by Lemma \ref{th:typeIIIend}(i), we have $\cS'_{\Gamma}=\cS'$.
Hence by \eqref{eq:nonfixedgeneral}, 
the first term in the right hand side of \eqref{eq:alwaysgeneral} 
applied to $\Gamma$ is computed as
\begin{align*}  
 \sum_{\overrightarrow{v}\in T_{\cS'}\Gamma}\rd_{\overrightarrow{v}}
 \bigl(\cS\mapsto\rho(\cS,f(\cS)\wedge_{\cS_0}\cS)\bigr)
 =&v_{\Gamma}(\cS')-1 =(v_{\Gamma}(\cS')-2)+1\\
 =&(-2)\cdot\nu_{\Gamma}(\{\cS'\})
 +\bigl((r_{\sP^1,\Gamma})_*\delta_{\cS_{\Gamma}}\bigr)(\{\cS'\}).
\end{align*}
Alternatively, if $r_{\sP^1,\Gamma}(f(\cS'))\neq\cS'$, then 
by \eqref{eq:nonfixedgeneral}, we similarly have
\begin{gather*}
 \sum_{\overrightarrow{v}\in T_{\cS'}\Gamma}\rd_{\overrightarrow{v}}
 \bigl(\cS\mapsto\rho(\cS,f(\cS)\wedge_{\cS_0}\cS)\bigr)
=v_{\Gamma}(\cS')-2 =(-2)\cdot\nu_{\Gamma}(\{\cS'\}).
\end{gather*}
On the other hand, 
(no matter whether $r_{\sP^1,\Gamma}(f(\cS'))=\cS'$ or not,)
the second term in the right hand side of \eqref{eq:alwaysgeneral} 
applied to $\Gamma$ is computed as follows;
if $\cS'$ is type II, then 
by Theorems \ref{th:semiglobal} and \ref{th:2nd},
the inclusion $\Fix(f)\subset\Gamma_{\FR}$, and \eqref{eq:directionaldegree}, 
we have
\begin{align*}
&\sum_{\overrightarrow{v}\in(T_{\cS'}\sP^1)\setminus(T_{\cS'}\Gamma)}(f^*\delta_{\cS_0})(U_{\overrightarrow{v}})\\
=&
\sum_{\overrightarrow{v}\in(T_{\cS'}\sP^1)\setminus(T_{\cS'}\Gamma)}s_{\overrightarrow{v}}(f)
+\sum_{\overrightarrow{v}\in(T_{\cS'}\sP^1)\setminus(T_{\cS'}\Gamma):\atop
f_*(\overrightarrow{v})=\overrightarrow{f(\cS')\cS_0}(=\overrightarrow{f(\cS')\cS'}\text{ by the choice of }\cS_0)}m_{\overrightarrow{v}}(f)\\
=&
\begin{cases}
\displaystyle
\sum_{\overrightarrow{v}\in(T_{\cS'}\sP^1)\setminus(T_{\cS'}\Gamma)}0=0 & \text{if }\cS'\in\Gamma\setminus\{\cS'_{\Gamma}\}\\
\displaystyle\sum_{\overrightarrow{v}\in T_{\cS'}\sP^1}s_{\overrightarrow{v}}(f)
=d-\deg_{\cS'}(f) & \text{if }\cS'=\cS'_{\Gamma}
 \end{cases}\\
&+
\begin{cases}
\displaystyle\sum_{\emptyset}m_{\overrightarrow{v}}(f)=0 & \text{if }\cS'\in\Gamma\setminus\{\cS'_{\Gamma}\}\\
\displaystyle\sum_{\overrightarrow{v}\in T_{\cS'}\sP^1:f_*(\overrightarrow{v})=\overrightarrow{f(\cS')\cS_0}}m_{\overrightarrow{v}}(f)=\deg_{\cS'}(f)
& \text{if }\cS'=\cS'_{\Gamma}
\end{cases}\\
=&
\begin{cases}
0 & \text{if }\cS'\in\Gamma\setminus\{\cS'_{\Gamma}\}\\
d & \text{if }\cS'=\cS'_{\Gamma}
 \end{cases}\\
=&\bigl((r_{\sP^1,\Gamma})_*\delta_{\cS_{\Gamma}}\bigr)(\{\cS'\})
+(d-1)\cdot\bigl((r_{\sP^1,\Gamma})_*\delta_{\cS_{\Gamma}}\bigr)(\{\cS'\}).
\end{align*}
Alternatively, if $\cS'$ is type III, then $T_{\cS'}\Gamma=T_{\cS'}\sP^1$ 
under the above assumption, so that
we have $((r_{\sP^1,\Gamma})_*\delta_{\cS_{\Gamma}})(\{\cS'\})=0$ and
the second term in the right hand side of \eqref{eq:alwaysgeneral}
applied to $\Gamma$ vanishes.

Hence \eqref{eq:extendedtree} holds in this subcase.

{\bfseries (Subcase b)} Suppose finally that $\cS'$ is a type III end point of $\Gamma$.
If $(f(\cS')\neq)\cS'=\cS'_{\Gamma}$,
then $r_{\sP^1,\Gamma}(f(\cS'))=\cS'$ by Lemma \ref{th:typeIIIend}(ii), so that
for every type II point $\cS''\in\Gamma$ close enough to $\cS'$, 
we also have $f(\cS'')\neq\cS''$.
Hence by \eqref{eq:truncation} and the already seen \eqref{eq:extendedtree}
applied to $[\cS'',\cS']$ and $\cS''$, we have
\begin{align*}
&(d-1)\cdot\bigl(\nu_{f,\Gamma}-(r_{\sP^1,\Gamma})_*\delta_{\cS_{\Gamma}}\bigr)(\{\cS'\})\\
=&-(d-1)\cdot\bigl(\nu_{f,[\cS'',\cS']}-(r_{\sP^1,\Gamma})_*\delta_{\cS_{\Gamma}}\bigr)(\{\cS''\})\\
=&-\bigl((-2)\cdot\nu_{[\cS'',\cS']}+(1\text{ or }2)\cdot(r_{\sP^1,\Gamma})_*\delta_{\cS_{\Gamma}}\bigr)(\{\cS''\})\\
=&-\bigl((-2)\cdot\nu_{[\cS'',\cS']}\bigr)(\{\cS''\})
=1=\bigl((-2)\cdot\nu_{\Gamma}+2(r_{\sP^1,\Gamma})_*\delta_{\cS_{\Gamma}}\bigr)(\{\cS'\}).
\end{align*}
Alternatively, if $(f(\cS')\neq)\cS'\neq \cS'_{\Gamma}$,
then $((r_{\sP^1,\Gamma})_*\delta_{\cS_{\Gamma}})(\{\cS'\})=0$.
For every type II point $\cS''\in r_{\sP^1,\Gamma}^{-1}(\cS')$ close enough 
to $\cS'$, we have not only $f(\cS'')\neq\cS''$ but also
$\#T_{\cS'}[\cS'',\cS'_{\Gamma}]=2$.
Hence by \eqref{eq:enlarging}, 
the already seen \eqref{eq:extendedtree} applied to $[\cS'',\cS'_{\Gamma}]$ and $\cS'$, and 
the already seen \eqref{eq:extendedtree} applied to $[\cS'',\cS']$ and $\cS''$,
we have
\begin{align*}
&(d-1)\cdot\bigl(\nu_{f,\Gamma}-(r_{\sP^1,\Gamma})_*\delta_{\cS_{\Gamma}}\bigr)(\{\cS'\})\\
=&(d-1)\cdot\bigl(\nu_{f,[\cS'',\cS'_{\Gamma}]}-(r_{\sP^1,[\cS'',\cS'_{\Gamma}]})_*\delta_{\cS_{\Gamma}}\bigr)(\{\cS'\})\\
&+(d-1)\cdot\bigl(\nu_{f,[\cS'',\cS']}-(r_{\sP^1,[\cS'',\cS']})_*\delta_{\cS_{\Gamma}}\bigr)(\{\cS''\})\\
=&((-2)\cdot\nu_{[\cS'',\cS'_{\Gamma}]}(\{\cS'\})+0)
+((-2)\cdot\nu_{[\cS'',\cS']}(\{\cS''\})+0)\\
=&0+(-1)=-1=(-2)\cdot\nu_{\Gamma}(\{\cS'\})\\
=&
\begin{cases}
 ((-2)\cdot\nu_{\Gamma}+(r_{\sP^1,\Gamma})_*\delta_{\cS_{\Gamma}}
)(\{\cS'\}) 
&\text{if }(f(\cS')\neq)\cS'\neq r_{\sP^1,\Gamma}(f(\cS')),\\
((-2)\cdot\nu_{\Gamma}+2(r_{\sP^1,\Gamma})_*\delta_{\cS_{\Gamma}})(\{\cS'\})
&\text{if }(f(\cS')\neq)\cS'=r_{\sP^1,\Gamma}(f(\cS')).
\end{cases}
\end{align*}
Hence \eqref{eq:extendedtree} also holds in this subcase. 

\subsection{Proof of (iii)}
Fix $\cS'\in\sP^1\setminus\Gamma_{\FR}$.
By \eqref{eq:slopeI} 
and the locally affine continuity of $\Crucial_f$ on $(\sP^1,\tilde{\rho})$
at every type I or IV point, we have \eqref{eq:slopeoutside}
if $\cS'$ is type I or type IV. 
So, suppose that $\cS'$ is either type II or III.
Then we have
$\overrightarrow{(r_{\sP^1,\Gamma_{\FR}}(\cS'))\cS'}\in 
(T_{r_{\sP^1,\Gamma_{\FR}}(\cS')}\sP^1)\setminus(T_{r_{\sP^1,\Gamma_{\FR}}(\cS')}\Gamma_{\FR})$ and
$[\cS',r_{\sP^1,\Gamma_{\FR}}(\cS')]\subset 
U_{\overrightarrow{(r_{\sP^1,\Gamma_{\FR}}(\cS'))\cS'}}\cup\{r_{\sP^1,\Gamma_{\FR}}(\cS')\}$.
Hence 
from \eqref{eq:outer} and \eqref{eq:extendedtree} applied to 
$[\cS',r_{\sP^1,\Gamma_{\FR}}(\cS')]$,
we have
\begin{align*}
 &\rd_{\overrightarrow{\cS'(r_{\sP^1,\Gamma_{\FR}}(\cS'))}}\Crucial_f
=\bigl(\Delta_{[\cS',r_{\sP^1,\Gamma_{\FR}}(\cS')]}(\Crucial_f|[\cS',r_{\sP^1,\Gamma_{\FR}}(\cS')])\bigr)(\{\cS'\})\\
=&-\nu_{[\cS',r_{\sP^1,\Gamma_{\FR}}(\cS')]}(\{\cS'\})
+\left(\nu_{f,[\cS',r_{\sP^1,\Gamma_{\FR}}(\cS')]}-(r_{\sP^1,[\cS',r_{\sP^1,\Gamma_{\FR}}(\cS')]})_*\delta_{r_{\sP^1,\Gamma_{\FR}}(\cS')}\right)(\{\cS'\})\\
=&-\frac{1}{2}
+\frac{1}{d-1}\cdot
 \begin{cases}
 0 & \text{if }f(\cS')=\cS',\\
 -1& \text{if }f(\cS')\neq\cS,
\end{cases}
\end{align*}
which with the piecewise affine continuity of $\Crucial_f$ on $(\sH^1,\rho)$
also yields \eqref{eq:slopeoutside} in this case. 

Let us see the remaining assertion.
Suppose in addition that $f(\cS')=\cS'$. Then
since $1/(d-1)\neq 0$, by \eqref{eq:slopeoutside}
and the piecewise affine continuity of $f$ 
on $(\sH^1,\rho)$, 
the component in $([\cS',r_{\sP^1,\Gamma_{\FR}}(\cS')],\rho)$ of 
the set of all fixed points of $f$ 
in $[\cS',r_{\sP^1,\Gamma_{\FR}}(\cS')]$
containing $\cS'$ is a closed interval $[\cS',\cS_*]$ for some 
$\cS_*\in(\cS',r_{\sP^1,\Gamma_{\FR}}(\cS')]$.
In fact $\cS_*=r_{\sP^1,\Gamma_{\FR}}(\cS')$; for, otherwise,
the component in $([\cS_*,r_{\sP^1,\Gamma_{\FR}}(\cS')],\rho)$
of the set of all fixed points of $f$ in 
$[\cS_*,r_{\sP^1,\Gamma_{\FR}}(\cS')]$ containing $\cS_*$
is a closed interval $[\cS_*,\cS_{**}]$ for some
$\cS_{**}\in(\cS_*,r_{\sP^1,\Gamma_{\FR}}(\cS')]$.
Then we must have $[\cS',\cS_*]\supset[\cS',\cS_*]\cup[\cS_*,\cS_{**}]=
[\cS',\cS_{**}]$, which is a contradiction.
Finally, for any type II point $\cS''\in(\cS',r_{\sP^1,\Gamma_{\FR}}(\cS')]$,
we have $f_*(\overrightarrow{\cS''\cS'})=\overrightarrow{\cS''\cS'}$. 
Now the proof
of the final assertion in (iii) is complete by Theorem \ref{th:1st3rd} and 
$U_{\overrightarrow{\cS''\cS'}}\cap\Fix(f)\subset U_{\overrightarrow{\cS''\cS'}}\cap\Gamma_{\FR}=\emptyset$. \qed

\section{Proof of Theorem \ref{th:diam}}
\label{sec:diam}

Let $f\in K(z)$ be of degree $d>1$. By the definition of $\Crucial_f$, we have
\begin{multline}
\frac{\rho(\cS,\cS_{\can})}{2}=(\Crucial_f(\cS)-\Crucial_f(\cS_{\can}))\\
+\frac{-\rho(\cS,f(\cS)\wedge_{\can}\cS)
+\int_{\sP^1}\rho(\cS_{\can},\cS\wedge_{\can}\cdot)\rd(f^*\delta_{\cS_{\can}})}{d-1}\quad\text{on }\sH^1.\label{eq:original}
\end{multline}
Fix a minimal lift $F$ of $f$.
By a standard argument from the elimination theory
(see, e.g., Kawaguchi--Silverman \cite[Proposition 2]{KS09}), we have
\begin{gather}
\frac{\|F(\cdot)\|}{\|\cdot\|^d}\ge|\Res F|(\le 1)\quad\text{on }K^2\setminus\{0\},\label{eq:homogeneouslower}
\end{gather}
so that by the definition of the continuous function $T_F$ on $\sP^1$
introduced after \eqref{eq:cohomological} (and the density of $\bP^1$ in $\sP^1$) and the equality \eqref{eq:potentialminimal}, we have
\begin{gather}
\int_{\sP^1}\rho(\cS_{\can},\cS\wedge_{\can}\cdot)\rd(f^*\delta_{\cS_{\can}})
=-T_F(\cS)\le-\log|\Res F|\quad\text{on }\sH^1.\label{eq:boundresultant}
\end{gather}

\begin{proof}[Proof of \eqref{eq:diamminresloc}]
Noting that
$\Crucial_f(\cS)-\Crucial_f(\cS_{\can})\le 0$ for every
$\cS\in\Crucial_f^{-1}(\min_{\sH^1}\Crucial_f)$, by \eqref{eq:original},
we have
\begin{gather*}
 \frac{\rho(\cS,\cS_{\can})}{2}
 \le 0+\frac{-\rho(\cS,f(\cS)\wedge_{\can}\cS)+
\int_{\sP^1}\rho(\cS_{\can},\cS\wedge_{\can}\cdot)\rd(f^*\delta_{\cS_{\can}})}{d-1},
\end{gather*}
which with \eqref{eq:boundresultant} yields \eqref{eq:diamminresloc}.
\end{proof}

\begin{proof}[Proof of \eqref{eq:diamcrucial}]
For every $\cS\in\Gamma_{\supp(\nu_{f,\Gamma_{\FR}})}\setminus\{\cS_{\can}\}$,
letting $s\mapsto\cS(s)$ be the arc-length parametrization of 
$([\cS_{\can},\cS],\rho)$ from $\cS_{\can}$ to $\cS$ by
the closed interval $[0,\rho(\cS_{\can},\cS)]$ in $\bR$, 
by the piecewise affine continuity of $\Crucial_f$
on $(\sH^1,\rho)$, we have
\begin{multline}
\Crucial_f(\cS)-\Crucial_f(\cS_{\can})
=\int_0^{\rho(\cS_{\can},\cS)}\Bigl(\frac{\rd}{\rd s}\Crucial_f(\cS(s))\Bigr)\rd s\\
 \le\rho(\cS_{\can},\cS)\cdot\sup_{\cS'\in[\cS_{\can},\cS)}
 \rd_{\overrightarrow{\cS'\cS}}\Crucial_f.\label{eq:integration}
\end{multline}

We claim that for every $\cS\in\Gamma_{\supp(\nu_{f,\Gamma_{\FR}})}\setminus\{\cS_{\can}\}$,
 \begin{gather}
\sup_{\cS'\in[\cS_{\can},\cS)}\rd_{\overrightarrow{\cS'\cS}}\Crucial_f
 \le
\begin{cases}
 0 &\text{if }\#\supp(\nu_{f,\Gamma_{\FR}})=1\\
 \frac{1}{2}-\frac{1}{d-1} & \text{if }\#\supp(\nu_{f,\Gamma_{\FR}})>1
\end{cases}\ge 0,
\label{eq:slope}
\end{gather}
which will be seen below and where
we recalled $\#\supp(\nu_{f,\Gamma_{\FR}})=1$ when $d=2$.
Once \eqref{eq:slope} is at our disposal, 
for every $\cS\in\Gamma_{\supp(\nu_{f,\Gamma_{\FR}})}$,
also using \eqref{eq:original} and \eqref{eq:integration}, 
we will have
\begin{multline*}
 \frac{\rho(\cS,\cS_{\can})}{2}
 \le
\Bigl(\frac{1}{2}-\frac{1}{d-1}\Bigr)\cdot\rho(\cS,\cS_{\can})
\\
 +\frac{-\rho(\cS,f(\cS)\wedge_{\can}\cS)+\int_{\sP^1}\rho(\cS_{\can},\cS\wedge_{\can}\cdot)\rd(f^*\delta_{\cS_{\can}})}{d-1}\quad\text{if }d>2,
\end{multline*}
which with \eqref{eq:boundresultant} yields \eqref{eq:diamcrucial}.
\end{proof}

\begin{proof}[Proof of \eqref{eq:slope}]
 Fix $\cS\in\Gamma_{\supp(\nu_{f,\Gamma_{\FR}})}\setminus\{\cS_{\can}\}$.
 Then by \eqref{eq:slopeoutside}, 
 for every $\cS'\in\bigl[\cS_{\can},r_{\sP^1,\Gamma_{\supp(\nu_{f,\Gamma_{\FR}})}}(\cS_{\can})\bigr)$, 
 \begin{gather*}
\rd_{\overrightarrow{\cS'\cS}}\Crucial_f\le-\frac{1}{2}\le 0,
 \end{gather*}
which yields \eqref{eq:slope} in the case of $\#\supp(\nu_{f,\Gamma_{\FR}})=1$,
so in particular when $d=2$.
Suppose now that $\#\supp(\nu_{f,\Gamma_{\FR}})>1$ (so $d>2$), and
fix an end point $\cS''$ of $\Gamma_{\supp(\nu_{f,\Gamma_{\FR}})}$ 
such that $\cS\in[r_{\sP^1,\Gamma_{\supp(\nu_{f,\Gamma_{\FR}})}}(\cS_{\can}),\cS'']$. 
Then $\nu_{f,\Gamma_{\FR}}(\{\cS''\})>0$, and in turn
by \eqref{eq:weight}, we have
\begin{gather}
 \frac{1}{d-1}\le\nu_{f,\Gamma_{\FR}}(\{\cS''\})=1-\nu_{f,\Gamma_{\FR}}(\Gamma_{\FR}\setminus\{\cS''\}).\label{eq:crucialtreeend}
\end{gather}

For every $\cS'\in[r_{\sP^1,\Gamma_{\supp(\nu_{f,\Gamma_{\FR}})}}(\cS_{\can}),\cS)$,
 by the convexity (and the piecewise affine continuity) of $\Crucial_f$
 on $([r_{\sP^1,\Gamma_{\supp(\nu_{f,\Gamma_{\FR}})}}(\cS_{\can}),\cS''],\rho)$, we have
 \begin{gather*}
\rd_{\overrightarrow{\cS'\cS}}\Crucial_f
\le-\rd_{\overrightarrow{\cS''(r_{\sP^1,\Gamma_{\supp(\nu_{f,\Gamma_{\FR}})}}(\cS_{\can}))}}\Crucial_f,
 \end{gather*}
and moreover, by \eqref{eq:slopeinside} and \eqref{eq:crucialtreeend}, we also have
 \begin{align*}
 &\rd_{\overrightarrow{\cS''(r_{\sP^1,\Gamma_{\supp(\nu_{f,\Gamma_{\FR}})}}(\cS_{\can}))}}\Crucial_f\\
 =&\frac{1}{2}-\nu_{f,\Gamma_{\FR}}\Bigl(\Gamma_{\FR}\cap U_{\overrightarrow{\cS''(r_{\sP^1,\Gamma_{\supp(\nu_{f,\Gamma_{\FR}})}}(\cS_{\can}))}}\Bigr)\\
=&\frac{1}{2}-\nu_{f,\Gamma_{\FR}}(\Gamma_{\FR}\setminus\{\cS''\})
\ge-\frac{1}{2}+\frac{1}{d-1}.
 \end{align*}
Hence \eqref{eq:slope} also holds in the case of $\#\supp(\nu_{f,\Gamma_{\FR}})>1$.
\end{proof}

\begin{remark}\label{th:geomiterate}
Fix $n\in\bN$. The $n$-th iteration $F^n$ of a minimal lift $F$ of $f$ is 
not necessarily a minimal lift of $f^n$ 
(see, e.g.,\
\cite[Proposition 3]{JacobsWilliams16}).
Under the assumption (*) that {\em $F^n$ is still a minimal lift of $f^n$},
setting $d>1$, by an argument similar to that used to
obtain \eqref{eq:boundresultant}, we have
\begin{multline*}
 -\int_{\sP^1}\rho(\cS_{\can},\cS\wedge_{\can}\cdot)\rd((f^n)^*\delta_{\cS_{\can}})\\
=T_{F^n}(\cS)=\sum_{j=0}^{n-1}d^{n-1-j}\cdot T_F(f^j(\cS))
\ge\frac{d^n-1}{d-1}\log|\Res F| 
\end{multline*}
on $\sH^1$,
which with the definition of $\Crucial_{f^n}$ yields
\begin{multline*}
 \Crucial_{f^n}(\cS)-\Crucial_{f^n}(\cS_{\can})\\
\ge\frac{\rho(\cS,\cS_{\can})}{2}
+\frac{\rho(\cS,f^n(\cS)\wedge_{\can}\cS)}{d^n-1}+\frac{\log|\Res F|}{d-1}\quad\text{on }\sH^1.
\end{multline*}
This improves Jacobs \cite[Proposition 8, where
the condition (*) was implicitly assumed]{Jacobs17}.
\end{remark}

\section{Proof of Theorem \ref{th:quantitative}}
\label{sec:quantitative}

Let $f\in K(z)$ be of degree $d>1$. 

\begin{lemma}\label{th:outsidebound}
Let $\Gamma$ be a non-trivial finite subtree $\Gamma$ in $\sP^1$,
and suppose that there are
a point $\cS_{\Gamma}\in\Gamma_{\FR}$
and a direction
$\overrightarrow{v_{\Gamma}}\in(T_{\cS_{\Gamma}}\sP^1)\setminus(T_{\cS_{\Gamma}}\Gamma_{\FR})$ such that
$\Gamma\subset U_{\overrightarrow{v_{\Gamma}}}\cup\{\cS_{\Gamma}\}$. Then
setting $\cS'_{\Gamma}:=r_{\sP^1,\Gamma}(\cS_{\Gamma})$, we have
\begin{gather}
\bigl|\nu_{f,\Gamma}-(r_{\sP^1,\Gamma})_*\delta_{\cS_{\Gamma}}\bigr|(\Gamma)
\le\frac{2\cdot\#(\{\text{end points of }\Gamma\}
\setminus\{\cS'_{\Gamma}\})}{d-1}.
\label{eq:errorbranch}
\end{gather}
\end{lemma}

The following general fact is independent of $f$;
for every non-trivial finite subtree $\Gamma$ in $\sP^1$,
every (possibly trivial) subtree $\Gamma'$ in $\Gamma$, 
and every component $C$ of $\Gamma\setminus\Gamma'$ in $(\Gamma,\rho)$,
the closure $\overline{C}$ of $C$ in $(\Gamma,\rho)$
is a non-trivial finite subtree in $\Gamma$ and $\overline{C}\setminus C=\{\cS_C\}$ 
for some end point $\cS_C$ of $\overline{C}$, 
and we compute
\begin{align}
\notag&\bigl|(-2)\cdot\nu_{\Gamma}\bigr|(C)\\
\notag=&(-2)\cdot\nu_{\Gamma}(\overline{C}\setminus\{\text{end points of }\overline{C}\})-(-2)\cdot\nu_{\Gamma}(\{\text{end points of }\overline{C}\}\setminus\{\cS_C\})\\
\notag=&(-2)\cdot\nu_{\overline{C}}(\overline{C}\setminus\{\text{end points of }\overline{C}\})+\#(\{\text{end points of }\overline{C}\}\setminus\{\cS_C\})\\
\notag=&(-2)\cdot\nu_{\overline{C}}(\overline{C})
-(-2)\cdot\nu_{\overline{C}}(\{\text{end points of }\overline{C}\})
+\#(\{\text{end points of }\overline{C}\}\setminus\{\cS_C\})\\
\notag=&(-2)\cdot\nu_{\overline{C}}(\overline{C})
+2\cdot\#(\{\text{end points of }\overline{C}\}\setminus\{\cS_C\})+1\\
=&2\cdot\#(\{\text{end points of }\overline{C}\}\setminus\{\cS_C\})-1.\label{eq:errobranchcomponent}
\end{align}

\begin{proof}[Proof of Lemma $\ref{th:outsidebound}$]
Fix a non-trivial finite subtree $\Gamma$ in $\sP^1$ as in the Lemma, and set 
$\cS'_{\Gamma}:=r_{\sP^1,\Gamma}(\cS_{\Gamma})$.
 
{\bfseries (a)}
Suppose that $f$ has a fixed point in $\Gamma$, and
let $\Gamma'$ be the set of all fixed points of $f$ in $\Gamma$.
Then using Theorem \ref{th:defining}(iii),
$\Gamma'$ is a (possibly trivial) subtree in $\Gamma$ 
and $\cS'_{\Gamma}\in\Gamma'$.
Let $F_{\Gamma}$ be the set of all points in $\Gamma$ satisfying
the condition (A2) in \eqref{eq:extendedtree}. Then $F_{\Gamma}\subset\Gamma'$.

If $\Gamma'\neq\{\cS'_{\Gamma}\}$, then
from Theorem \ref{th:defining}(iii), we have
$F_{\Gamma}\subset\{\text{end points of }\Gamma'\}\setminus\{\cS'_{\Gamma}\}$.
Also using Theorem \ref{th:defining}(ii), we have 
\begin{gather*}
(d-1)\cdot\bigl(\nu_{f,\Gamma}-(r_{\sP^1,\Gamma})_*\delta_{\cS_{\Gamma}}\bigr)
\begin{cases}
 =(-2)\cdot\nu_{\Gamma} &\text{on }\Gamma\setminus\Gamma',\\
 =0 & \text{on }\Gamma'\setminus F_{\Gamma},\\
=(-2)\cdot\nu_{\Gamma}+1 & \text{on }F_{\Gamma}.
\end{cases}
\end{gather*}
Hence, by \eqref{eq:errobranchcomponent} and
\begin{gather*}
\sum_{\cS\in\{\text{end points of }\Gamma'\}\setminus\{\cS'_{\Gamma}\}}(v_{\Gamma}(\cS)-1)
\le \#\{\text{components of }\Gamma\setminus\Gamma'\},
\end{gather*}
we have
\begin{align*}
 &(d-1)\cdot\bigl|\nu_{f,\Gamma}-(r_{\sP^1,\Gamma})_*\delta_{\cS_{\Gamma}}\bigr|(\Gamma)\\
=&\bigl|(-2)\cdot\nu_{\Gamma}\bigr|(\Gamma\setminus\Gamma')
+\bigl|(d-1)\cdot
\bigl(\nu_{f,\Gamma}-(r_{\sP^1,\Gamma})_*\delta_{\cS_{\Gamma}}\bigr)\bigr|(\Gamma')\\
\le& 
\bigl(2\cdot\#(\{\text{end points of }\Gamma\}\setminus\Gamma')
-\#\{\text{components of }(\Gamma\setminus\Gamma')\}\bigr)\\
&+\sum_{\cS\in F_{\Gamma}}(v_{\Gamma}(\cS)-1)\\ 
\le&2\cdot\#(\{\text{end points of }\Gamma\}\setminus\Gamma')
\le 2\cdot\#(\{\text{end points of }\Gamma\}
\setminus\{\cS'_{\Gamma}\}),
\end{align*}
so \eqref{eq:errorbranch} holds in this case.

Alternatively, if $\Gamma'=\{\cS'_{\Gamma}\}$, then
using Theorem \ref{th:defining}(ii), we have 
\begin{gather*}
(d-1)\cdot\bigl(\nu_{f,\Gamma}-(r_{\sP^1,\Gamma})_*\delta_{\cS_{\Gamma}}\bigr)
=
\begin{cases}
(-2)\cdot\nu_{\Gamma} &\text{on }\Gamma\setminus\Gamma'=\Gamma\setminus\{\cS'_{\Gamma}\},\\
0 & \text{on }\Gamma'=\{\cS'_{\Gamma}\}\text{ and }F_{\Gamma}=\emptyset,\\
(-2)\cdot\nu_{\Gamma}+2 & \text{on }\Gamma'=\{\cS'_{\Gamma}\}=F_{\Gamma}.
\end{cases}
\end{gather*}
Hence by \eqref{eq:errobranchcomponent} and
\begin{gather*}
 v_{\Gamma}(\cS_{\Gamma}')
 =\#T_{\cS_{\Gamma}'}\Gamma=\#\bigl\{\text{components of }\Gamma\setminus\{\cS_{\Gamma}'\}\bigr\},
\end{gather*}
we similarly have
\begin{align*}
 &(d-1)\cdot\bigl|
\nu_{f,\Gamma}-(r_{\sP^1,\Gamma})_*\delta_{\cS_{\Gamma}}\bigr|(\Gamma)\\
\le&\bigl|(-2)\cdot\nu_{\Gamma}\bigr|(\Gamma\setminus\{\cS_{\Gamma}'\})
+\bigl|(-2)\cdot\nu_{\Gamma}(\{\cS_{\Gamma}'\})+2\bigr|\\
\le& 
\bigl(2\cdot\#(\{\text{end points of }\Gamma\}\setminus\{\cS_{\Gamma}'\})-\#\{\text{components of }(\Gamma\setminus\{\cS_{\Gamma}'\})\}\bigr)\\
&+v_{\Gamma}(\cS_{\Gamma}')\\
=&2\cdot\#(\{\text{end points of }\Gamma\}\setminus\{\cS_{\Gamma}'\}),
\end{align*}
so \eqref{eq:errorbranch} also holds in this case.

{\bfseries (b)}
If $f$ has no fixed points in $\Gamma$, then
set $\Gamma':=\Gamma_{\{\cS_{\Gamma}'\}}=\{\cS_{\Gamma}'\}$.
Using Theorem \ref{th:defining}(ii), 
an argument similar to that in the latter half of the case (a) yields
\begin{align*}
 &(d-1)\cdot\bigl|
\nu_{f,\Gamma}-(r_{\sP^1,\Gamma})_*\delta_{\cS_{\Gamma}}\bigr|(\Gamma)\\
\le&\bigl|(-2)\cdot\nu_{\Gamma}\bigr|(\Gamma\setminus\{\cS_{\Gamma}'\})
+\bigl|(-2)\cdot\nu_{\Gamma}(\{\cS_{\Gamma}'\})+(1\text{ or }2)\bigr|\\
\le& 
\bigl(2\cdot\#(\{\text{end points of }\Gamma\}\setminus\{\cS_{\Gamma}'\})-\#\{\text{components of }(\Gamma\setminus\{\cS_{\Gamma}'\})\}\bigr)\\
&+\bigl(v_{\Gamma}(\cS_{\Gamma}')+(-1\text{ or }0)\bigr)\\
\le&2\cdot\#(\{\text{end points of }\Gamma\}\setminus\{\cS_{\Gamma}'\}).
\end{align*}
This completes the proof of Lemma \ref{th:outsidebound}.
\end{proof}

We recall a construction of the $f$-equilibrium
(or canonical) measure
$\mu_f$ on $\sP^1$
(\cite[\S10]{BR10}, \cite[\S2]{ChambertLoir06}, \cite[\S3.1]{FR09}).

\begin{lemma}\label{th:construction}
 For every $\cS_0\in\sH^1$ and every $n\in\bN$,
\begin{gather}
 \mu_f-\frac{(f^n)^*\delta_{\cS_0}}{d^n}
 =\Delta\Biggl(\sum_{j=n}^{\infty}\frac{\int_{\sP^1}\left(-\rho(\cS_0,f^j(\cdot)\wedge_{\cS_0}\cS')\right)\rd(f^*\delta_{\cS_0})(\cS')}{d^{j+1}}\Biggr)
\quad\text{on }\sP^1.\label{eq:telescope}
\end{gather}
\end{lemma}
 
\begin{proof}
 For every $\cS_0\in\sH^1$ and every $n\in\bN$,
 by \eqref{eq:fundamental} and \eqref{eq:functorial}, we have
\begin{gather*}
 \frac{(f^n)^*\delta_{\cS_0}}{d^n}-\delta_{\cS_0}
=\Delta\Biggl(\sum_{j=0}^{n-1}
\frac{\int_{\sP^1}
\bigl(-\rho(\cS_0,f^j(\cdot)\wedge_{\cS_0}\cS')\bigr) 
\rd(f^*\delta_{\cS_0})(\cS')
}{d^{j+1}}\Biggr)\quad\text{on }\sP^1,
\end{gather*} 
and 
$\cS\mapsto\int_{\sP^1}(-\rho(\cS_0,\cS\wedge_{\cS_0}\cdot))\rd(f^*\delta_{\cS_0})$ is continuous (and bounded) on $\sP^1$.
Hence the weak limit $\lim_{n\to\infty}(f^n)^*\delta_{\cS_0}/d^n$
on $\sP^1$ exists and satisfies the definition (or characterization) of $\mu_f$
(stated in Subsection \ref{sec:dynamics}), and we have \eqref{eq:telescope}. 
\end{proof}

The following improves Jacobs \cite[Theorem 4]{Jacobs17}.

\begin{lemma}
For every $n\in\bN$, every $\cS\in\sH^1$, and every $\cS_0\in\sH^1$,
\begin{multline}
\biggl|\Crucial_{f^n}(\cS)-\Crucial_{f^n}(\cS_0)-
\biggl(\frac{\rho(\cS,\cS_0)}{2}
-\int_{\sP^1}\rho(\cS_0,\cS\wedge_{\cS_0}\cdot)\rd\mu_f\biggr)\biggr|\\
\le\frac{2\bigl(C_{\cS_0,f}/(d-1)+\rho(\cS,\cS_0)\bigr)}{d^n-1},\label{eq:potentialconv}
\end{multline} 
where $C_{\cS_0,f}:=\sup_{\cS\in\sP^1}\int_{\sP^1}\rho(\cS_0,\cS\wedge_{\cS_0}\cdot)\rd(f^*\delta_{\cS_0})<\infty$ as in Theorem $\ref{th:quantitative}$. 
\end{lemma}

\begin{proof}
For every $n\in\bN$, every $\cS\in\sH^1$, and every $\cS_0\in\sH^1$,
by \eqref{eq:conjugate}, Green's formula, and \eqref{eq:fundamental}, we have 
\begin{align*}
&\biggl|\Crucial_{f^n}(\cS)-\Crucial_{f^n}(\cS_0)-
\biggl(\frac{\rho(\cS,\cS_0)}{2}
-\int_{\sP^1}\rho(\cS_0,\cS\wedge_{\cS_0}\cdot)\rd\mu_f\biggr)\biggr|\\
\le&\frac{d^n}{d^n-1}\biggl|\int_{\sP^1}\rho(\cS_0,\cS\wedge_{\cS_0}\cdot)
\rd\Bigl(\mu_f-\frac{(f^n)^*\delta_{\cS_0}}{d^n}\Bigr)\biggr|\\
&+\frac{|\rho(\cS,f(\cS)\wedge_{\cS_0}\cS)-\int_{\sP^1}\rho(\cS_0,\cS\wedge_{\cS_0}\cdot)\rd\mu_f|}{d^n-1}\\
\le&\frac{d^n}{d^n-1}\Biggl|\int_{\sP^1}\sum_{j=n}^{\infty}
\frac{\int_{\sP^1}\rho(\cS_0,f^j(\cdot)\wedge_{\cS_0}\cS')
\rd(f^*\delta_{\cS_0})(\cS')}{d^{j+1}}
\Delta\bigl(\rho(\cS_0,\cS\wedge_{\cS_0}\cdot)\bigr)\Biggr|\\
&+\frac{2\cdot\rho(\cS,\cS_0)}{d^n-1}
\le\frac{2\bigl(C_{\cS_0,f}/(d-1)+\rho(\cS,\cS_0)\bigr)}{d^n-1},
\end{align*} 
which completes the proof.
\end{proof}

Fix a continuous test function $\phi$ on $\sP^1$ such that
$\phi|\Gamma$ is continuous on $(\Gamma,\rho)$ for some finite tree $\Gamma$ in
$\sH^1$ and that $\phi=(r_{\sP^1,\Gamma})^*\phi$ on $\sP^1$.
Since \eqref{eq:quantitative} is clear when $\Gamma$ is trivial,
we also assume that $\Gamma$ is non-trivial.

Fix now $n\in\bN$. 
Set
\begin{gather*}
 \mathcal{C}_n=\mathcal{C}_{n,\Gamma}:= 
\bigl\{\text{components }C\text{ of }
 (\Gamma\setminus\Gamma_{f^n,\FR})\text{ in }(\Gamma,\rho)\bigr\},
\end{gather*}
so that for each $C\in \mathcal{C}_n$,
the closure $\overline{C}$ of $C$ in $(\Gamma,\rho)$
is a non-trivial finite subtree in $\Gamma$,
$r_{\Gamma,\Gamma_{f^n,\FR}}(C)=\{\cS_{n,C}\}$ 
for some point $\cS_{n,C}\in\Gamma_{f^n,\FR}$, and 
$\overline{C}\subset U_{\overrightarrow{v_{n,C}}}\cup\{\cS_{n,C}\}$ 
for some direction 
$\overrightarrow{v_{n,C}}\in(T_{\cS_{n,C}}\sP^1)
\setminus(T_{\cS_{n,C}}\Gamma_{f^n,\FR})$. 
If in addition $\Gamma\cap\Gamma_{f^n,\FR}=\emptyset$, then 
$\mathcal{C}_n=\{\Gamma\}$ and $\overline{\Gamma}=\Gamma$, 
and setting $\cS_{n,\Gamma}':=r_{\sP^1,\Gamma}(\cS_{n,\Gamma})$, we have
$r_{\sP^1,\Gamma}(\Gamma_{f^n,\FR})=\{\cS_{n,\Gamma}'\}$;
also set $I_n=I_{n,\Gamma}:=\Gamma_{\{\cS_{n,\Gamma}',\,\cS_{n,\Gamma}\}}=[\cS_{n,\Gamma}',\cS_{n,\Gamma}]$(,
which joins $\Gamma_{f^n,\FR}$ with $\Gamma$).

Set 
$\Gamma_n
:=\Gamma_{\{\text{end points of either }\Gamma_{f^n,\FR}\text{ or }\Gamma\}}$,
so that both $\Gamma$ and $\Gamma_{f^n,\FR}$ are subtrees in $\Gamma_n$.
Recalling
the definitions \eqref{eq:valency} and \eqref{eq:crucialgeneral}, 
we have
\begin{align*}
 \nu_{f^n,\Gamma_n}
 =&(\iota_{\Gamma_{f^n,\FR},\Gamma_n})_*\nu_{f^n,\Gamma_{f^n,\FR}}\\
&+\begin{cases}
 \sum_{C\in \mathcal{C}_n}(\iota_{\overline{C},\Gamma_n})_*\bigl(\nu_{f^n,\overline{C}}-(r_{\sP^1,\overline{C}})_*\delta_{\cS_{n,C}}\bigr) &\text{if }\Gamma\cap\Gamma_{f^n,\FR}\neq\emptyset\\
 (\iota_{\Gamma,\Gamma_n})_*\nu_{f^n,\Gamma} &\text{if }\Gamma\cap\Gamma_{f^n,\FR}=\emptyset 
 \end{cases}\\ 
&+\begin{cases}
  0 &\text{if }\Gamma\cap\Gamma_{f^n,\FR}\neq\emptyset\\
  \bigl(\iota_{I_n,\Gamma_n}\bigr)_*\bigl(\nu_{f^n,I_n}
 -(r_{\sP^1,I_n})_*(\delta_{\cS_{n,\Gamma}}+\delta_{\cS'_{n,\Gamma}})\bigr)
  &\text{if }\Gamma\cap\Gamma_{f^n,\FR}=\emptyset
 \end{cases}\\
=&(\iota_{\Gamma_{f^n,\FR},\Gamma_n})_*\nu_{f^n,\Gamma_{f^n,\FR}}
+\sum_{C\in \mathcal{C}_n}(\iota_{\overline{C},\Gamma_n})_*\bigl(\nu_{f^n,\overline{C}}-(r_{\sP^1,\overline{C}})_*\delta_{\cS_{n,C}}\bigr)\\
&+\begin{cases}
  0 &\text{if }\Gamma\cap\Gamma_{f^n,\FR}\neq\emptyset\\
  \bigl(\iota_{I_n,\Gamma_n}\bigr)_*\bigl(\nu_{f^n,I_n}
 -(r_{\sP^1,I_n})_*\delta_{\cS'_{n,\Gamma}}\bigr)
  &\text{if }\Gamma\cap\Gamma_{f^n,\FR}=\emptyset
 \end{cases}\quad\text{on }\Gamma_n;
\end{align*}
even if $\Gamma\cap\Gamma_{f^n,\FR}=\emptyset$, we have
\begin{multline*}
 (r_{\Gamma_n,\Gamma})_*
 \bigl((\iota_{I_n, \Gamma_n})_*
 (\nu_{f^n,I_n}-(r_{\sP^1,I_n})_*\delta_{\cS'_{n,\Gamma}})\bigr)\\
 =\bigl((\nu_{f^n,I_n}-(r_{\sP^1,I_n})_*\delta_{\cS'_{n,\Gamma}})(I_n)\bigr)\cdot
 (r_{\sP^1,\Gamma})_*\delta_{\cS_{n,\Gamma}}
 =0\quad\text{on }\Gamma.
\end{multline*}
Hence also by the formula $\eqref{eq:formulameasure}$ for $\nu_{f^n}$
(and $\overline{C}\subset\Gamma$), we have
\begin{align*}
&(r_{\Gamma_n,\Gamma})_*\nu_{f^n,\Gamma_n}\\ 
=&
(r_{\Gamma_n,\Gamma})_*(\iota_{\Gamma_{f^n,\FR},\Gamma_n})_*\nu_{f^n,\Gamma_{f^n,\FR}}
+\sum_{C\in \mathcal{C}_n}(r_{\Gamma_n,\Gamma})_*(\iota_{\overline{C},\Gamma_n})_*\bigl(\nu_{f^n,\overline{C}}-(r_{\sP^1,\overline{C}})_*\delta_{\cS_{n,C}}\bigr)\\
=&
(r_{\sP^1,\Gamma})_*\nu_{f^n}
+\sum_{C\in \mathcal{C}_n}(\iota_{\overline{C},\Gamma})_*\bigl(\nu_{f^n,\overline{C}}-(r_{\sP^1,\overline{C}})_*\delta_{\cS_{n,C}}\bigr).
\end{align*}

Now fix also $\cS_0\in\sH^1$. Then 
using the above computation of $(r_{\Gamma_n,\Gamma})_*\nu_{f^n,\Gamma_n}$,
we have
\begin{align*}
 &\Delta_{\Gamma}
 (\Bigl(\Crucial_{f^n}-\Crucial_{f^n}(\cS_0)-
 \frac{\rho(\cdot,\cS_0)}{2}
 +\int_{\sP^1}\rho(\cS_0,\cdot\wedge_{\cS_0}\cS')\rd\mu_f(\cS')\Bigr)|\Gamma)\\
=&\Delta_{\Gamma}
(\Bigl(\frac{(\cS\mapsto\rho(\cS,f^n(\cS)\wedge_{\cS_0}\cS))}{d^n-1}\\
&+\int_{\sP^1}(-\rho(\cS_0,\cdot\wedge_{\cS_0}\cS'))
 \rd\Bigl(\frac{(f^n)^*\delta_{\cS_0}}{d^n-1}(\cS')-\mu_f\Bigr)(\cS')\Bigr)|\Gamma)\quad(\text{by }\eqref{eq:conjugate}\text{ for }f^n)
\\
=&(r_{\Gamma_n,\Gamma})_*
\Bigl(\nu_{f^n,\Gamma_n}-\frac{(r_{\sP^1,\Gamma_n})_*((f^n)^*\delta_{\cS_0}-\delta_{\cS_0})}{d^n-1}\Bigr)
\quad(\text{by }\eqref{eq:laplacian}\text{ and } 
\eqref{eq:geometric}\text{ for }\nu_{f^n,\Gamma_n})\\
&+(r_{\sP^1,\Gamma})_*\biggl(\frac{(f^n)^*\delta_{\cS_0}}{d^n-1}-\mu_f
-\Bigl(\frac{d^n}{d^n-1}-1\Bigr)\delta_{\cS_0}\biggr)
\quad(\text{and by }\eqref{eq:fundamental}\\
&\hspace*{0.65\textwidth}\text{ and Fubini's theorem})\\
=&(r_{\Gamma_n,\Gamma})_*\nu_{f^n,\Gamma_n}-(r_{\sP^1,\Gamma})_*\mu_f\\
 =&(r_{\sP^1,\Gamma})_*(\nu_{f^n}-\mu_f)
+\sum_{C\in \mathcal{C}_n}
 (\iota_{\overline{C},\Gamma})_*\bigl(\nu_{f^n,\overline{C}}-
 (r_{\sP^1,\overline{C}})_*\delta_{\cS_{n,C}}\bigr).
\end{align*}
Once this computation of the (signed) Radon measure
$(r_{\sP^1,\Gamma})_*(\nu_{f^n}-\mu_f)$ on $\Gamma$
is at our disposal,
setting $\Gamma_n':=r_{\sP^1,\Gamma}(\Gamma_{f^n,\FR})$, we have
\begin{align*}
&\left|\int_{\sP^1}\phi\rd(\nu_{f^n}-\mu_f)\right|
=\left|\int_{\Gamma}\phi\rd\bigl((r_{\sP^1,\Gamma})_*(\nu_{f^n}-\mu_f)\bigr)\right|
\quad (\text{since }\phi=(r_{\sP^1,\Gamma})^*\phi)\\
\le&\frac{2\bigl(C_{\cS_0,f}/(d-1)+\sup_{\Gamma}\rho(\cdot,\cS_0)\bigr)}{d^n-1}
\cdot|\Delta_{\Gamma}(\phi|\Gamma)|(\Gamma)\quad(\text{by Green's formula and \eqref{eq:potentialconv}})\\
&+\biggl(\sum_{C\in \mathcal{C}_n}\bigl|\nu_{f^n,\overline{C}}-
 (r_{\sP^1,\overline{C}})_*\delta_{\cS_{n,C}}\bigr|(\overline{C})\biggr)\cdot\sup_{\Gamma}|\phi|\\
\le&\frac{2\bigl(C_{\cS_0,f}/(d-1)+\sup_{\Gamma}\rho(\cdot,\cS_0)\bigr)}{d^n-1}
\cdot|\Delta_{\Gamma}(\phi|\Gamma)|(\Gamma)\\
&+\frac{2\cdot\#(\{\text{end points of }\Gamma\}\setminus \Gamma'_n)}{d^n-1}\cdot\sup_{\Gamma}|\phi|\quad(\text{by }\eqref{eq:errorbranch}\text{ for }f^n\text{ and each }\overline{C}),
\end{align*}
which together with
$|\Delta_{\Gamma}(\phi|\Gamma)|(\Gamma)=|\Delta\phi|(\sP^1)$
and $\sup_{\Gamma}|\phi|=\sup_{\sP^1}|\phi|$
yields \eqref{eq:quantitative}.
\qed

\begin{acknowledgement}
The author thanks the referee for a very careful scrutiny and invaluable comments.
This research was partially supported by JSPS Grant-in-Aid 
for Scientific Research (C), 15K04924 and 19K03541.
\end{acknowledgement}

\def\cprime{$'$}

\end{document}